\documentclass[10pt]{article}
\usepackage{enumitem}
\usepackage{amsfonts}
\usepackage{amsmath}
\usepackage{graphicx}
\usepackage{epstopdf}
\usepackage{fancyhdr}
\usepackage{amsthm}
\usepackage{geometry}
\usepackage{subfigure}
\usepackage[figuresright]{rotating}
\usepackage{caption}
\usepackage{bm}
\usepackage{amssymb}
\usepackage{amscd}
\usepackage{float}
\usepackage{color}
\usepackage{tabularx}
\usepackage{diagbox}
\usepackage{tabularx}
\usepackage{mathtools}
\usepackage{multicol}
\usepackage{multirow}
\usepackage{makecell}
\usepackage[justification=centering]{caption}
\usepackage{rotfloat}
\usepackage{longtable}
\allowdisplaybreaks[4]

\newtheorem{thm}{Theorem}[section]
\newtheorem{cor}{Corollary}[section]
\newtheorem{lemma}{Lemma}[section]

\newtheorem{remark}{Remark}[section]

\newtheorem{defn}{Definition}[section]

\newcounter{nextauthor}
\def\mathrm{\mbox}

\numberwithin{remark}{section}

\begin{document}
\title{{\bf Robust equilibrium strategy for mean-variance-skewness portfolio selection problem}\thanks{This work was supported by the National Natural Science Foundation of China (11671282, 12171339).}}
\author{Jian-hao Kang$^a$, Nan-jing Huang$^a$\thanks{Corresponding author. E-mail addresses: nanjinghuang@hotmail.com; njhuang@scu.edu.cn}, Zhihao Hu$^{b,c}$ and Ben-Zhang Yang$^b$ \\
{\small\it  a. Department of Mathematics, Sichuan University, Chengdu, Sichuan 610064, P.R. China}\\
{\small\it  b. National Institution for Finance $\&$ Development, Chinese Academy of Social Science, Beijing, P.R. China}\\
{\small\it  c.  Institute of Finance $\&$ Banking, Chinese Academy of Social Science, Beijing, P.R. China}}
\date{}
\maketitle \vspace*{-9mm}
\begin{abstract}
\noindent
This paper considers a robust time-consistent mean-variance-skewness portfolio selection problem for an ambiguity-averse investor by taking into account wealth-dependent risk aversion and wealth-dependent skewness preference as well as model uncertainty. The robust equilibrium investment strategy and corresponding equilibrium value function are characterized for such a problem by employing an extended Hamilton-Jacobi-Bellman-Isaacs (HJBI) system via a game theoretic approach. Furthermore, the robust equilibrium investment strategy and corresponding equilibrium value function are obtained in semi-closed form for a special robust time-consistent mean-variance-skewness portfolio selection problem. Finally, some numerical experiments are provided to indicate several new findings concerned with the robust equilibrium investment strategy and the utility losses.
\end{abstract}

\section{Introduction}
The mean-variance criterion, proposed by Markowitz \cite{Markowitz1952}, has long been deemed as the foundation of modern portfolio selection theory. From then on, the one-period model \cite{Markowitz1952} has been widely extended to general models, for instance, the discrete multi-period model \cite{Li2000}, the continuous time model \cite{Zhou2000}, the incomplete market model \cite{Lim2004}, the nonnegative terminal state constrained model \cite{Bielecki2005} and more on. Different methods are used to solve these models, for example, Zhou and Li \cite{Zhou2000} apply the "embedding" approach to change the original problem into a standard linear quadratic problem, Lim \cite{Lim2004} adopts the stochastic maximum principle to obtain the forward backward stochastic differential equation and Bielecki et al. \cite{Bielecki2005} deal with the constraints by the Lagrange multiplier method. The common point of these papers is that their models are solved at the initial time and the corresponding optimal controls are the so-called pre-committed optimal strategies that are optimal only at the fixed initial time, namely, the pre-committed strategies are not time-consistent.

In some real situations, it is wise for investors to take into account the dynamic changes of the latest information to constantly make the optimal decisions, which inspires researchers and practitioners to seek the time-consistent optimal strategies. Strotz \cite{Strotz1955} firstly uses game theory to deal with time inconsistency. Thereafter, many authors have followed the game theoretic approach to study various time-inconsistent problems under different conditions. For instance, Basak and Chabakauri \cite{Basak2010} utilize a backward sequential game to introduce the equilibrium feedback strategy for a mean-variance portfolio problem with constant risk aversion. Bj{\"{o}}rk and Murgoci \cite{Bjork2009} investigate a class of time-inconsistent control problem in a general Markovian model under a game theoretic framework. They derive a system of extended Hamilton-Jacobi-Bellman (HJB) equation and prove the associated verification theorem to determine the equilibrium strategy as well as the equilibrium value function. Hu et al. \cite{Hu2012} define a time-consistent open-loop equilibrium solution to discuss a general time-inconsistent stochastic linear-quadratic control problem. Other important discussions and results of time-inconsistent problems and portfolio selection problems can be also found in the recent literature \cite{Bjork2014,  Ekeland2006, Ekeland2008}.

Along with great process in investigating mean-variance problems through a time-consistent approach, especially ones in \cite{Bjork2009, Hu2012}, the characterization of the investors' risk attitude has also attracted considerable research attention. To capture the dynamics of the investors' risk aversion, the constant risk aversion coefficients in mean-variance models could be replaced by the state dependent risk aversion.  For instance, Bj\"{o}rk et al. \cite{Bjork2014} point out that it is unrealistic from an economic point of view to assume the risk aversion parameter in \cite{Bjork2009} is a constant. They consider a new mean-variance portfolio selection with state dependent risk aversion and show that the equilibrium dollar amount invested in the risky asset is proportional to current wealth when the risk aversion function is inversely proportional to wealth. Following the work \cite{Bjork2014}, Li and Li \cite{LiY2013Optimal} further investigate the optimal time-consistent investment and reinsurance problem under mean-variance criterion with state dependent risk aversion. On the other hand, in the context of finance and economy, there has been rapid development in extending mean-variance models by adding higher order risk preferences. For example, Eeckhoudt and Schlesinger \cite{Eeckhoudt2006} depict the 3rd-order and 4th-order risk preferences outside the expected utility framework, and analyze the effects of 4th-order risk preference on saving in \cite{Eeckhoudt2008}. Lai \cite{Lai2012} uses real price data to study how hedging behaviours may vary when hedgers take into account skewness and excess kurtosis of hedging returns in the decision models. Recently, Ebert and Hilpert \cite{Ebert2019} show that investors' preference for positive skewness satisfies the popularity of technical analysis. It is worth mentioning that the introduction of higher order risk preferences into the portfolio selection problems in the continuous-time financial market results in nonlinear structure such that it is difficult to find (semi-)closed form for the optimal investment strategies. Skewness, as the moment after mean and variance, has attracted the attention of researchers and practitioners because of induced asymmetries in ex post returns. Very recently, Mu et al. \cite{Mu2019} study a portfolio selection problem with both wealth-dependent risk aversion and wealth-dependent skewness preference and obtain the equilibrium investment strategy in semi-closed form. More related studies can also be found in \cite{Brocas2016, Dai2021, Gu2020, Kimball1990} and the references therein.

All of the aforementioned works focus only on the portfolio problems for the ambiguity-neutral investors. However, Ellsberg \cite{Ellsberg1961} discloses the inadequacy of utility theory and shows that investors are not only risk-averse but also ambiguity-averse. Since there exist many factors in the financial market which are beyond the knowledge of the investor, but may have significant effects on the investor's decisions, which means that ambiguity is inevitable in practice. Thus, it is important and necessary to consider the impact of ambiguity on portfolio choice problems for the ambiguity-averse investors. For example, Maenhout \cite{Maenhout2004} investigates the impacts of ambiguity on the intertemporal portfolio selection in a setting with constant investment opportunities. Zeng et al. \cite{Zeng2016} study the equilibrium strategy of a robust optimal reinsurance-investment problem under the mean-variance criterion in a model with jumps. Pun \cite{Pun2018} establishes a general analytical framework for continuous-time stochastic control problems for an ambiguity-averse agent with time-inconsistent preference. Similar investigations and results of robust portfolio selection can be also found in \cite{Branger1, Branger2, Wang2021, Yan2020, Yang20a, Yi2013}.

To the best of our knowledge, there is few work in literature considering robust time-consistent mean-variance-skewness portfolio selection problems for the ambiguity-averse investor under the mean-variance-skewness criterion with wealth-dependent risk aversion and wealth-dependent skewness preference as well as model uncertainty in the continuous time economy.  Inspired and motivated by the above research works, we attempt to fill this gap by establishing such a model and studying how wealth-dependent risk aversion and wealth-dependent skewness preference as well as model uncertainty impact the robust equilibrium investment strategy in the continuous-time financial market. The main contribution of this paper is threefold. First, wealth-dependent risk aversion, skewness preference with wealth-dependent coefficient and model uncertainty are introduced into the mean-variance portfolio problem and then a new robust investment model is established, which extends the models investigated in \cite{Mu2019,Pun2018}. Specially, the wealth-dependent risk aversion, wealth-dependent
skewness preference and ambiguity-aversion could be dependent, while some studies (see, for example, \cite{Yan2020, Zeng2016}) only consider the independent case of risk aversion and ambiguity-aversion for simplicity. Second, the detailed verification of the admissibility of our newly derived robust equilibrium investment strategy is given, which makes our paper more complete from a mathematical point of view, whereas neither \cite{Mu2019} nor \cite{Pun2018} verifies it. Third, the effects of model parameters on the robust equilibrium investment strategy and the utility losses from ignoring the skewness preference and model uncertainty
are studied, which are not considered in \cite{Mu2019,Pun2018}.

The rest of this paper is organized as follows. Next section gives the characterization of the continuous time robust mean-variance-skewness portfolio selection problem. In Section 3, we adopt the game theoretic approach to derive an extended HJBI system and an verification theorem to capture the robust equilibrium investment strategy and corresponding equilibrium value function for the robust time-consistent mean-variance-skewness portfolio selection problem. In Section 4, we obtain the robust equilibrium investment strategy in semi-closed form for a special continuous time robust mean-variance-skewness portfolio selection model.  Before summarizing this paper in Section 6, we provide some numerical experiments to illustrate our results in Section 5.

\section{Problem formulation}
 Throughout this paper, we assume that an investor can continuously invest in a financial market in which there are no transaction costs or taxes. Given a positive finite constant $T$ representing the terminal time, let $(\Omega,\mathcal{F},\mathbb{F},\mathbb{P})$ denote a complete probability space equipped with a natural filtration $\mathbb{F}=(\mathcal{F}_{t})_{t\in[0,T]}$ satisfying the usual hypothesis and a physical probability measure $\mathbb{P}$, where $\mathbb{F}$ is generated by $\mathbb{P}$-Brownian motions $W^{S}_{t}$ and $W^{Z}_{t}$ to be defined later. Let $\mathbb{C}^{1,2,1,2}([0,T]\times R^{m}\times [0,T]\times R^{n})$ be the set of all the continuous functions $\phi=\phi(t,x,s,y):[0,T]\times R^{m}\times [0,T]\times R^{n}\rightarrow R$ such that $\frac{\partial \phi}{\partial t}$, $\frac{\partial \phi}{\partial x}$, $\frac{\partial^{2} \phi}{\partial x^{2}}$, $\frac{\partial \phi}{\partial s}$, $\frac{\partial \phi}{\partial y}$, $\frac{\partial^{2} \phi}{\partial y^{2}}$ exist and are continuous in $(t,x,s,y)$ for generic integers $m,n \geq1$. With any generic integers $p,n \geq1$ and a given measure $\mathbb{P}$, we denote
$$
\mathbb{L}_{\mathcal{F}}^{p}(0,T;R^{n},\mathbb{P})=\left\{X:[0,T]\times\Omega\rightarrow R^{n}\bigg| X \text{ is }  \{\mathcal{F}_{t}\}_{t\in[0,T]}\text{-adapted}, \;\mathbb{E}^{\mathbb{P}}\left[\int^{T}_{0}\parallel X(s)\parallel^{p}\mathrm{d}s\right]<\infty\right\}.
$$

\subsection{Model assumptions}
We consider a financial market with a risk-free asset (e.g., bank deposit) and $M$ risky assets (e.g., stocks). Specifically, we assume that the price process $S^{0}_{t}$ of the risk-free asset follows
\begin{equation}\label{eq1}
  \mathrm{d}S^{0}_{t}=r(t,Z_{t})S^{0}_{t}\mathrm{d}t
\end{equation}
and the price process $S^{i}_{t}$ of the $i$th risky asset satisfies the following stochastic differential equation (SDE)
\begin{equation}\label{eq2}
  \mathrm{d}S^{i}_{t}=\mu_{i}(t,Z_{t})S^{i}_{t}\mathrm{d}t+\sigma_{i}(t,Z_{t})'S^{i}_{t}\mathrm{d}W^{S}_{t}, \quad i=1,\cdots,M.
\end{equation}
Here, $r$ and $\mu_{i}$ are one dimensional functions, $\sigma_{i}$ is an $M$-dimensional vector function, $W^{S}_{t}$ is a standard $M$-dimensional Brownian motion, and $Z_{t}$ is an $N$-dimensional vector denoting the state variable in the financial market and is assumed to evolve as follows
\begin{equation}\label{eq3}
  \mathrm{d}Z_{t}=\chi(t,Z_{t})\mathrm{d}t+\nu(t,Z_{t})\mathrm{d}W^{Z}_{t},
\end{equation}
where $\chi$ is an $N$-dimensional vector function, $\nu\nu'$ is an $N\times N$ matrix and $W^{Z}_{t}$ is an $N$-dimensional Brownian motion satisfying $Cov(W^{S}_{t},(W^{Z}_{t})')$ $ =\rho(t)t$ where $\rho(t)$ is an $M\times N$ matrix function of $t$ and its elements satisfy $\rho_{i,j}(t)\in[-1,1]$ for $i=1,\cdots,M,\;j=1,\cdots,N$. Note that the randomness of this financial market is captured by an $(M+N)$-dimensional Brownian motion $({W^{S}_{t}}',{W^{Z}_{t}}')'$. For convenience, the $M\times1$ vector $[\mu_{1}(t,Z_{t}),\cdots,\mu_{M}(t,Z_{t})]'$ and the $M\times M$ matrix $[\sigma_{1}(t,Z_{t}),\cdots,\sigma_{M}(t,Z_{t})]'$ are denoted by $\mu_{t}$ and $\sigma_{t}$, respectively, when no confusion occurs. Similar expressions are also used later.

In what follows, we suppose that the self-financing investor can invest in the financial market at time $t$ over the finite horizon $[0,T]$. The investment strategy is denoted by $\bm{\pi}=\{u_{t}\}_{t\in[0,T]}$, where $u_{t}=(u^{1}_{t},\cdots,u^{M}_{t})'$
and $u^{i}_{t}$ corresponds to the money amount invested in the $i$th risky asset. Therefore, the money amount invested in the risk-free asset can be presented by $W_{t}-\sum_{i=1}^{M} u^{i}_{t}$, where $W_{t}$ is the investor's wealth at time $t$. Given a positive initial wealth $w_{0}$, the investor's wealth process $W_{t}$ for $t\in[0,T]$ is governed by
\begin{equation}\label{eq4}
   \mathrm{d}W_{t}=\left[r_{t}W_{t}+u_{t}'\left(\mu_{t}-r_{t} \textbf{1}\right)\right]\mathrm{d}t+u_{t}'\sigma_{t}\mathrm{d}W^{S}_{t},
\end{equation}
where $\textbf{1}:=(1,\cdots,1)'\in R^{M}$.

\subsection{Ambiguity-averse investor under mean-variance-skewness preference}
In the following, we are interested in the investment problem of the ambiguity-averse investor under mean-variance-skewness criterion with wealth-dependent risk aversion and skewness preference. To proceed, we denote a state variable by $X=(W,Z')'$ and call $\mathbb{P}$ the reference measure.

It is well known that the mean-variance criterion is famous and popular in portfolio selection problems because mean and variance can be used to describe the expected return and the risk of investment, respectively. Moreover, in order to capture investor's risk preference, state-dependent risk aversion and higher order risk preferences have been considered in portfolio selection problems (see, for example, \cite{Bjork2014, LiY2013Optimal, Mu2019}). We note that in \cite{Mu2019},  the ambiguity-neutral investor's objective function is assumed to have the following form
\begin{equation}\label{eq5}
  \sup\limits_{\bm{\pi}\in\Pi}\quad \mathbb{E}^{\mathbb{P}}_{t,x}\left[W_{T}\right]-\frac{\gamma(W_{t})}{2}Var^{\mathbb{P}}_{t,x}\left[W_{T}\right]+\frac{\phi(W_{t})}{3}Skew^{\mathbb{P}}_{t,x}\left[W_{T}\right],
\end{equation}
where $x=X_{t}$, $\mathbb{E}^{\mathbb{P}}_{t,x}$ and $Var^{\mathbb{P}}_{t,x}$ are the conditional expectation and the conditional variance under the single physical probability measure $\mathbb{P}$, respectively, $W_{T}$ is the investor's terminal wealth under the investment strategy $\bm{\pi}$, $\Pi$ is a set of admissible investment strategies that will be given below, $\gamma(W_{t})$ is the wealth-dependent risk aversion, $\phi(W_{t})$ is the wealth-dependent preference for skewness, and $$
Skew^{\mathbb{P}}_{t,x}\left[W_{T}\right]=\mathbb{E}^{\mathbb{P}}_{t,x}\left[(W_{T})^{3}\right]-3\mathbb{E}^{\mathbb{P}}_{t,x}\left[W_{T}\right]
\mathbb{E}^{\mathbb{P}}_{t,x}\left[(W_{T})^{2}\right]+2(\mathbb{E}^{\mathbb{P}}_{t,x}\left[W_{T}\right])^{3}.
$$
In addition, similar to the previous investigations on the high-order preference \cite{Brocas2016, Eeckhoudt2006, Kimball1990}, we can suppose that $\gamma(W_{t})>0$, $\frac{\mathrm{d} \gamma(W_{t})}{\mathrm{d}W_{t}}<0$, $\phi(W_{t})>0$ and $\frac{\mathrm{d} \phi(W_{t})}{\mathrm{d}W_{t}}<0$. Equations $\eqref{eq2}$-$\eqref{eq5}$ are referred to as the reference model. Note that the ambiguity-neutral investor has the complete confidence in the reference model under the reference measure $\mathbb{P}$.

However, any choice of probability measure would result in model misspecifications because the investor is unable to obtain all the information about the true model in general. Thus, it is desirable to consider models with uncertainties. More specifically, we assume that the investor chooses the reference measure $\mathbb{P}$ to regard it as an approximation of the true measure and then take into account alternative measures to obtain the robust optimal investment strategy. It is well known that the corresponding alternative models can be conceptually defined as the models that are "similar" to the reference model and the "similarity" can be characterized by the mathematical concept of measure equivalence \cite{Anderson2003}. Consider a set of measures that are equivalent to $\mathbb{P}$: $\mathcal{Q}:=\{\mathbb{Q}:\mathbb{Q}\sim\mathbb{P}\}$. By the Girsanov theorem, we know that there exists an $\mathbb{F}$-adapted stochastic process $\textbf{q}=\{(q^{S'}_{t},q^{Z'}_{t})'\}_{t\in[0,T]}$, where $q^{S}$ and $q^{Z}$ are valued in $R^{M}$ and $R^{N}$, respectively, such that
\begin{equation}\label{eq6}
  \frac{\mathrm{d}\mathbb{Q}}{\mathrm{d}\mathbb{P}}\bigg|_{\mathcal{F}_{t}}=\mathcal{\varepsilon}_{t}\left(q^{S}\cdot W^{S}\right)\mathcal{\varepsilon}_{t}\left(q^{Z}\cdot W^{Z}\right),
\end{equation}
where the stochastic exponential $\mathcal{\varepsilon}_{t}$ is defined as
$$
\mathcal{\varepsilon}_{t}\left(Y\cdot W^{Y}\right)=\exp\left(-\frac{1}{2}\int^{t}_{0}\|Y(s)\|^{2}\mathrm{d}s+\int^{t}_{0}Y(s)'\mathrm{d}W^{Y}(s)\right)
$$
for any measurable process $Y$ adapted to the natural filtration generated by the standard Brownian motion $W^{Y}$. Moreover, if $\textbf{q}$ satisfies the Novikov condition
\begin{equation}\label{eq7}
\mathbb{E}^{\mathbb{P}}\left[\exp\left(\frac{1}{2}\int^{T}_{0}\|\textbf{q}_{s}\|^{2}\mathrm{d}s\right)\right]<\infty,
\end{equation}
then two processes
\begin{equation*}
   W^{S,\mathbb{Q}}_{t}=W^{S}_{t}-\int^{t}_{0}q^{S}_{s}\mathrm{d}s,  \quad W^{Z,\mathbb{Q}}_{t}=W^{Z}_{t}-\int^{t}_{0}q^{Z}_{s}\mathrm{d}s
\end{equation*}
are standard $\mathbb{Q}$-Brownian motions valued in $R^{M}$ and $R^{N}$, respectively. Thus, equations $\eqref{eq2}$-$\eqref{eq4}$ in the reference model can be replaced by the following SDEs:
\begin{align*}
\begin{cases}
\mathrm{d}S^{i}_{t}=\left(\mu_{i}(t,Z_{t})S^{i}_{t}+\sigma_{i}(t,Z_{t})^{'}S^{i}_{t}q^{S}_{t}\right)\mathrm{d}t+\sigma_{i}(t,Z_{t})^{'}S^{i}_{t}\mathrm{d}W^{S,\mathbb{Q}}_{t},\\
\mathrm{d}Z_{t}=\left(\chi(t,Z_{t})+\nu(t,Z_{t})q^{Z}_{t}\right)\mathrm{d}t+\nu(t,Z_{t})\mathrm{d}W^{Z,\mathbb{Q}}_{t},\\
\mathrm{d}W_{t}=\left[r_{t}W_{t}+u_{t}'\left(\mu_{t}-r_{t} \textbf{1}\right)+u_{t}'\sigma_{t}q^{S}_{t}\right]\mathrm{d}t+u_{t}'\sigma_{t}\mathrm{d}W^{S,\mathbb{Q}}_{t}
\end{cases}
\end{align*}
and the dynamic process of $X$ can be expressed as
\begin{equation}\label{eq8}
  \mathrm{d}X_{t}=\Gamma(t,x,\bm{\pi},\textbf{q})\mathrm{d}t+\Psi(t,x,\bm{\pi},\textbf{q})\mathrm{d}W^{X}_{t}.
\end{equation}
Here
$$
\Gamma(t,x,\bm{\pi},\textbf{q})=
\begin{bmatrix}
r_{t}W_{t}+u_{t}'\left(\mu_{t}-r_{t} \textbf{1}\right)+u_{t}'\sigma_{t}q^{S}_{t}\\
\chi_{t}+\nu_{t}q^{Z}_{t}
\end{bmatrix},
\quad
\Psi(t,x,\bm{\pi},\textbf{q})=
\begin{bmatrix}
  u_{t}'\sigma_{t} &  \textbf{0}\\
  \textbf{0}  &  \nu_{t}
\end{bmatrix},
\quad
\mathrm{d}W^{X}_{t}=
\begin{bmatrix}
\mathrm{d}W^{S,\mathbb{Q}}_{t}\\
\mathrm{d}W^{Z,\mathbb{Q}}_{t}
\end{bmatrix},
$$
where $\textbf{0}$ represents a vector of zeros with appropriate dimension. In addition, let $\Omega(t,x,\bm{\pi},\textbf{q})$ denote the variance-covariance matrix of $X_{t}$ satisfying
$$
\Omega(t,x,\bm{\pi},\textbf{q})=(\Omega_{i,j})=\Psi(t,x,\bm{\pi},\textbf{q})\cdot
\begin{bmatrix}
 I_{M} &   \rho_{t}\\
 \rho'_{t}  &  I_{N}
\end{bmatrix}\cdot\Psi(t,x,\bm{\pi},\textbf{q})',
$$
where $I_{i}$ is an $i\times i\;(i=M, N)$ identity matrix and the operator $\cdot$ denotes the Hadamard product of two matrices.

Based on the above discussions, we incorporate the model uncertainty into the investment problem of the ambiguity-averse investor under mean-variance-skewness preference. Let $W^{\mathbb{P}}_{t}=({W^{S}_{t}}^{'},{W^{Z}_{t}}^{'})^{'}$. Motivated by the investigations of \cite{Pun2018, Pun2}, the objective function $\eqref{eq5}$ can be changed into
\begin{equation}\label{eq9}
  \sup\limits_{\bm{\pi}\in \Pi}\inf\limits_{\textbf{q}\in\bm{Q}}\quad \mathbb{E}^{\mathbb{Q}}_{t,x}\left[W_{T}\right]-\frac{\gamma(W_{t})}{2}Var^{\mathbb{Q}}_{t,x}\left[W_{T}\right]
  +\frac{\phi(W_{t})}{3}Skew^{\mathbb{Q}}_{t,x}\left[W_{T}\right]+D_{t,x}(\mathbb{Q}\parallel\mathbb{P}),
\end{equation}
where $\bm{Q}$ is a set of admissible measures that will be given below and $D_{t,x}(\mathbb{Q}\parallel\mathbb{P})$ is a generalized Kullback-Leibler divergence between measures $\mathbb{Q}$ and $\mathbb{P}$, which is given by the following form \cite{Maenhout2004}
$$
\mathbb{E}^{\mathbb{Q}}_{t,x}\left[\int^{T}_{t}C(x,s,X_{s};\textbf{q})\mathrm{d}s\right]:=\mathbb{E}^{\mathbb{Q}}_{t,x}
\left[\int^{T}_{t}\frac{\textbf{q}'_{s}\Xi(x)^{-1}\textbf{q}_{s}}{2\Phi(s,X_{s}) }\mathrm{d}s\right].
$$
Here, $\Xi(x)=diag(\xi_{1}(x),\cdots,\xi_{M+N}(x))$ ($\xi_{i}(x)>0$) is a diagonal matrix describing the state-dependent ambiguity aversion level about the uncertainty source $W^{\mathbb{P}}_{it}$ of the state variable $X_{t}$ for $i=1,\cdots,M+N$, and $\Phi(\cdot,\cdot)$ is an $R^{+}$-valued function measuring the strength of the investor's preference for ambiguity, i.e., $\Phi(\cdot,\cdot)$ stands for the degree of the investor's confidence in the reference model. We notice that functions $\gamma(\cdot)$, $\phi(\cdot)$, $\Xi(\cdot)$ and $\Phi(\cdot,\cdot)$ are state-dependent and all $\gamma(W_{t})$, $\phi(W_{t})$ and $\Xi(x)$ depend on the wealth value $W_{t}$ at time $t$ especially, which reflects that the wealth-dependent risk aversion, wealth-dependent skewness preference and ambiguity-aversion could be dependent in this paper.

Before solving the model $\eqref{eq9}$, we consider a more general model as follows
\begin{align}\label{eq10}
&\sup\limits_{\bm{\pi}\in \Pi}\inf\limits_{\textbf{q}\in\bm{Q}} \; J(t,x;\bm{\pi},\textbf{q}):=\\
&\sup\limits_{\bm{\pi}\in \Pi}\inf\limits_{\textbf{q}\in\bm{Q}} \; \left\{\mathbb{E}^{\mathbb{Q}}_{t,x}\left[F(t,x,W_{T})\right]+G\left(t,x,\mathbb{E}^{\mathbb{Q}}_{t,x}\left[W_{T}\right]\right)
  +K\left(t,x,\mathbb{E}^{\mathbb{Q}}_{t,x}\left[W_{T}\right],\mathbb{E}^{\mathbb{Q}}_{t,x}\left[(W_{T})^{2}\right]\right)
  +D_{t,x}(\mathbb{Q}\parallel\mathbb{P})\right\},\nonumber
\end{align}
where $F$, $G$ and $K$ are arbitrary functions valued in $R$, and $D_{t,x}(\mathbb{Q}\parallel\mathbb{P})=\mathbb{E}^{\mathbb{Q}}_{t,x}\left[\int^{T}_{t}C(t,x,s,X_{s};\textbf{q})\mathrm{d}s\right]$.

Clearly, if functions $F$, $G$ and $K$ are given by
\begin{equation}\label{eq11}
  F(t,x,y)=y-\frac{\gamma(x)}{2}y^{2}+\frac{\phi(x)}{3}y^{3},\quad G(t,x,g)=\frac{\gamma(x)}{2}g^{2}+\frac{2\phi(x)}{3}g^{3}, \quad K(t,x,g,h)=-\phi(x)gh,
\end{equation}
respectively, then the model $\eqref{eq10}$ can degenerate into $\eqref{eq9}$.

Since the state-dependence of $F$, $G$, $K$ and $D_{t,x}$ as well as the non-linearity of $G$ and $K$ make Bellman's dynamic programming principle invalid, model \eqref{eq10} belongs to the robust time-inconsistent stochastic control problems. The overarching goal of this paper is to deal with the time-inconsistency by using the method of game theory and to obtain the robust equilibrium investment strategy for model \eqref{eq10}.
\begin{remark}\label{remark1}
{\rm(i)} If $K\equiv0$, then model $\eqref{eq10}$ becomes the one studied in \cite{Pun2018, Pun2}; {\rm(ii)} If $D_{t,x}(\mathbb{Q}\parallel\mathbb{P})=0$, then model $\eqref{eq10}$ taking the form of $\eqref{eq11}$ degenerates into the one investigated in \cite{Mu2019}. Moreover, if $K\equiv0$, then model $\eqref{eq10}$ is the same as the one considered in \cite{Bjork2014}.
\end{remark}

\section{Robust equilibrium investment strategy}
This section adopts the game theoretic framework to formulate the robust time-inconsistent stochastic control problem \eqref{eq10} and attempts to derive the robust equilibrium investment strategy.

When performing a stochastic optimization, one typically uses a model which describes the dynamics of relevant quantities and then finds
the strategy which optimizes their subjective performance. Once the investment strategy $\bm{\pi}$ is identified as a control variable, the corresponding admissible control-measure strategy $(\bm{\pi},\bm{q})$ for the robust problem should also be determined. Following the work \cite{Pun2018, Pun2}, we can define the admissible control-measure strategy for \eqref{eq10} as follows.
\begin{defn}\label{definition3.1}
A control-measure strategy $(\bm{\pi},\bm{q})$ of  \eqref{eq10} is admissible if the following conditions are satisfied:
\begin{itemize}
\item[(i)] for any $t\in[0,T]$, $\bm{\pi}=\{u_{t}\}$ is $\mathcal{F}_{t}$-progressively measurable;
\item[(ii)] $\bm{q}$ satisfies the Novikov condition $\eqref{eq7}$;
\item[(iii)] for any initial point $(t, x)\in[0,T]\times R^{1+N}$ and $\bm{q}$, the SDE $\eqref{eq8}$ for $\{X_{s}\}_{s\in[t,T]}$ with $X_{t}=x$ admits a unique strong solution;
\item[(iv)] for any initial point $(t, x)\in[0,T]\times R^{1+N}$,
$$
\sup\limits_{\mathbb{Q}\in \mathcal{Q}}\mathbb{E}^{\mathbb{Q}}_{t,x}\left[\int^{T}_{t}\mid C(t,x,s,X_{s};\textbf{q})\mid \mathrm{d}s+\mid F(t,x,W_{T})\mid\right]<\infty,\quad \sup\limits_{\mathbb{Q}\in \mathcal{Q}}\mathbb{E}^{\mathbb{Q}}_{t,x}\left[|W_{T}|^{2}\right]<\infty
$$
and
$$
\sup\limits_{\mathbb{Q}\in \mathcal{Q}}\left[G\left(t,x,\mathbb{E}^{\mathbb{Q}}_{t,x}\left[W_{T}\right]\right)
  +K\left(t,x,\mathbb{E}^{\mathbb{Q}}_{t,x}\left[W_{T}\right],\mathbb{E}^{\mathbb{Q}}_{t,x}\left[(W_{T})^{2}\right]\right)\right]<\infty.
$$
\end{itemize}
\end{defn}

We denote by $\Pi \times \bm{Q}$ the set of admissible control-measure strategies. Following the idea of \cite{Pun2018}, we can view the robust time-inconsistent stochastic control problem $\eqref{eq10}$ over $[0,T]$ as a non-cooperative game with infinite investors indexed by $t\in[0,T]$, where investor $t$ can be regarded as the future incarnation of the investor as well as she/he can control $X$ at time $t$ by $\bm{\pi}$ and meanwhile plays another Stackelberg game against the "nature" who can control the measure through $\bm{q}$. Thus,  problem $\eqref{eq10}$ can be seen as "games in subgames". In this case, the control-measure strategy $(\bm{\pi},\bm{q})$ in Definition \ref{definition3.1} can be rewritten as
\begin{equation*}
  (\bm{\pi},\bm{q})=\left(\{u(t,x)'\}_{t\in[0,T]},
\{q^{S}(t,x,u(t,x))'\}_{t\in[0,T]},\{q^{Z}(t,x,u(t,x))'\}_{t\in[0,T]}\right)'.
\end{equation*}
It reflects that $\bm{q}$ is a function of $\bm{\pi}$ but not vice versa. Moreover, the time-consistent control-measure strategy $(\bm{\pi}^{\ast},\bm{q}^{\ast})$ is a subgame perfect Nash equilibrium, which means that for any $t\in[0,T)$, if investors $s\in(t,T]$ choose the optimal control-measure strategies $({\bm{\pi}_{s}}^{\ast},{\bm{q}_{s}}^{\ast})_{s\in(t,T]}$, respectively, then the investor $t$ would choose her/his optimal control-measure strategy $({\bm{\pi}_{t}}^{\ast},{\bm{q}_{t}}^{\ast})$, where
\begin{equation*}
  ({\bm{\pi}_{s}}^{\ast},{\bm{q}_{s}}^{\ast})=\left(u^{\ast}(s,X_{s})',
{q^{S}}^{\ast}(s,X_{s},u^{\ast}(s,X_{s}))',{q^{Z}}^{\ast}(s,X_{s},u^{\ast}(s,X_{s}))'\right)'
\end{equation*}
and
\begin{equation*}
 ({\bm{\pi}_{t}}^{\ast},{\bm{q}_{t}}^{\ast})=\left(u^{\ast}(t,X_{t})',
{q^{S}}^{\ast}(t,X_{t},u^{\ast}(t,X_{t}))',{q^{Z}}^{\ast}(t,X_{t},u^{\ast}(t,X_{t}))'\right)'.
\end{equation*}
Since the Lebesgue measure of one point $t$ is zero, the control $u^{\ast}(t,X_{t})$ has no impact. Therefore, we instead study the problem $\eqref{eq10}$ over $[t,t+\varepsilon)$ with a minimal time elapse $\varepsilon>0$ to get the robust optimal control given that investors $s\in[t+\varepsilon,T]$ have obtained their robust optimal controls (see, for example, \cite{Basak2010, Bjork2009, Bjork2014}).

Then, we follow \cite{Pun2018} to give the descriptions of the robust equilibrium investment strategy and equilibrium value function.
\begin{defn}\label{definition3.2}
Given an admissible control-measure strategy $(\bm{\pi}^{\ast},\bm{q}^{\ast})=(u^{\ast'},{q^{S}}^{\ast'},{q^{Z}}^{\ast'})'$. For any fixed state $(t,x)\in[0,T)\times R^{1+N}$ and a fixed real number $\varepsilon>0$, we define a strategy $(\bm{\pi}^{\varepsilon},\bm{q}^{\varepsilon})=(u^{\varepsilon'},q^{S,\varepsilon'},q^{Z,\varepsilon'})'$ as
\begin{align}\label{eq12}
\begin{cases}
u^{\varepsilon}(s,y)=u(s,y)\textbf{1}_{\{s\in[t,t+\varepsilon)\}}+u^{\ast}(s,y)\textbf{1}_{\{s\in[t+\varepsilon,T]\}},\\
q^{S,\varepsilon}(s,y,u^{\varepsilon})=q^{S}(s,y,u)\textbf{1}_{\{s\in[t,t+\varepsilon)\}}+{q^{S}}^{\ast}(s,y,u^{\ast})\textbf{1}_{\{s\in[t+\varepsilon,T]\}},\\
q^{Z,\varepsilon}(s,y,u^{\varepsilon})=q^{Z}(s,y,u)\textbf{1}_{\{s\in[t,t+\varepsilon)\}}+{q^{Z}}^{\ast}(s,y,u^{\ast})\textbf{1}_{\{s\in[t+\varepsilon,T]\}},
\end{cases}
\end{align}
where $y\in R^{1+N}$, $(\bm{\pi},\bm{q})=(u',q^{S'},q^{Z'})'$ is an arbitrary chosen admissible control-measure strategy and $\textbf{1}_{A}$ denotes the indicator function of set $A$. If
\begin{itemize}
\item[(i)] for any admissible $(\bm{\pi},\bm{q})$
$$
\underset{\varepsilon\rightarrow0}\liminf \frac{J(t,x;\bm{\pi}^{\varepsilon},\bm{q}^{\varepsilon})-J(t,x;\bm{\pi}^{\varepsilon},\bm{q}^{\ast})}{\varepsilon}\geq0;
$$
\item[(ii)] and for all $\bm{\pi}$
$$
\underset{\varepsilon\rightarrow0}\limsup \frac{J(t,x;\bm{\pi}^{\varepsilon},\bm{q}^{\ast})-J(t,x;\bm{\pi}^{\ast},\bm{q}^{\ast})}{\varepsilon}\leq0;
$$
\end{itemize}
then $(\bm{\pi}^{\ast},\bm{q}^{\ast})$ is called the equilibrium control-measure strategy, $\bm{\pi}^{\ast}$ is called the robust equilibrium investment strategy and the corresponding equilibrium value function can be defined as $V(t,x)=J(t,x;\bm{\pi}^{\ast},\bm{q}^{\ast})$.
\end{defn}

It follows from the above definition that our target is to solve the robust time-inconsistent stochastic control problem $\eqref{eq10}$ by finding the equilibrium control-measure strategy. Motivated by \cite{Pun2018, Pun2}, we can obtain a recursive equation for the objective function $J(t,x;\bm{\pi},\textbf{q})$ as follows.
\begin{lemma}\label{lemma2.1}
For any $s>t$ and arbitrary chosen admissible control-measure strategy $(\bm{\pi},\bm{q})\in\Pi \times \bm{Q}$, one has
\begin{equation}\label{eq13}
  J(t,x;\bm{\pi},\textbf{q})=\mathbb{E}^{\mathbb{Q}}_{t,x}\left[J(s,X_{s};\bm{\pi},\textbf{q})+\int^{s}_{t}C(\tau,X_{\tau},\tau,X_{\tau};\textbf{q})\mathrm{d}\tau\right]+
  L^{\bm{\pi},\textbf{q}}(t,x,s)
\end{equation}
with
$$
L^{\bm{\pi},\textbf{q}}(t,x,s)=L^{\bm{\pi},\textbf{q}}_{C}(t,x,s)+L^{\bm{\pi},\textbf{q}}_{H}(t,x,s)+L^{\bm{\pi},\textbf{q}}_{G}(t,x,s)
+L^{\bm{\pi},\textbf{q}}_{K}(t,x,s),
$$
where
\begin{align*}
L^{\bm{\pi},\textbf{q}}_{C}(t,x,s)=&\mathbb{E}^{\mathbb{Q}}_{t,x}\left[\int^{s}_{t}\left[C(t,x,\tau,X_{\tau};\textbf{q})
-C(\tau,X_{\tau},\tau,X_{\tau};\textbf{q})\right]\mathrm{d}\tau\right],\\
L^{\bm{\pi},\textbf{q}}_{H}(t,x,s)=&\mathbb{E}^{\mathbb{Q}}_{t,x}\left[h^{\bm{\pi},\textbf{q}}(s,X_{s},t,x)\right]
-\mathbb{E}^{\mathbb{Q}}_{t,x}\left[h^{\bm{\pi},\textbf{q}}(s,X_{s},s,X_{s})\right],\\
L^{\bm{\pi},\textbf{q}}_{G}(t,x,s)=&G\left(t,x,\mathbb{E}^{\mathbb{Q}}_{t,x}\left[g^{\bm{\pi},\textbf{q}}(s,X_{s})\right]\right)
-\mathbb{E}^{\mathbb{Q}}_{t,x}\left[G\left(s,X_{s},g^{\bm{\pi},\textbf{q}}(s,X_{s})\right)\right],\\
L^{\bm{\pi},\textbf{q}}_{K}(t,x,s)=&K\left(t,x,\mathbb{E}^{\mathbb{Q}}_{t,x}\left[g^{\bm{\pi},\textbf{q}}(s,X_{s})\right],\mathbb{E}^{\mathbb{Q}}_{t,x}\left[k^{\bm{\pi},\textbf{q}}(s,X_{s})\right]\right)
-\mathbb{E}^{\mathbb{Q}}_{t,x}\left[K\left(s,X_{s},g^{\bm{\pi},\textbf{q}}(s,X_{s}),k^{\bm{\pi},\textbf{q}}(s,X_{s})\right)\right],\\
h^{\bm{\pi},\textbf{q}}(t,x,s,y)=&\mathbb{E}^{\mathbb{Q}}_{t,x}\left[\int^{T}_{t}C(s,y,\tau,X_{\tau};\textbf{q})\mathrm{d}\tau+F(s,y,W_{T})\right],\\
g^{\bm{\pi},\textbf{q}}(t,x)=&\mathbb{E}^{\mathbb{Q}}_{t,x}\left[W_{T}\right],\;k^{\bm{\pi},\textbf{q}}(t,x)=\mathbb{E}^{\mathbb{Q}}_{t,x}\left[(W_{T})^{2}\right].
\end{align*}
\end{lemma}
\begin{proof}
See Appendix A.
\end{proof}

It is worth mentioning that if either one of functions $F$, $G$, $K$ and $D_{t,x}$ is state-dependent, or one of $G(t,x,y)$ and $K(t,x,y,z)$ is nonlinear in $y$ and/or $z$, then $L^{\bm{\pi},\textbf{q}}$ will be generally nonzero and it quantifies the investor's incentives to deviate from the optimal investment strategy of time $t$ during the period $[t,s]$, which yields the violation of Bellman's dynamic programming principle. To further deal with the time inconsistency, we adopt the framework of game theory and assume that the equilibrium control-measure strategy for the problem $\eqref{eq10}$ exists and the corresponding equilibrium value function at time $t$ is denoted by $V(t,x)$. For any initial state $(t,x)$ and a real number $\varepsilon>0$, we consider the problem $\eqref{eq10}$ over $[t, t+\varepsilon)$ where the control-measure strategies over $[t+\varepsilon, T]$ are fixed as $(\bm{\pi}^{\ast},\bm{q}^{\ast})$ and the control-measure strategies over $[t, t+\varepsilon)$ are arbitrary, written as $(\bm{\pi},\bm{q})$. In view of the recursive equation $\eqref{eq13}$ for the objective function $J(t,x;\bm{\pi},\textbf{q})$ and Definition \ref{definition3.2}, we can derive the following extended dynamic programming equation (EDPE) for the value function $V$:
\begin{equation*}
  V(t,x)=\sup\limits_{\bm{\pi}\in \Pi}\inf\limits_{\textbf{q}\in\bm{Q}}\left\{\mathbb{E}^{\mathbb{Q}}_{t,x}\left[V(t+\varepsilon,X_{t+\varepsilon})+\int^{t+\varepsilon}_{t}C(\tau,X_{\tau},\tau,X_{\tau};\textbf{q})\mathrm{d}\tau\right]+
  L^{\bm{\pi}^{\varepsilon},\bm{q}^{\varepsilon}}(t,x,t+\varepsilon)\right\}
\end{equation*}
with
$$
L^{\bm{\pi}^{\varepsilon},\bm{q}^{\varepsilon}}(t,x,t+\varepsilon)=L^{\bm{\pi}^{\varepsilon},\bm{q}^{\varepsilon}}_{C}(t,x,t+\varepsilon)
+L^{\bm{\pi}^{\varepsilon},\bm{q}^{\varepsilon}}_{H}(t,x,t+\varepsilon)
+L^{\bm{\pi}^{\varepsilon},\bm{q}^{\varepsilon}}_{G}(t,x,t+\varepsilon)+L^{\bm{\pi}^{\varepsilon},\bm{q}^{\varepsilon}}_{K}(t,x,t+\varepsilon),
$$
where
\begin{align*}
&L^{\bm{\pi}^{\varepsilon},\bm{q}^{\varepsilon}}_{C}(t,x,t+\varepsilon)=\mathbb{E}^{\mathbb{Q}}_{t,x}\left[\int^{t+\varepsilon}_{t}\left[C(t,x,\tau,X_{\tau};\textbf{q})
-C(\tau,X_{\tau},\tau,X_{\tau};\textbf{q})\right]\mathrm{d}\tau\right],\\
&L^{\bm{\pi}^{\varepsilon},\bm{q}^{\varepsilon}}_{H}(t,x,t+\varepsilon)=\mathbb{E}^{\mathbb{Q}}_{t,x}\left[h^{\bm{\pi}^{\ast},\bm{q}^{\ast}}(t+\varepsilon,X_{t+\varepsilon},t,x)\right]
-\mathbb{E}^{\mathbb{Q}}_{t,x}\left[h^{\bm{\pi}^{\ast},\bm{q}^{\ast}}(t+\varepsilon,X_{t+\varepsilon},t+\varepsilon,X_{t+\varepsilon})\right],\\
&L^{\bm{\pi}^{\varepsilon},\bm{q}^{\varepsilon}}_{G}(t,x,t+\varepsilon)=G\left(t,x,\mathbb{E}^{\mathbb{Q}}_{t,x}
\left[g^{\bm{\pi}^{\ast},\bm{q}^{\ast}}(t+\varepsilon,X_{t+\varepsilon})\right]\right)
-\mathbb{E}^{\mathbb{Q}}_{t,x}\left[G\left(t+\varepsilon,X_{t+\varepsilon},g^{\bm{\pi}^{\ast},\bm{q}^{\ast}}(t+\varepsilon,X_{t+\varepsilon})\right)\right],\\
&L^{\bm{\pi}^{\varepsilon},\bm{q}^{\varepsilon}}_{K}(t,x,t+\varepsilon)=K\left(t,x,\mathbb{E}^{\mathbb{Q}}_{t,x}
\left[g^{\bm{\pi}^{\ast},\bm{q}^{\ast}}(t+\varepsilon,X_{t+\varepsilon})\right],\mathbb{E}^{\mathbb{Q}}_{t,x}\left[k^{\bm{\pi}^{\ast},\bm{q}^{\ast}}(t+\varepsilon,X_{t+\varepsilon})\right]\right)\\
&\quad\quad\quad\quad\quad\quad\quad\quad\quad
-\mathbb{E}^{\mathbb{Q}}_{t,x}\left[K\left(t+\varepsilon,X_{t+\varepsilon},g^{\bm{\pi}^{\ast},\bm{q}^{\ast}}(t+\varepsilon,X_{t+\varepsilon}),k^{\bm{\pi}^{\ast},\bm{q}^{\ast}}(t+\varepsilon,X_{t+\varepsilon})\right)\right],\\
&h^{\bm{\pi}^{\ast},\bm{q}^{\ast}}(t+\varepsilon,X_{t+\varepsilon},s,y)=\mathbb{E}^{\mathbb{Q}^{\ast}}_{t+\varepsilon,X_{t+\varepsilon}}\left[\int^{T}_{t+\varepsilon}C(s,y,\tau,X_{\tau};\textbf{q})\mathrm{d}\tau+F(s,y,W_{T})\right],\\
&g^{\bm{\pi}^{\ast},\bm{q}^{\ast}}(t+\varepsilon,X_{t+\varepsilon})=\mathbb{E}^{\mathbb{Q}^{\ast}}_{t+\varepsilon,X_{t+\varepsilon}}\left[W_{T}\right],\;
k^{\bm{\pi}^{\ast},\bm{q}^{\ast}}(t+\varepsilon,X_{t+\varepsilon})=\mathbb{E}^{\mathbb{Q}^{\ast}}_{t+\varepsilon,X_{t+\varepsilon}}\left[(W_{T})^{2}\right].
\end{align*}

Define
\begin{align}\label{eq14}
  \mathcal{A^{\bm{\pi},\textbf{q}}}=&\frac{\partial}{\partial t}+\sum\limits_{i=1}^{1+N}\Gamma_{i}(t,x,\bm{\pi},\textbf{q})\frac{\partial}{\partial x^{i}}
  +\frac{1}{2}\sum\limits_{i,j=1}^{1+N}\Omega_{i,j}(t,x,\bm{\pi},\textbf{q})\frac{\partial^{2}}{\partial x^{i}\partial x^{j}}\nonumber\\
  =&\frac{\partial}{\partial t}+\Gamma(t,x,\bm{\pi},\textbf{q})\frac{\partial}{\partial x}+\frac{1}{2}\Omega(t,x,\bm{\pi},\textbf{q})\cdot\frac{\partial^{2}}{\partial x\partial x'},
\end{align}
where $x^{i}$ is the $i$th element of $x$. In the sequel, we omit the superscripts of $h^{\bm{\pi},\textbf{q}}$, $g^{\bm{\pi},\textbf{q}}$ and $k^{\bm{\pi},\textbf{q}}$ for convenience. Let $w=W_{t}$ and
\begin{align}\label{eq15}
\mathcal{L}(t,x,\bm{\pi},\textbf{q},h,g,k)=&\left.\mathcal{A^{\bm{\pi},\textbf{q}}}h(t,x,s,y)\right|_{s=t,y=x}-\mathcal{A^{\bm{\pi},\textbf{q}}}h(t,x,t,x)
+\frac{\partial G(t,x,y)}{\partial y}\bigg|_{y=g(t,x)}\mathcal{A^{\bm{\pi},\textbf{q}}}g(t,x)\nonumber\\
&-\mathcal{A^{\bm{\pi},\textbf{q}}}G(t,x,g(t,x))+\frac{\partial K(t,x,y,z)}{\partial y}\bigg|_{y=g(t,x),z=k(t,x)}\mathcal{A^{\bm{\pi},\textbf{q}}}g(t,x)\nonumber\\
&+\frac{\partial K(t,x,y,z)}{\partial z}\bigg|_{y=g(t,x),z=k(t,x)}\mathcal{A^{\bm{\pi},\textbf{q}}}k(t,x)-\mathcal{A^{\bm{\pi},\textbf{q}}}K(t,x,g(t,x),k(t,x)).
\end{align}
Similar to the investigations in \cite{Pun2}, we can use the EDPE for the value function $V$, Definition \ref{definition3.2} and equations $\eqref{eq14}$ and $\eqref{eq15}$ to give the following characterizations of the HJBI system.
\begin{defn}\label{definition3.3}
The system of extended HJBI equations for $V$, $h$, $g$ and $k$ are defined as follows:
\begin{itemize}
\item[(i)] The function $V(t,x)$ satisfies the HJBI equation
\begin{equation}\label{eq16}
  \sup\limits_{\bm{\pi}\in \Pi}\inf\limits_{\textbf{q}\in\bm{Q}}\left(\mathcal{A^{\bm{\pi},\textbf{q}}}V(t,x)+C(t,x,t,x;\textbf{q})
  +\mathcal{L}(t,x,\bm{\pi},\textbf{q},h,g,k)\right)=0,\quad 0\leq t <T
\end{equation}
with the boundary condition $V(T,x)=F(T,x,w)+G(T,x,w)+K(T,x,w,w^{2})$, where $\mathcal{A^{\bm{\pi},\textbf{q}}}$ and $\mathcal{L}$ are given by $\eqref{eq14}$ and $\eqref{eq15}$, respectively.
\item[(ii)] For any $(s,y)\in[0,T]\times R^{1+N}$, the function $h(t,x,s,y)$ solves the partial differential equation (PDE)
\begin{equation}\label{eq17}
\mathcal{A^{\bm{\pi}^{*},\textbf{q}^{*}}}h(t,x,s,y)+C(s,y,t,x;\textbf{q}^{*})=0, \quad 0\leq t <T
\end{equation}
with $h(T,x,s,y)=F(s,y,w)$, where $(\bm{\pi}^{*}, \textbf{q}^{*})$ realizes the supremum and infimum in equation $\eqref{eq16}$.
\item[(iii)] The function $g(t,x)$ solves the PDE
\begin{equation}\label{eq18}
  \mathcal{A^{\bm{\pi}^{*},\textbf{q}^{*}}}g(t,x)=0, \quad 0\leq t <T
\end{equation}
with $g(T,x)=w$.
\item[(iv)] The function $k(t,x)$ solves the PDE
\begin{equation}\label{eq19}
   \mathcal{A^{\bm{\pi}^{*},\textbf{q}^{*}}}k(t,x)=0, \quad 0\leq t <T
\end{equation}
with $k(T,x)=w^{2}$.
\end{itemize}
\end {defn}

Note that the HJBI equations $\eqref{eq16}$-$\eqref{eq19}$ need to be solved simultaneously. Next, we have an verification theorem to show that the solution to the HJBI system equals to the equilibrium value function and the corresponding investment strategy is the robust equilibrium investment strategy.
\begin{thm}\label{thm3.1}
Assume that $V(t,x),h(t,x,s,y),g(t,x)$ and $k(t,x)$ solve equations $\eqref{eq16}$-$\eqref{eq19}$ in Definition \ref{definition3.3}, $V,g,k\in\mathbb{C}^{1,2}([0,T]\times R^{1+N})$, $h\in\mathbb{C}^{1,2,1,2}([0,T]\times R^{1+N}\times [0,T]\times R^{1+N})$,
$G\in\mathbb{C}^{1,2,2}([0,T]\times R^{1+N}\times R)$ and $K\in\mathbb{C}^{1,2,2,2}([0,T]\times R^{1+N}\times R\times R)$.
Moreover, we suppose that there exists an admissible control-measure strategy $(\bm{\pi}^{*}, \textbf{q}^{*})$, which realizes the supremum and infimum in the equation $\eqref{eq16}$
\begin{align*}
&\mathcal{A^{\bm{\pi}^{*},\textbf{q}^{*}}}V(t,x)+C(t,x,t,x;\textbf{q}^{*})
  +\mathcal{L}(t,x,\bm{\pi}^{*},\textbf{q}^{*},h,g,k)\\
=&\sup\limits_{\bm{\pi}\in \Pi}\left\{\mathcal{A^{\bm{\pi},\textbf{q}^{*}(\bm{\pi})}}V(t,x)+C(t,x,t,x;\textbf{q}^{*}(\bm{\pi}))
  +\mathcal{L}(t,x,\bm{\pi},\textbf{q}^{*}(\bm{\pi}),h,g,k)\right\}\\
=&\sup\limits_{\bm{\pi}\in \Pi}\inf\limits_{\textbf{q}\in\bm{Q}}\left\{\mathcal{A^{\bm{\pi},\textbf{q}}}V(t,x)+C(t,x,t,x;\textbf{q})
  +\mathcal{L}(t,x,\bm{\pi},\textbf{q},h,g,k)\right\}
\end{align*}
for any $(t,x)$. Then $(\bm{\pi}^{*}, \textbf{q}^{*})$ is the equilibrium control-measure strategy, $\bm{\pi}^{*}$ is the robust equilibrium investment strategy and $V$ is the corresponding equilibrium value function as described in Definition \ref{definition3.2}.
\end{thm}
\begin{proof}
See Appendix B.
\end{proof}

\begin{remark}\label{remark2}
Instead of following \cite{Bjork2009, Pun2018} to assume that
$$\frac{\partial \phi}{\partial t},\;\Gamma(t,x,\bm{\pi},\textbf{q})\frac{\partial \phi}{\partial x},\;\Omega(t,x,\bm{\pi},\textbf{q})\cdot\frac{\partial^{2} \phi}{\partial x\partial x'}\in\mathbb{L}_{\mathcal{F}}^{1}(0,T;R,\mathbb{Q})\quad and \quad\Psi(t,x,\bm{\pi},\textbf{q})'\frac{\partial \phi}{\partial x}\in\mathbb{L}_{\mathcal{F}}^{2}(0,T;R^{1+N},\mathbb{Q})$$
for $\phi(t,x)\in\{V(t,x),h(t,x,\cdot,\cdot),g(t,x),k(t,x),G(t,x,g(t,x)),K(t,x,g(t,x),k(t,x))\}$ and any admissible control-measure strategy $(\bm{\pi},\bm{q})$ as well as the corresponding measure $\mathbb{Q}$, we give the assumption of the smoothness of these functions. This is because the above integrability conditions are used to make the stochastic integral valid when computing the expectation of these functions, the result can be also obtained by the Dynkin's Theorem under the condition of the smoothness of these functions. Similar studies with assumptions of smoothness can be also found in \cite{LiY2013Optimal, Wang2021, Zeng2016} and so on.
\end{remark}

\section{Robust equilibrium investment strategy in semi-closed form under a special case}
Although Theorem \ref{thm3.1} presented in Section 3 provide us a method to determine the robust equilibrium investment strategy, it is difficult to obtain the (semi-)analytical solutions of model $\eqref{eq10}$ due to its complexity.  On the other hand, researchers and practitioners prefer the (semi-)analytical solutions because they are more applicable in practice.

In order to derive the (semi-)analytical solutions of model $\eqref{eq10}$, we consider a special investment problem over a time horizon $[0, T]$ under state-dependent risk aversion and skewness preference in this section. More specially, the objective function in $\eqref{eq10}$ is specified as
\begin{equation}\label{eq20}
  \sup\limits_{\bm{\pi}\in \Pi}\inf\limits_{\textbf{q}\in\bm{Q}}\; \left\{\mathbb{E}^{\mathbb{Q}}_{t,x}\left[W_{T}\right]-\frac{\gamma(W_{t})}{2}Var^{\mathbb{Q}}_{t,x}\left[W_{T}\right]
  +\frac{\phi(W_{t})}{3}Skew^{\mathbb{Q}}_{t,x}\left[W_{T}\right]+\mathbb{E}^{\mathbb{Q}}_{t,x}
\left[\int^{T}_{t}\frac{\textbf{q}'_{s}\Xi(x)^{-1}\textbf{q}_{s}}{2\Phi(s,X_{s}) }\mathrm{d}s\right]\right\}
\end{equation}
with $\gamma(W_{t})=\frac{\gamma_{0}}{W_{t}}$ and $\phi(W_{t})=\frac{\phi_{0}}{W_{t}^{2}}$, where $\gamma_{0}$ and $\phi_{0}$ are positive constants representing the risk aversion coefficient and the skewness preference parameter, respectively. First, compared with $\gamma(W_{t})$, the skewness preference $\phi(W_{t})$ gradually dominates the risk aversion when the investor's wealth decreases, which is in accordance with the existing literature that the heaviest lottery players are poor \cite{Kumar2009}. Second, these expressions for $\gamma(W_{t})$ and $\phi(W_{t})$ imply a linear equilibrium investment strategy as shown later, which reflects the amount to invest in the risky asset is proportional to the current wealth value. If for some period there exists a negative $W_{t}$, the investor goes bankrupt and investment is allowed to continue. In fact, when $W_{t}$ takes negative values, the maximization operation of outer layer in the equation \eqref{eq20} becomes  unreasonable since the investor by no means wants to maximizing the variance of the portfolio. In order to be consistent with the previous assumptions in Section 3, $W_{t}>0$ is now assumed, which will be verified later for the robust equilibrium investment strategy.  Assume that  $r(t,Z_{t})\equiv r(t)$, $\mu_{i}(t,Z_{t})=\mu_{i}(t)$ and $\sigma_{i}(t,Z_{t})=\sigma_{i}(t)$ in equations $\eqref{eq1}$ and $\eqref{eq2}$ without regard to the state variable $Z_{t}$ in the financial market. Let $\beta_{t}=\mu_{t}-r_{t} \textbf{1}$, $\Sigma_{t}=\sigma'_{t}\sigma_{t}$ and $\Theta_{t}=\beta'_{t}\Sigma_{t}^{-1}\beta_{t}$. In addition, we assume that $r_{t}$, $\beta_{t}$, $\sigma_{t}$ and $\Theta_{t}$ are deterministic, Lipschitz continuous and bounded functions in $t$ over $[0,T]$.

It follows that the general model $\eqref{eq10}$ taking the form $\eqref{eq11}$ can be degenerated into model $\eqref{eq20}$ by setting $X=W$, which means that there is only one state variable $W$ in this setting. Then, we can use the discussions in Section 3 to obtain the corresponding robust equilibrium investment strategy. The operator $\mathcal{A^{\bm{\pi},\textbf{q}}}$ is specified as
\begin{equation}\label{eq21}
  \mathcal{A^{\bm{\pi},\textbf{q}}}=\frac{\partial}{\partial t}+\Gamma(t,w,\bm{\pi},\textbf{q})\frac{\partial}{\partial w}+\Omega(t,w,\bm{\pi},\textbf{q})\frac{\partial^{2}}{\partial w^{2}}
\end{equation}
with $\Gamma(t,w,\bm{\pi},\textbf{q})=\left[r_{t}w+u_{t}'\beta_{t}+u_{t}'\sigma_{t}q^{S}_{t}\right]$ and $\Omega(t,w,\bm{\pi},\textbf{q})=\frac{1}{2}u_{t}'\Sigma_{t}u_{t}$. Note that functions $F$, $G$ and $H$ in the equation $\eqref{eq11}$ under our setting are independent of $t$ and can be specified as
\begin{equation}\label{eq22}
  F(w,y)=y-\frac{\gamma(w)}{2}y^{2}+\frac{\phi(w)}{3}y^{3},\quad G(w,g)=\frac{\gamma(w)}{2}g^{2}+\frac{2\phi(w)}{3}g^{3},\quad K(w,g,h)=-\phi(w)gh,
\end{equation}
respectively. Functions $h$ is specified as $h(t,w,y)$. Note that $C(t,w,s,W_{s};\textbf{q})=\frac{\textbf{q}'_{s}\Xi(w)^{-1}\textbf{q}_{s}}{2\Phi(s,W_{s}) }$. For simplicity, we assume that $\Xi(w)=\xi I_{M\times M}$ with $\xi>0$.

For the above special case, we rewrite the descriptions of the admissible control-measure strategy, which are more specific and easier to verify than those in Definition \ref{definition3.1}.
\begin{defn}\label{definition4.1}
A control-measure strategy $(\bm{\pi},\bm{q})$ is admissible if the following conditions are satisfied:
\begin{itemize}
\item[(i)] for any $t\in[0,T]$, $\bm{\pi}=\{u_{t}\}$ is $\mathcal{F}_{t}$-progressively measurable;
\item[(ii)] $\bm{q}$ satisfies $\mathbb{E}^{\mathbb{P}}\left[\exp\left(8\int^{T}_{0}\|q^{S}_{s}\|^{2}\mathrm{d}s\right)\right]<\infty$;
\item[(iii)] for any initial point $(t, w)\in[0,T]\times R$ and $\bm{q}$, the following SDE
$$
\mathrm{d}W_{s}=\left[r_{s}W_{s}+u_{s}'\beta_{s}+u_{s}'\sigma_{s}q^{S}_{s}\right]\mathrm{d}s+u_{s}'\sigma_{s}\mathrm{d}W^{S,\mathbb{Q}}_{s}
$$
for $\{W_{s}\}_{s\in[t,T]}$ with $W_{t}=w$ admits a unique strong solution;
\item[(iv)] for each $\bm{q}$ satisfying $(ii)$ and the corresponding measure $\mathbb{Q}$, one has $\mathbb{E}^{\mathbb{Q}}\left[\left(\int^{T}_{0}\|u_{s}\|^{2}\mathrm{d}s\right)^{4}\right]<\infty$;
\item[(v)] for each $\bm{q}$ satisfying $(ii)$ and the corresponding measure $\mathbb{Q}$, one has
$\mathbb{E}^{\mathbb{Q}}\left[\sup\limits_{t\in[0,T]}\Phi^{-2}(t,W_{t}) \right]<\infty.$
\end{itemize}
\end{defn}

We denote by $\Pi_{1} \times \bm{Q}_{1}$ the set of admissible control-measure strategies as described in Definition \ref{definition4.1}. On the basis of Definition \ref{definition4.1}, we have the following lemma, which can be used to show the well-posedness of our problem $\eqref{eq20}$ later.
\begin{lemma}\label{lemma4.1}
With any admissible control-measure strategy $(\bm{\pi},\bm{q})\in\Pi_{1} \times \bm{Q}_{1}$ and the corresponding measure $\mathbb{Q}$, the wealth process $W_{t}$ satisfies $\mathbb{E}^{\mathbb{Q}}\left[\sup\limits_{t\in[0,T]}\mid W_{t}\mid^{4}\right]<\infty$.
\end{lemma}
\begin{proof}
See Appendix C.
\end{proof}

Now, on the basis of Lemma \ref{lemma4.1}, we compare the conditions between Definition \ref{definition3.1} and Definition \ref{definition4.1}.
\begin{remark}\label{remark3}
It is easy to see that the condition $(ii)$ in Definition \ref{definition4.1} implies the condition $(ii)$ in Definition \ref{definition3.1} and the SDE in the condition $(iii)$ in Definition \ref{definition4.1} is a specific form of that in Definition \ref{definition3.1} under our special case. It follows from Lemma \ref{lemma4.1} that
$\mathbb{E}^{\mathbb{Q}}\left[\sup\limits_{t\in[0,T]}\mid W_{t}\mid^{4}\right]<\infty$.
Thus, by the condition $(v)$ in Definition \ref{definition4.1} and the proof of Lemma \ref{lemma4.1}, one has
\begin{align*}
\mathbb{E}^{\mathbb{Q}}_{t,w}\left[\int^{T}_{t}\mid C(w,s,W_{s};\textbf{q})\mid \mathrm{d}s\right]\leq  L\left(\mathbb{E}^{\mathbb{Q}}_{t,w}\left[\sup\limits_{t\in[t,T]}\frac{1}{\Phi^{2}(s,W_{s}) }\right]\right)^{\frac{1}{2}}
\left(\mathbb{E}^{\mathbb{Q}}_{t,w}\left[\left(\int^{T}_{t}\|q^{S}_{s}\|^{2}\mathrm{d}s\right)^{2}\right]\right)^{\frac{1}{2}}<\infty,
\end{align*}
where $L$ is a positive constant. If we assume $w\neq0$, then we have
\begin{align*}
\mathbb{E}^{\mathbb{Q}}_{t,w}\left[\mid F(w,W_{T})\mid\right]\leq &\mathbb{E}^{\mathbb{Q}}_{t,w}\left[\mid W_{T} \mid + \mid \frac{\gamma(w)}{2}  (W_{T})^{2}\mid
+\mid \frac{\phi(w)}{3} (W_{T})^{3}\mid\right] <\infty.
\end{align*}
Similarly, we can show that
$$
G\left(w,\mathbb{E}^{\mathbb{Q}}_{t,w}\left[W_{T}\right]\right)
  +K\left(w,\mathbb{E}^{\mathbb{Q}}_{t,w}\left[W_{T}\right],\mathbb{E}^{\mathbb{Q}}_{t,w}\left[(W_{T})^{2}\right]\right)<\infty.
$$
On the whole, with the additional assumption of $w\neq0$, the conditions of Definition \ref{definition4.1} imply the conditions of Definition \ref{definition3.1}. By the way, it shows the well-posedness of our problem $\eqref{eq20}$. We would like to point out that the condition of $w\neq0$ will be verified for our derived equilibrium control-measure strategy later.
\end{remark}

In order to solve the HJBI equation $\eqref{eq16}$, we can first reduce $\eqref{eq16}$ into a highly nonlinear HJB equation. Since $V(t,w)=h(t,w,w)+G(w,g(t,w))+K(w,g(t,w),k(t,w))$, one has
$$
\mathcal{A^{\bm{\pi},\textbf{q}}}V(t,w)+C(w,t,w;\textbf{q})
  +\mathcal{L}(t,w,\bm{\pi},\textbf{q},h,g,k)=\mathcal{G}(t,w,\bm{\pi},\textbf{q},h,g,k)+C(w,t,w;\textbf{q}),
$$
where
\begin{equation*}
\begin{aligned}
\mathcal{G}(t,w,\bm{\pi},\textbf{q},h,g,k)=&\left(\mathcal{A^{\bm{\pi},\textbf{q}}}h(t,w,y)\right)|_{y=w}
+\frac{\partial G(w,y)}{\partial y}\bigg|_{y=g(t,w)}\mathcal{A^{\bm{\pi},\textbf{q}}}g(t,w)\\
&+\frac{\partial K(w,y,z)}{\partial y}\bigg|_{y=g(t,w),z=k(t,w)}\mathcal{A^{\bm{\pi},\textbf{q}}}g(t,w)+\frac{\partial K(w,y,z)}{\partial z}\bigg|_{y=g(t,w),z=k(t,w)}\mathcal{A^{\bm{\pi},\textbf{q}}}k(t,w).
\end{aligned}
\end{equation*}
We can take the minimum of the quadratic form of $\textbf{q}$ in the equation $\eqref{eq16}$ to derive that
\begin{align}\label{eq23}
   {q^{S}_{t}}^{\ast}=-\Phi(t,w)\xi \sigma_{t}'u_{t}\delta_{1},
\end{align}
where
\begin{equation}\label{eq24}
\begin{aligned}
\delta_{1}=&\frac{\partial h(t,w,y)}{\partial w}\bigg|_{y=w}+\frac{\partial G(w,y)}{\partial y}\bigg|_{y=g(t,w)}\frac{\partial g(t,w)}{\partial w}\\
&+\frac{\partial K(w,y,z)}{\partial y}\bigg|_{y=g(t,w),z=k(t,w)}\frac{\partial g(t,w)}{\partial w}+\frac{\partial K(w,y,z)}{\partial z}\bigg|_{y=g(t,w),z=k(t,w)}\frac{\partial k(t,w)}{\partial w}\\
=&\frac{\partial h(t,w,y)}{\partial w}\bigg|_{y=w}+\left(\gamma(w)g(t,w)+2\phi(w)g^{2}(t,w)\right)\frac{\partial g(t,w)}{\partial w}
-\phi(w)k(t,w)\frac{\partial g(t,w)}{\partial w}-\phi(w)g(t,w)\frac{\partial k(t,w)}{\partial w}.
\end{aligned}
\end{equation}
Thus, we change $\eqref{eq16}$ into the following highly nonlinear HJB equation
\begin{equation}\label{eq25}
  \sup\limits_{\bm{\pi}\in \Pi}\left(\mathcal{G}(t,w,\bm{\pi},\textbf{q}^{\ast},h,g,k)+C(w,t,w;\textbf{q}^{\ast})\right)=0,\quad 0\leq t <T
\end{equation}
with $\textbf{q}^{\ast}=\textbf{q}^{\ast}(\bm{\pi})={q^{S}_{t}}^{\ast}(t,w,u_{t})$. This means that the system of extended HJBI equations $\eqref{eq16}$-$\eqref{eq19}$ is changed into a system of extended HJB equations $\eqref{eq25}$, $\eqref{eq17}$-$\eqref{eq19}$.

Since the risk aversion and skewness preference coefficients are state-dependent in our setting, the function $h(t,w,y)$ is dependent on $y$. Let
\begin{equation}\label{eq26}
\begin{aligned}
\delta_{2}:=&\frac{\partial^{2} h(t,w,y)}{\partial w^{2}}\bigg|_{y=w}+\frac{\partial G(w,y)}{\partial y}\bigg|_{y=g(t,w)}\frac{\partial^{2} g(t,w)}{\partial w^{2}}\\
&+\frac{\partial K(w,y,z)}{\partial y}\bigg|_{y=g(t,w),z=k(t,w)}\frac{\partial^{2} g(t,w)}{\partial w^{2}}+\frac{\partial K(w,y,z)}{\partial z}\bigg|_{y=g(t,w),z=k(t,w)}\frac{\partial^{2} k(t,w)}{\partial w^{2}}\\
=&\frac{\partial^{2} h(t,w,y)}{\partial w^{2}}\bigg|_{y=w}+\left(\gamma(w)g(t,w)+2\phi(w)g^{2}(t,w)\right)\frac{\partial^{2} g(t,w)}{\partial w^{2}}\\
&-\phi(w)k(t,w)\frac{\partial^{2} g(t,w)}{\partial w^{2}}-\phi(w)g(t,w)\frac{\partial^{2} k(t,w)}{\partial w^{2}}.
\end{aligned}
\end{equation}
Following the investigations in \cite{Pun2018}, we adopt the ambiguity preference function as $\Phi(t,w)=-\delta_{2}/\delta_{1}^{2}$. Obviously, this specified ambiguity preference function is state-dependent and thus remains the property of homothetic robust \cite{Maenhout2004}. In other words, it is economically meaningful. On the other hand, such an ambiguity preference function facilitates analytical tractability. Therefore, ${q^{S}_{t}}^{\ast}$ in the equation $\eqref{eq23}$ can be rewritten as
${q^{S}_{t}}^{\ast}=\xi \frac{\delta_{2}}{\delta_{1}}\sigma_{t}'u_{t}$
and the HJB equation $\eqref{eq25}$ becomes
\begin{equation*}
  0=\sup\limits_{\bm{\pi}\in \Pi}\left\{\cdot\cdot\cdot+u_{t}'\beta_{t}\delta_{1}+\frac{\xi+1}{2}u_{t}'\Sigma_{t}u_{t}\delta_{2}+\cdot\cdot\cdot\right\},
\end{equation*}
where the terms without $u_{t}$ are suppressed. Now, we take the maximum of the above quadratic form of $u_{t}$ and then obtain the robust equilibrium investment strategy as
$$
u_{t}^{\ast}=-\frac{1}{\xi+1}\Sigma_{t}^{-1}\beta_{t}\frac{\delta_{1}}{\delta_{2}}.
$$

Considering terminal conditions of $h$, $g$ and $k$, i.e., $h(T,w,y)=w-\frac{\gamma_{0}}{2y}w^{2}+\frac{\phi_{0}}{3y^{2}}w^{3}$, $g(T,w)=w$ and $k(T,w)=w^{2}$, we assume the following forms for $h$, $g$ and $k$
\begin{equation}\label{eq27}
 h(t,w,y)=h_{1}(t)w-\frac{\gamma_{0}}{2y}h_{2}(t)w^{2}+\frac{\phi_{0}}{3y^{2}}h_{3}(t)w^{3},\quad g(t,w)=g_{1}(t)w,\quad k(t,w)=k_{1}(t)w^{2}
\end{equation}
with $h_{1}(T)=h_{2}(T)=h_{3}(T)=g_{1}(T)=k_{1}(T)=1$.
Then, one has
\begin{equation}\label{eq28}
   \left\{ \begin{aligned}
   &\frac{\partial h(t,w,y)}{\partial t}=\frac{\partial h_{1}(t)}{\partial t}w-\frac{\gamma_{0}}{2y}\frac{\partial h_{2}(t)}{\partial t}w^{2}+\frac{\phi_{0}}{3y^{2}}\frac{\partial h_{3}(t)}{\partial t}w^{3},\; \frac{\partial g(t,w)}{\partial t}=\frac{\partial g_{1}(t)}{\partial t}w,\\
   &\frac{\partial k(t,w)}{\partial t}=\frac{\partial k_{1}(t)}{\partial t}w^{2}, \quad \frac{\partial g(t,w)}{\partial w}=g_{1}(t), \quad \frac{\partial k(t,w)}{\partial w}=2k_{1}(t) w, \quad \frac{\partial^{2} k(t,w)}{\partial w^{2}}=2k_{1}(t),\\
   &\frac{\partial h(t,w,y)}{\partial w}=h_{1}(t)-\frac{\gamma_{0}}{y}h_{2}(t)w+\frac{\phi_{0}}{y^{2}}h_{3}(t)w^{2},\quad
   \frac{\partial h(t,w,y)}{\partial w}\bigg|_{y=w}=h_{1}(t)-\gamma_{0}h_{2}(t)+\phi_{0}h_{3}(t),\\
   &\frac{\partial^{2} h(t,w,y)}{\partial w^{2}}=-\frac{\gamma_{0}}{y}h_{2}(t)+\frac{2\phi_{0}}{y^{2}}h_{3}(t)w, \quad  \frac{\partial^{2} h(t,w,y)}{\partial w^{2}}\bigg|_{y=w}=-\frac{\gamma_{0}}{w}h_{2}(t)+\frac{2\phi_{0}}{w}h_{3}(t).
  \end{aligned}\right.
\end{equation}
Moreover, $(u_{t}^{\ast},{q^{S}_{t}}^{\ast})$ can be given by
\begin{equation}\label{eq29}
   \left\{ \begin{aligned}
   &u_{t}^{\ast}=\frac{w}{\xi+1}\Sigma_{t}^{-1}\beta_{t}\frac{h_{1}(t)+\gamma_{0}\left(g_{1}^{2}(t)-h_{2}(t)\right)+\phi_{0}
   \left(h_{3}(t)+2g_{1}^{3}(t)-3g_{1}(t)k_{1}(t)\right)}{\gamma_{0}h_{2}(t)+2\phi_{0}(g_{1}(t)k_{1}(t)-h_{3}(t))},\\
   &{q^{S}_{t}}^{\ast}=-\frac{\xi}{\xi+1}\sigma_{t}'\Sigma_{t}^{-1}\beta_{t}.
  \end{aligned}\right.
\end{equation}
Let
$$
f(t)=\frac{h_{1}(t)+\gamma_{0}\left(g_{1}^{2}(t)-h_{2}(t)\right)+\phi_{0}
\left(h_{3}(t)+2g_{1}^{3}(t)-3g_{1}(t)k_{1}(t)\right)}{\gamma_{0}h_{2}(t)+2\phi_{0}(g_{1}(t)k_{1}(t)-h_{3}(t))}.
$$
Substituting equations $\eqref{eq27}$-$\eqref{eq29}$ into equations $\eqref{eq17}$-$\eqref{eq19}$ as well as separating the terms through the order of $w$, $\frac{w^{2}}{y}$ and $\frac{w^{3}}{y^{2}}$, we can derive the following form:
\begin{equation}\label{eq30}
\left\{\begin{aligned}
   \frac{\partial h_{1}(t)}{\partial t}&+\left[r_{t}+\frac{\Theta_{t}f(t)}{(\xi+1)^{2}}\right]h_{1}(t)+\frac{\xi \Theta_{t}f^{2}(t)}{2(\xi+1)^{2}}
   \left[\gamma_{0} h_{2}(t)+2\phi_{0}(g_{1}(t)k_{1}(t)-h_{3}(t))\right]=0,\\
   \frac{\partial h_{2}(t)}{\partial t}&+\left[2r_{t}+\frac{2\Theta_{t}f(t)}{(\xi+1)^{2}}+\frac{\Theta_{t}f^{2}(t)}{(\xi+1)^{2}}\right]h_{2}(t)=0,\\
   \frac{\partial h_{3}(t)}{\partial t}&+3\left[r_{t}+\frac{\Theta_{t}f(t)}{(\xi+1)^{2}}+\frac{\Theta_{t}f^{2}(t)}{(\xi+1)^{2}}\right]h_{3}(t)=0,\\
   \frac{\partial g_{1}(t)}{\partial t}&+\left[r_{t}+\frac{\Theta_{t}f(t)}{(\xi+1)^{2}}\right]g_{1}(t)=0,\\
   \frac{\partial k_{1}(t)}{\partial t}&+\left[2r_{t}+\frac{2\Theta_{t}f(t)}{(\xi+1)^{2}}+\frac{\Theta_{t}f^{2}(t)}{(\xi+1)^{2}}\right]k_{1}(t)=0.
\end{aligned}\right.
\end{equation}
It is easy to see that $h_{2}(t)=k_{1}(t)$ for any $t\in[0,T]$. After further calculations, we can obtain
\begin{equation}\label{eq31}
\left\{\begin{aligned}
   h_{2}(t)=&k_{1}(t)=e^{\int^{T}_{t}\left[2r_{s}+\frac{2\Theta_{s}f(s)}{(\xi+1)^{2}}+\frac{\Theta_{s}f^{2}(s)}{(\xi+1)^{2}}\right]\mathrm{d}s},\\
   h_{3}(t)=&e^{\int^{T}_{t}3\left[r_{s}+\frac{\Theta_{s}f(s)}{(\xi+1)^{2}}+\frac{\Theta_{s}f^{2}(s)}{(\xi+1)^{2}}\right]\mathrm{d}s},\\
   g_{1}(t)=&e^{\int^{T}_{t}\left[r_{s}+\frac{\Theta_{s}f(s)}{(\xi+1)^{2}}\right]\mathrm{d}s},\\
   h_{1}(t)=&g_{1}(t)+\int^{T}_{t}\frac{\xi \Theta_{s}f^{2}(s)}{2(\xi+1)^{2}}\left[\gamma_{0} h_{2}(s)+2\phi_{0}(g_{1}(s)h_{2}(s)-h_{3}(s))\right]
   \times e^{\int^{s}_{t}\left[r_{u}+\frac{\Theta_{u}f(u)}{(\xi+1)^{2}}\right]\mathrm{d}u}\mathrm{d}s.
\end{aligned}\right.
\end{equation}
In addition, the expression of $f(t)$ can be rewritten as an integral equation
\begin{eqnarray}\label{eq32}
f(t)&=&\left[\gamma_{0}e^{\int^{T}_{t}\left[2r_{s}+\frac{2\Theta_{s}f(s)}{(\xi+1)^{2}}+\frac{\Theta_{s}f^{2}(s)}{(\xi+1)^{2}}\right]\mathrm{d}s}+2\phi_{0}
   e^{\int^{T}_{t}\left[3r_{s}+\frac{3\Theta_{s}f(s)}{(\xi+1)^{2}}+\frac{\Theta_{s}f^{2}(s)}{(\xi+1)^{2}}\right]\mathrm{d}s}\left(
   1-e^{\int^{T}_{t}\frac{2\Theta_{s}f^{2}(s)}{(\xi+1)^{2}}\mathrm{d}s}\right)\right]^{-1}\nonumber\\
   &&\times\left\{\int^{T}_{t}\frac{\xi \Theta_{s}f^{2}(s)}{2(\xi+1)^{2}}\left[\gamma_{0}e^{\int^{T}_{s}\left[2r_{u}+\frac{2\Theta_{u}f(u)}{(\xi+1)^{2}}
   +\frac{\Theta_{u}f^{2}(u)}{(\xi+1)^{2}}\right]\mathrm{d}u}
   +2\phi_{0}e^{\int^{T}_{s}\left[3r_{u}+\frac{3\Theta_{u}f(u)}{(\xi+1)^{2}}+\frac{\Theta_{u}f^{2}(u)}{(\xi+1)^{2}}\right]\mathrm{d}u}\right.\right.\nonumber\\
   &&\left.\times\left(
   1-e^{\int^{T}_{s}\frac{2\Theta_{u}f^{2}(u)}{(\xi+1)^{2}}\mathrm{d}u}\right)\right]
   \times e^{\int^{s}_{t}\left[r_{u}+\frac{\Theta_{u}f(u)}{(\xi+1)^{2}}\right]\mathrm{d}u}\mathrm{d}s+e^{\int^{T}_{t}\left[r_{s}+\frac{\Theta_{s}f(s)}{(\xi+1)^{2}}\right]\mathrm{d}s}\nonumber\\
   &&+\phi_{0}e^{\int^{T}_{t}3\left[r_{s}+\frac{\Theta_{s}f(s)}{(\xi+1)^{2}}\right]\mathrm{d}s}\left(e^{\int^{T}_{t}\frac{3\Theta_{s}f^{2}(s)}{(\xi+1)^{2}}\mathrm{d}s}+2-3
   e^{\int^{T}_{t}\frac{\Theta_{s}f^{2}(s)}{(\xi+1)^{2}}\mathrm{d}s}\right)\nonumber\\
   &&\left.+\gamma_{0}e^{\int^{T}_{t}2\left[r_{s}+\frac{\Theta_{s}f(s)}{(\xi+1)^{2}}\right]\mathrm{d}s}\left(1-
   e^{\int^{T}_{t}\frac{\Theta_{s}f^{2}(s)}{(\xi+1)^{2}}\mathrm{d}s}\right)\right\}
\end{eqnarray}
with $f(T)=\frac{1}{\gamma_{0}}$.
\begin{thm}\label{thm4.1}
The integral equation $\eqref{eq32}$ admits a unique solution $f\in C^{1}[0,T]$.
\end{thm}
\begin{proof}
See Appendix D.
\end{proof}

On the basis of the above investigations, we are now easy to derive the following main result.
\begin{thm}\label{thm4.2}
The robust mean-variance-skewness portfolio selection problem $\eqref{eq20}$ with $\Phi(t,w)=-\delta_{2}/\delta_{1}^{2}$ admits the robust value function
\begin{equation}\label{eq35}
  V(t,w)=\left[h_{1}(t)-\frac{\gamma_{0}}{2}\left(h_{2}(t)-g_{1}^{2}(t)\right)+\frac{\phi_{0}}{3}\left(2g_{1}^{3}(t)-3g_{1}(t)h_{2}(t)+h_{3}(t)\right)\right]w
\end{equation}
and the equilibrium control-measure strategy is
\begin{equation}\label{eq36}
   \left\{ \begin{aligned}
   &u_{t}^{\ast}=\frac{w}{\xi+1}\Sigma_{t}^{-1}\beta_{t}\frac{h_{1}(t)+\gamma_{0}\left(g_{1}^{2}(t)-h_{2}(t)\right)+\phi_{0}
   \left(h_{3}(t)+2g_{1}^{3}(t)-3g_{1}(t)h_{2}(t)\right)}{\gamma_{0}h_{2}(t)+2\phi_{0}(g_{1}(t)h_{2}(t)-h_{3}(t))},\\
   &{q^{S}_{t}}^{\ast}=-\frac{\xi}{\xi+1}\sigma_{t}'\Sigma_{t}^{-1}\beta_{t},
  \end{aligned}\right.
\end{equation}
where $\delta_{1}$, $\delta_{2}$, $h_{1}(t)$, $h_{2}(t)$, $h_{3}(t)$ and $g_{1}(t)$ are given by equations $\eqref{eq24}$, $\eqref{eq26}$ and $\eqref{eq31}$.
\end{thm}
\begin{proof}
See Appendix E.
\end{proof}

\begin{remark}\label{remark4}
Clearly, $\eqref{eq36}$ implies the following facts: (i) the robust equilibrium investment strategy depends on the ambiguity averse parameter $\xi$ and the risk aversion coefficient $\gamma_{0}$ as well as the skewness preference parameter $\phi_{0}$; (ii) the robust equilibrium investment strategy is proportional to the current wealth value; (iii) the proportion $\frac{u_{t}^{\ast}}{w}$, invested in the risky assets, is same to the poor and the rich, which means it is independent of the initial wealth value.
\end{remark}

\begin{remark}\label{remark5}
Note that $\Phi$ is assumed to be an $R^{+}$-valued function which measures the degree of the investor's confidence in the reference model with the form $\Phi(t,w)=-\delta_{2}/\delta_{1}^{2}$. It follows from the equations $\eqref{eq26}$ and $\eqref{eq28}$ and the proof of Theorem \ref{thm4.2} that $$\delta_{2}=-\frac{1}{w}\left[\gamma_{0}h_{2}(t)+2\phi_{0}(g_{1}(t)k_{1}(t)-h_{3}(t))\right]$$ and $w>0$. Thus, we only need to verify that $\delta_{3}:=\gamma_{0}h_{2}(t)+2\phi_{0}(g_{1}(t)k_{1}(t)-h_{3}(t))>0$. Since it is difficult to verify this directly, similar to the work \cite{Wang2021}, we will illustrate $\delta_{3}>0$ by numerical experiments in Section 5.
\end{remark}

At the end of this section, we give some special cases of Theorem \ref{thm4.2}.

\begin{cor}\label{cor1}
 If the investor is ambiguity-neutral, i.e., the ambiguity-averse parameter $\xi\downarrow0$, then $D_{t,x}(\mathbb{Q}\parallel\mathbb{P})=0$. Therefore, the equilibrium value function and the equilibrium investment strategy are given by
\begin{equation}\label{eq37}
  \widetilde{V}(t,w)=\left[\widetilde{h}_{1}(t)-\frac{\gamma_{0}}{2}\left(\widetilde{h}_{2}(t)-\widetilde{h}_{1}^{2}(t)\right)+\frac{\phi_{0}}{3}
  \left(2\widetilde{h}_{1}^{3}(t)-3\widetilde{h}_{1}(t)\widetilde{h}_{2}(t)+\widetilde{h}_{3}(t)\right)\right]w
\end{equation}
and
\begin{equation}\label{eq38}
   \widetilde{u}_{t}^{\ast}=\Sigma_{t}^{-1}\beta_{t}\widetilde{f}(t)w,\\
\end{equation}
respectively, where
\begin{equation}\label{eq39}
\left\{\begin{aligned}
   \widetilde{h}_{1}(t)=&e^{\int^{T}_{t}\left[r_{s}+\Theta_{s}\widetilde{f}(s)\right]\mathrm{d}s},\\
   \widetilde{h}_{2}(t)=&e^{\int^{T}_{t}\left[2r_{s}+2\Theta_{s}\widetilde{f}(s)+\Theta_{s}\widetilde{f}^{2}(s)\right]\mathrm{d}s},\\
   \widetilde{h}_{3}(t)=&e^{\int^{T}_{t}3\left[r_{s}+\Theta_{s}\widetilde{f}(s)+\Theta_{s}\widetilde{f}^{2}(s)\right]\mathrm{d}s}\\
\end{aligned}\right.
\end{equation}
and $$\widetilde{f}(t)=\frac{\widetilde{h}_{1}(t)+\gamma_{0}\left(\widetilde{h}_{1}^{2}(t)-\widetilde{h}_{2}(t)\right)+\phi_{0}
\left(\widetilde{h}_{3}(t)+2\widetilde{h}_{1}^{3}(t)-3\widetilde{h}_{1}(t)\widetilde{h}_{2}(t)\right)}{\gamma_{0}\widetilde{h}_{2}(t)
+2\phi_{0}(\widetilde{h}_{1}(t)\widetilde{h}_{2}(t)-\widetilde{h}_{3}(t))}.$$
\end{cor}

\begin{remark}
If there exists only one risky asset in the financial market, then Corollary \ref{cor1} is reduced to Theorem 3 in \cite{Mu2019}.
\end{remark}

\begin{cor}\label{cor2} \cite[Proposition 7]{Pun2018}
 If we consider the portfolio selection problem for the ambiguity averse investor in which there is no skewness preference in the objective function $\eqref{eq20}$, i.e., the skewness preference coefficient $\phi_{0}$ equals to 0, then we have
$\hat{f}(t)=\frac{\hat{h}_{1}(t)+\gamma_{0}\left(\hat{g}_{1}^{2}(t)-\hat{h}_{2}(t)\right)}{\gamma_{0}\hat{h}_{2}(t)}$, where $\hat{h}_{1}(t)$, $\hat{h}_{2}(t)$ and $\hat{g}_{1}(t)$ satisfy
\begin{equation}\label{eq40}
\left\{\begin{aligned}
   \hat{h}_{2}(t)=&e^{\int^{T}_{t}\left[2r_{s}+\frac{2\Theta_{s}\hat{f}(s)}{(\xi+1)^{2}}+\frac{\Theta_{s}\hat{f}^{2}(s)}{(\xi+1)^{2}}\right]\mathrm{d}s},\\
   \hat{g}_{1}(t)=&e^{\int^{T}_{t}\left[r_{s}+\frac{\Theta_{s}\hat{f}(s)}{(\xi+1)^{2}}\right]\mathrm{d}s},\\
   \hat{h}_{1}(t)=&\hat{g}_{1}(t)+\int^{T}_{t}\frac{\xi \Theta_{s}\hat{f}^{2}(s)}{2(\xi+1)^{2}}\gamma_{0} \hat{h}_{2}(s)
   \times e^{\int^{s}_{t}\left[r_{u}+\frac{\Theta_{u}\hat{f}(u)}{(\xi+1)^{2}}\right]\mathrm{d}u}\mathrm{d}s.
\end{aligned}\right.
\end{equation}
Moreover, the equilibrium value function and the equilibrium control-measure strategy are given by
\begin{equation}\label{eq41}
  \hat{V}(t,w)=\left[\hat{h}_{1}(t)-\frac{\gamma_{0}}{2}\left(\hat{h}_{2}(t)-\hat{g}_{1}^{2}(t)\right)\right]w
\end{equation}
and
\begin{equation}\label{eq42}
   \left\{ \begin{aligned}
   &\hat{u}_{t}^{\ast}=\frac{w}{\xi+1}\Sigma_{t}^{-1}\beta_{t}\hat{f}(t),\\
   &{q^{S}_{t}}^{\ast}=-\frac{\xi}{\xi+1}\sigma_{t}'\Sigma_{t}^{-1}\beta_{t},
  \end{aligned}\right.
\end{equation}
respectively.
\end{cor}

\begin{cor}\label{cor3}
 If the investor is ambiguity-neutral and does not have skewness preference, which means that $\xi\downarrow0$ and $\phi_{0}=0$, then the equilibrium value function and the equilibrium investment strategy reduce to
\begin{equation}\label{eq43}
  \bar{V}(t,w)=\left[\bar{h}_{1}(t)-\frac{\gamma_{0}}{2}\left(\bar{h}_{2}(t)-\bar{h}_{1}^{2}(t)\right)\right]w
\end{equation}
and
\begin{equation}\label{eq44}
   \bar{u}_{t}^{\ast}=\Sigma_{t}^{-1}\beta_{t}\bar{f}(t)w,\\
\end{equation}
respectively, where
\begin{equation}\label{eq45}
\left\{\begin{aligned}
   \bar{h}_{1}(t)=&e^{\int^{T}_{t}\left[r_{s}+\Theta_{s}\bar{f}(s)\right]\mathrm{d}s},\\
   \bar{h}_{2}(t)=&e^{\int^{T}_{t}\left[2r_{s}+2\Theta_{s}\bar{f}(s)+\Theta_{s}\bar{f}^{2}(s)\right]\mathrm{d}s}
\end{aligned}\right.
\end{equation}
and $$\bar{f}(t)=\frac{\bar{h}_{1}(t)+\gamma_{0}\left(\bar{h}_{1}^{2}(t)-\bar{h}_{2}(t)\right)}{\gamma_{0}\bar{h}_{2}(t)}.$$
\end{cor}

\begin{remark}
If there exists only one risky asset in the financial market, then Corollary \ref{cor3} is reduced to Theorem 4.6 in \cite{Bjork2014}.
\end{remark}

\section{Numerical experiments}
In this section, we provide some numerical experiments to reveal the impacts of model parameters on the robust equilibrium investment strategy and utility losses from ignoring skewness preference and model uncertainty. Assume that there exist a risk-free asset and a risky asset in the financial market for simplicity. The basic values of model parameters are mainly referred to \cite{Mu2019, Wang2021} and are given as follows: the investment horizon is $T=5$ (years), the risk-free interest rate $r=0.05$, the risk aversion coefficient $\gamma_{0}=2$, the skewness preference parameter $\phi_{0}=0.50$, the ambiguity averse parameter $\xi=1$,  the initial wealth at time $t=0$ is $w_{0}=4$, $\mu=0.15$ and $\sigma=0.25$. Without loss of generality, we focus on our analysis at time $t=0$ for convenience. Unless otherwise specified, we change the value of one parameter in each figure below and investigate the sensitivity of the robust equilibrium investment strategy and utility losses with respect to (w.r.t.) the variation of that parameter. Moreover, we illustrate the condition in Remark \ref{remark5} can be satisfied for the values of $\delta_{3}$ under different settings. Note that the proportion invested in the risky asset is independent of the initial wealth value as shown in Remark \ref{remark4} (iii). Thus, similar to the studies \cite{Bi2019, Wang2021, Wang2021QF}, we discuss the impacts of the initial wealth value on the robust equilibrium investment strategy (i.e., the amount) instead of the impacts of the initial wealth value on the proportion in the sequel. The implementation is done with MATLAB R2017b.

\subsection{Impacts of model parameters on the robust equilibrium investment strategy}
\begin{figure}[htbp]
	\begin{minipage}[t]{0.5\linewidth}
		\centering
		\includegraphics[width=3.0in]{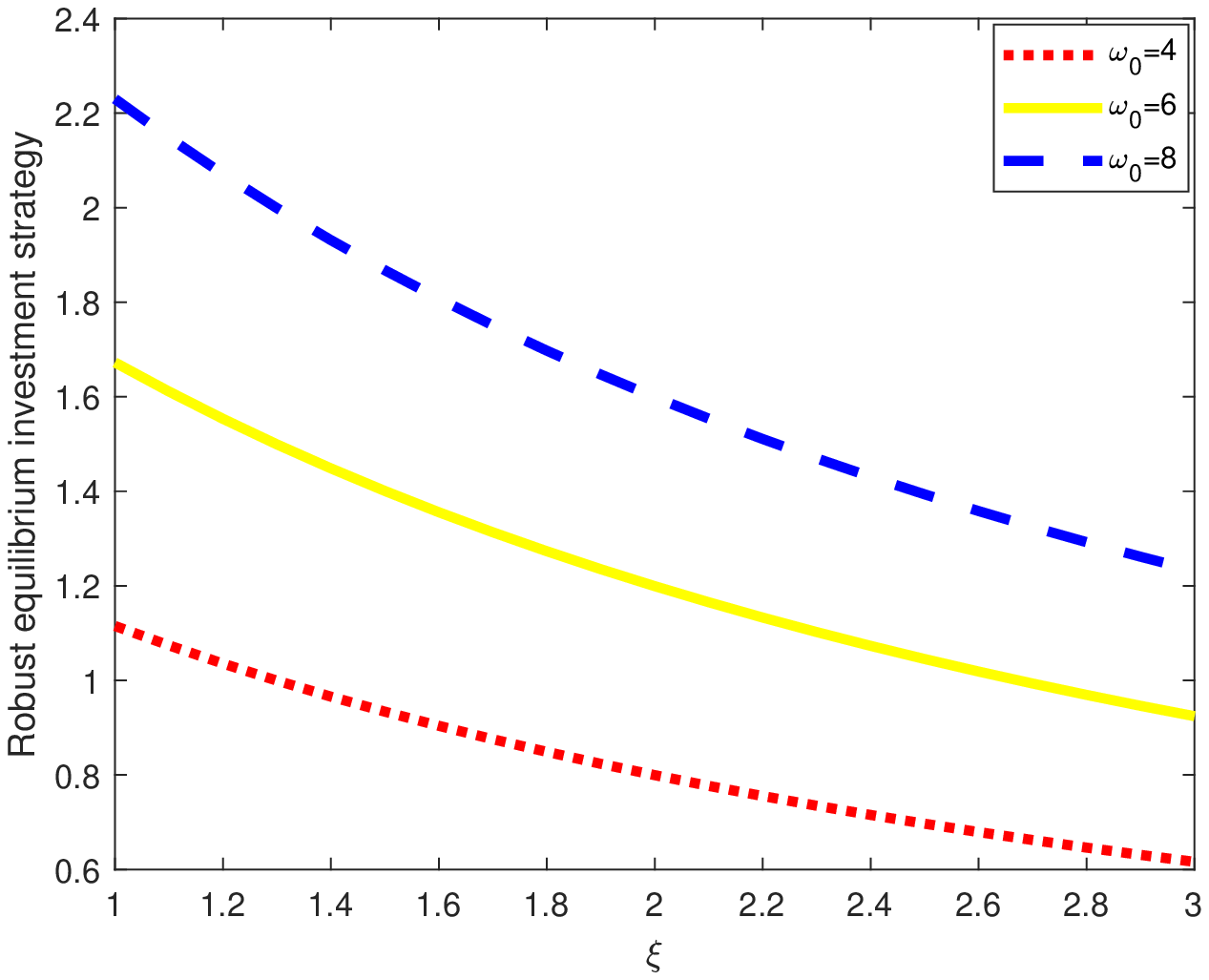}
        \caption*{(a)\;the impacts of $w_0$ and $\xi$ on $u^{\ast}$}
	\end{minipage}%
	\begin{minipage}[t]{0.5\linewidth}
		\centering
		\includegraphics[width=3.0in]{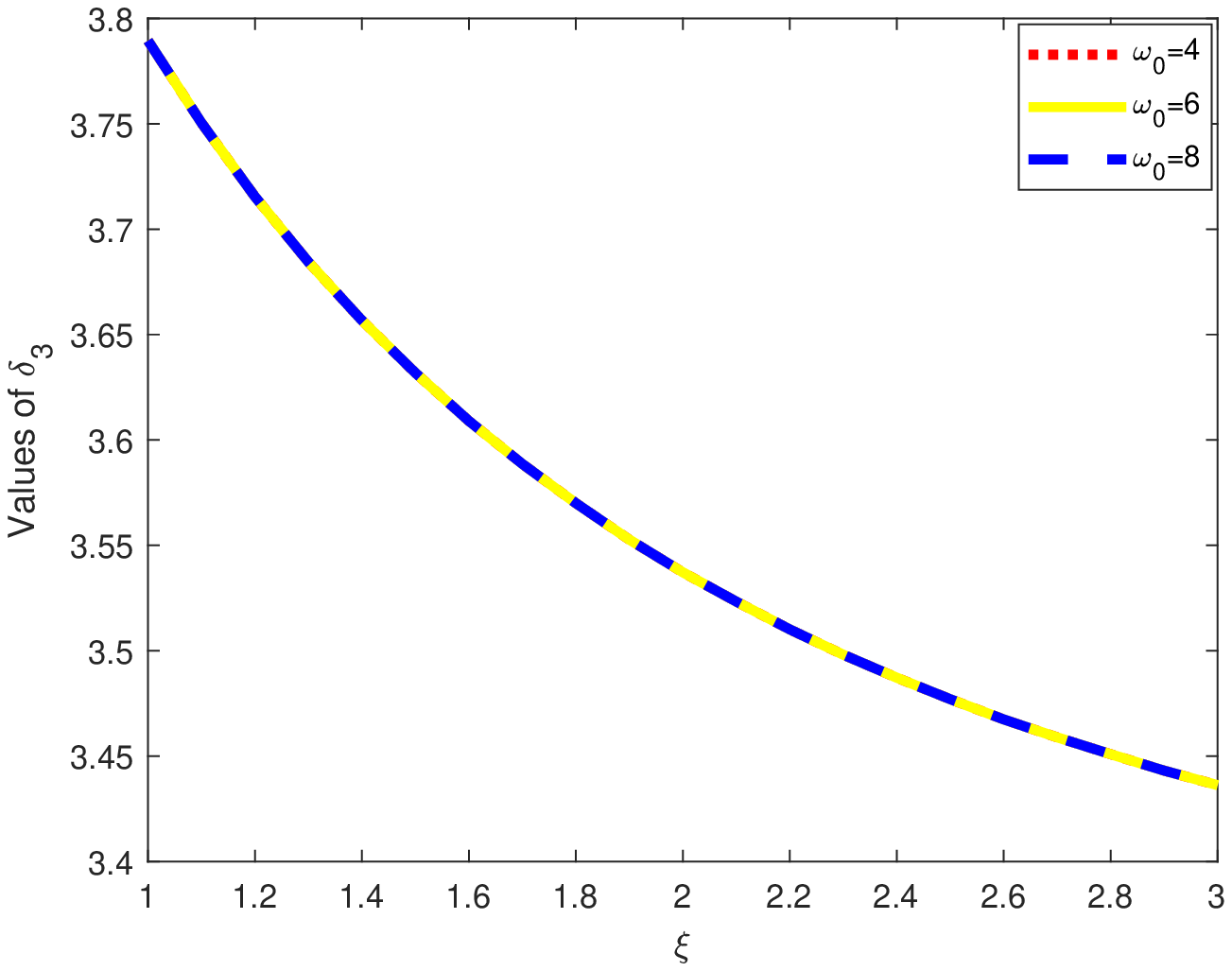}
        \caption*{(b)\;the corresponding values of $\delta_3$}
	\end{minipage}
    \caption{The impacts of $w_0$ and $\xi$ on the robust equilibrium investment strategy $u^{\ast}$ and values of $\delta_3$.}
	\label{fig1}
\end{figure}

Figure \ref{fig1} (a) shows the impacts of the initial wealth $w_0$ and the ambiguity averse parameter $\xi$ on the robust equilibrium investment strategy $u^{\ast}$  in the equation \eqref{eq36}. From Figure \ref{fig1} (a), we find that $u^{\ast}$ increases w.r.t. $w_0$, but decreases w.r.t. $\xi$. With the increase of $w_0$, the investor has greater ability to invest and take corresponding risks. However, along with the increase of $\xi$, the investor becomes more averse to the ambiguity in the financial market when the initial wealth is fixed. Thus, she/he would reduce her/his investment in the risky asset. It is also observed that $u^{\ast}$ decreases more sharply w.r.t. $\xi$ when $w_0$ is larger, which implies that rich investors should pay more attention to the ambiguity.

\begin{figure}[htbp]
	\begin{minipage}[t]{0.5\linewidth}
		\centering
		\includegraphics[width=3.0in]{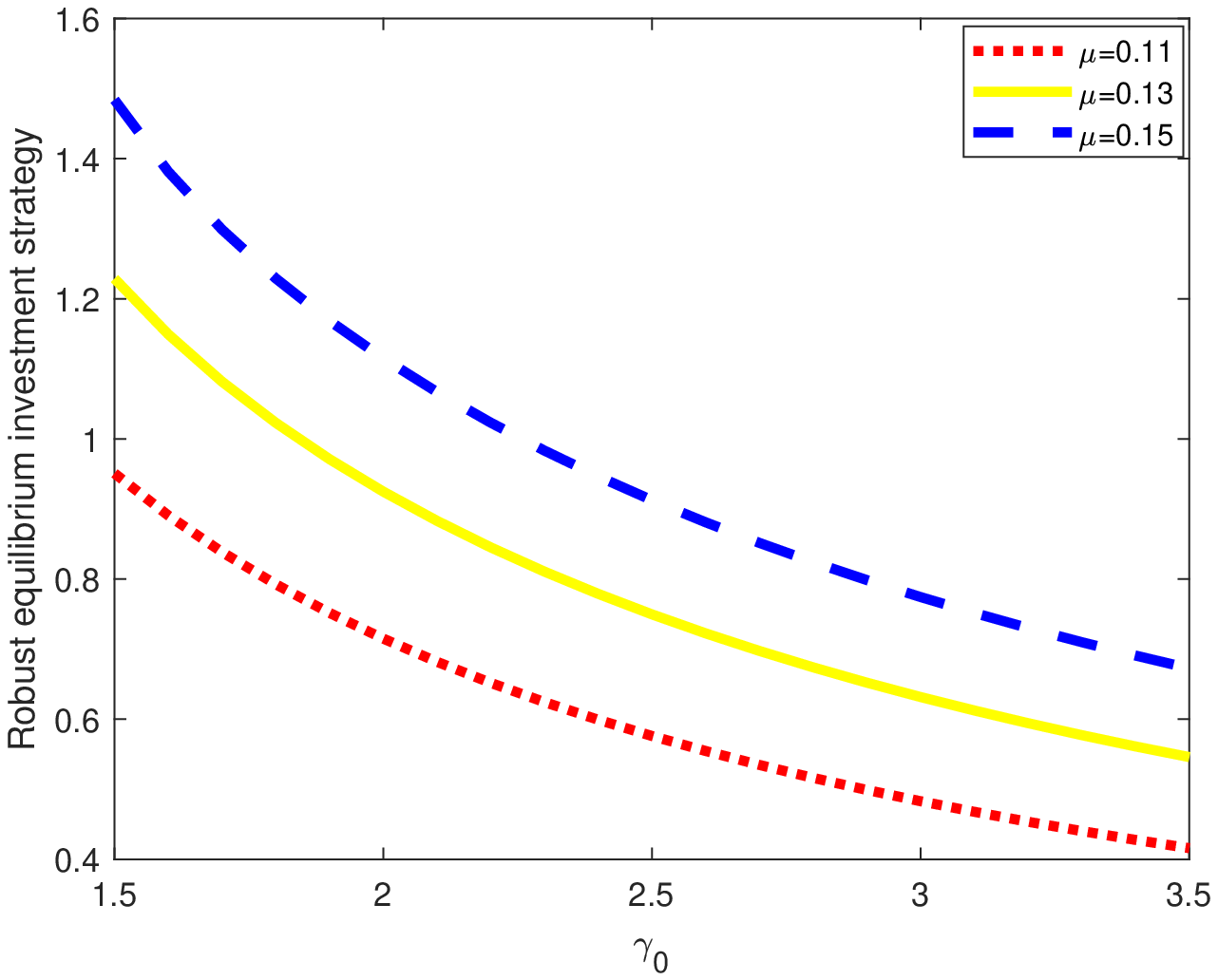}
        \caption*{(a)\;the impacts of $\mu$ and $\gamma_{0}$ on $u^{\ast}$}
	\end{minipage}%
	\begin{minipage}[t]{0.5\linewidth}
		\centering
		\includegraphics[width=3.0in]{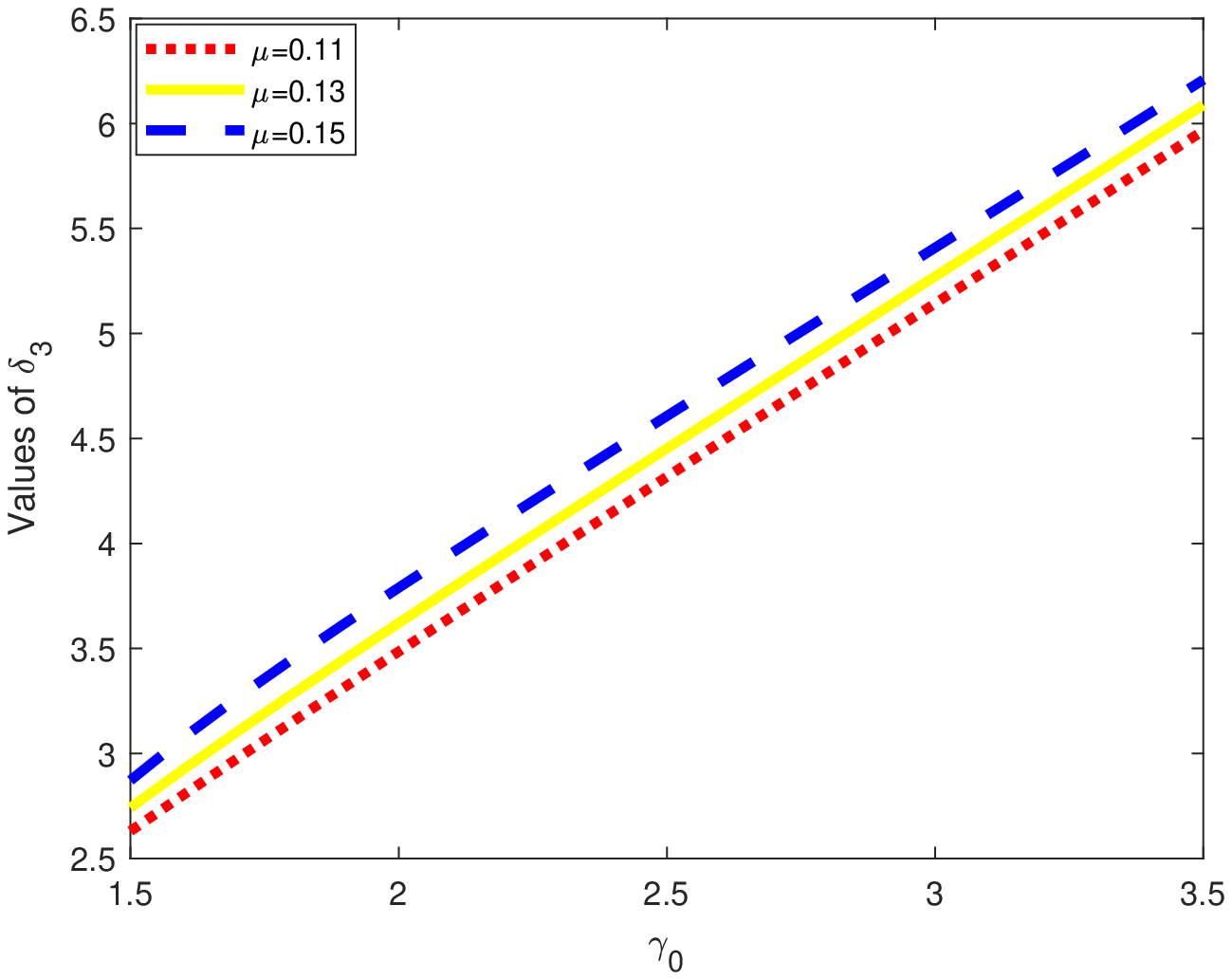}
        \caption*{(b)\;the corresponding values of $\delta_3$}
	\end{minipage}
    \caption{The impacts of $\mu$ and $\gamma_{0}$ on the robust equilibrium investment strategy $u^{\ast}$ and values of $\delta_3$.}
	\label{fig2}
\end{figure}

Figure \ref{fig2} (a) displays the variations of the robust equilibrium investment strategy $u^{\ast}$ with the risky asset's parameter $\mu$ and the risk aversion coefficient $\gamma_{0}$. It is easy to see that $u^{\ast}$ increases w.r.t. $\mu$, but decreases w.r.t. $\gamma_{0}$. Since $\mu$ denotes the expected return rate of the risky asset, the growth of $\mu$ means higher returns, which attracts the investor to invest more in the risky asset to pursue more interests. On the other hand, when $\gamma_{0}$ becomes larger, the investor is more risk averse about the risky asset. So the investor reduces the investment in the risky asset. It is also interesting to see that $u^{\ast}$ decreases more quickly w.r.t. $\gamma_{0}$ when $\mu$ is larger. This indicates that $\gamma_{0}$ has a greater impact on the investor's decisions when $\mu$ is larger.

\begin{figure}[htbp]
	\begin{minipage}[t]{0.5\linewidth}
		\centering
		\includegraphics[width=3.0in]{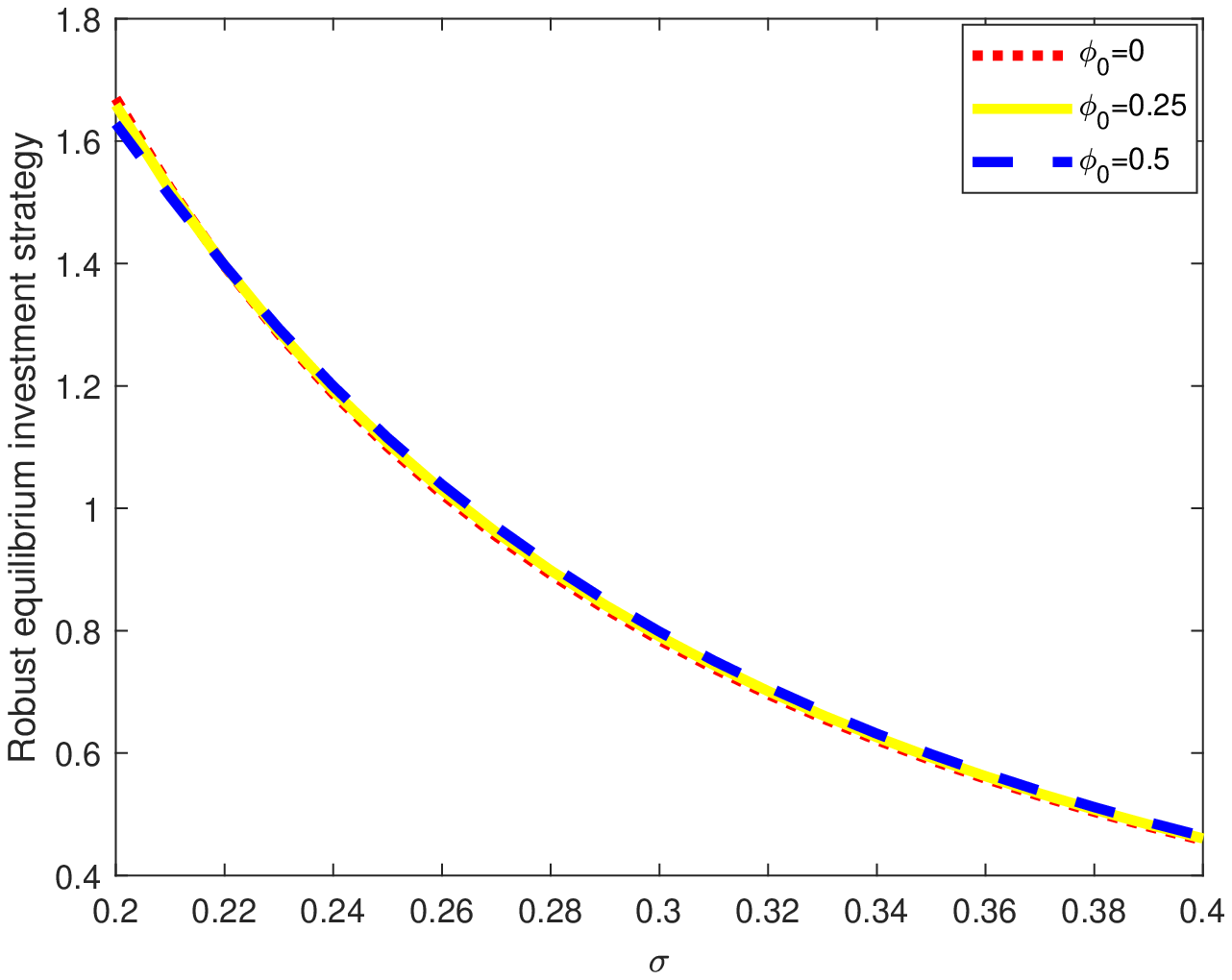}
        \caption*{(a)\;the impacts of $\phi_{0}$ and $\sigma$ on $u^{\ast}$}
	\end{minipage}%
	\begin{minipage}[t]{0.5\linewidth}
		\centering
		\includegraphics[width=3.0in]{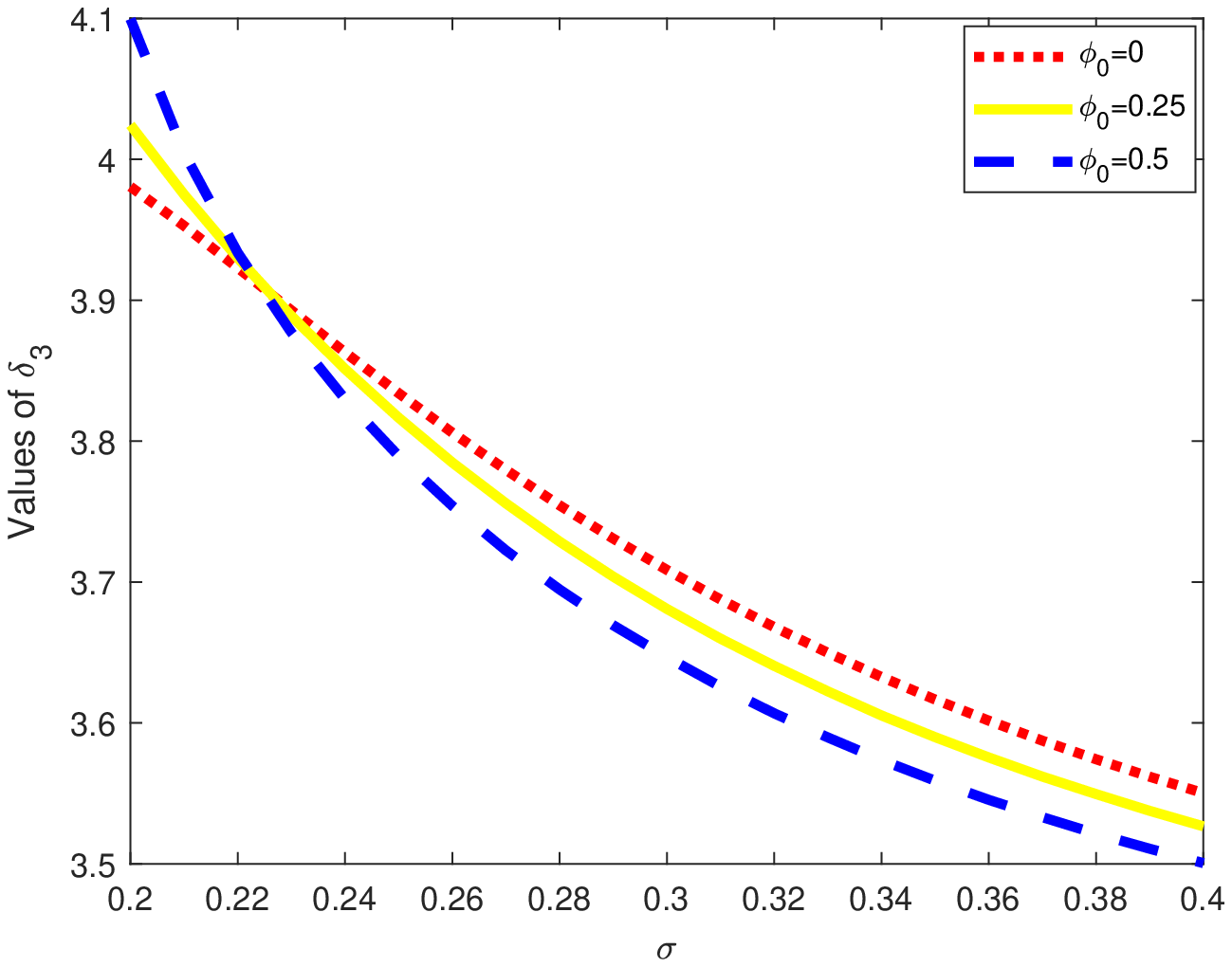}
        \caption*{(b)\;the corresponding values of $\delta_3$}
	\end{minipage}
    \caption{The impacts of $\phi_{0}$ and $\sigma$ on the robust equilibrium investment strategy $u^{\ast}$ and values of $\delta_3$.}
	\label{fig3}
\end{figure}

Figure \ref{fig3} (a) depicts the impacts of the skewness preference parameter $\phi_{0}$ and the volatility coefficient of the risky asset $\sigma$ on the robust equilibrium investment strategy $u^{\ast}$. We see that $u^{\ast}$ decreases evidently w.r.t. $\sigma$. As $\sigma$ increases, the price of the risky asset fluctuates greater and the financial market is more unstable. Therefore, the investor tends to reduce the investment in the risky asset to control risk. As the impacts of $\phi_{0}$ on $u^{\ast}$ are not obvious in Figure \ref{fig3} (a), we provide Figure \ref{fig12} to show the sensitivity of $u^{\ast}$ w.r.t. $\phi_{0}$. It is interesting to see that the difference between equilibrium investment strategies under different skewness preference generates $U$-shape lines with the increase of the volatility of the risky asset. First, we analyze the case when $\sigma\geq 0.22$. Overall, the difference between robust equilibrium investment strategies under different skewness preference is positive, which reflects that the mean-variance-skewness investor invests more wealth in the risky asset than the mean-variance investor. Furthermore, the greater the skewness preference, the greater the investment in risky asset. It is also observed that the difference first increases and then decreases as $\sigma$ increases. Second, when $\sigma<0.22$, the situation is totally different. Specially, the difference is negative, which means that the mean-variance-skewness investor invests less wealth in the risky asset than the mean-variance investor. Moreover, the greater the skewness preference, the smaller the investment in risky asset. It also reflects that the absolute value of the difference decreases when $\sigma$ increases. These phenomena have not been found in \cite{Mu2019}. We would like to point out that at first glance the difference between the comparing strategies is small, one can find the equilibrium strategy heavily depends on the initial wealth value in the equation \eqref{eq36}. When the magnitude of initial wealth is larger, the difference would be more obvious. Thus, it is meaningful and necessary to consider the skewness preference for investment in the financial market.

\begin{figure}[htbp]
    \centering
    {
        \begin{minipage}[t]{0.5\textwidth}
            \centering
            \includegraphics[width=1\textwidth]{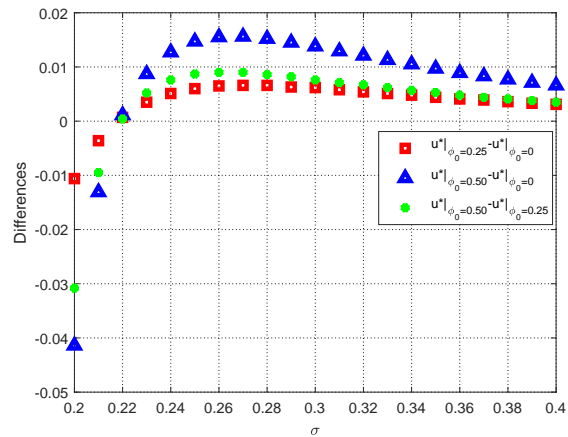}
        \end{minipage}
    }

    \caption{The difference between robust equilibrium investment strategies under different $\phi_{0}$ and $\sigma$.}
    \label{fig12}
\end{figure}

In Figures \ref{fig1} (b), \ref{fig2} (b) and \ref{fig3} (b), we also present the values of $\delta_3$ for each case. It is easy to see that $\delta_3$ are positive under different settings, which implies that the condition in Remark \ref{remark5} is satisfied, that is to say, our results are effective.

\subsection{Impacts of model parameters on the utility loss functions}
In what follows, we adopt numerical experiments to study the impacts of model parameters on the utility losses from ignoring skewness preference and model uncertainty for the ambiguity-averse investor.

First, we consider the utility loss from ignoring skewness preference, which is defined as \cite{Zeng2016}
$$
L_{1}(t)=1-\frac{\hat{V}(t,w)}{V(t,w)},
$$
where $\hat{V}(t,w)$ and $V(t,w)$ are given by $\eqref{eq41}$ and $\eqref{eq35}$, respectively.

\begin{figure}[htbp]
	\begin{minipage}[t]{0.5\linewidth}
		\centering
		\includegraphics[width=3.0in]{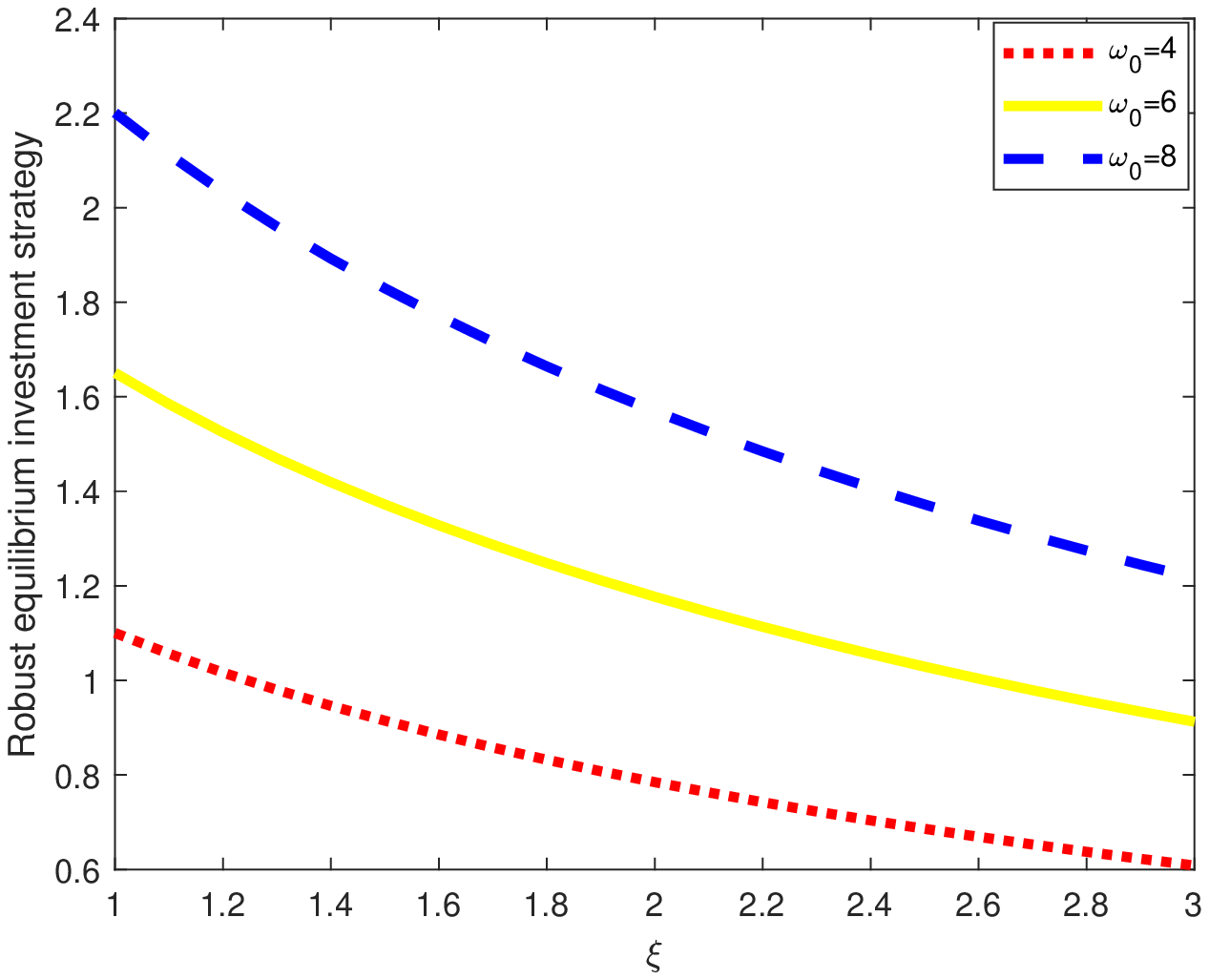}
        \caption*{(a)\;the impacts of $w_0$ and $\xi$ on $\hat{u}^{\ast}$}
	\end{minipage}%
	\begin{minipage}[t]{0.5\linewidth}
		\centering
		\includegraphics[width=3.0in]{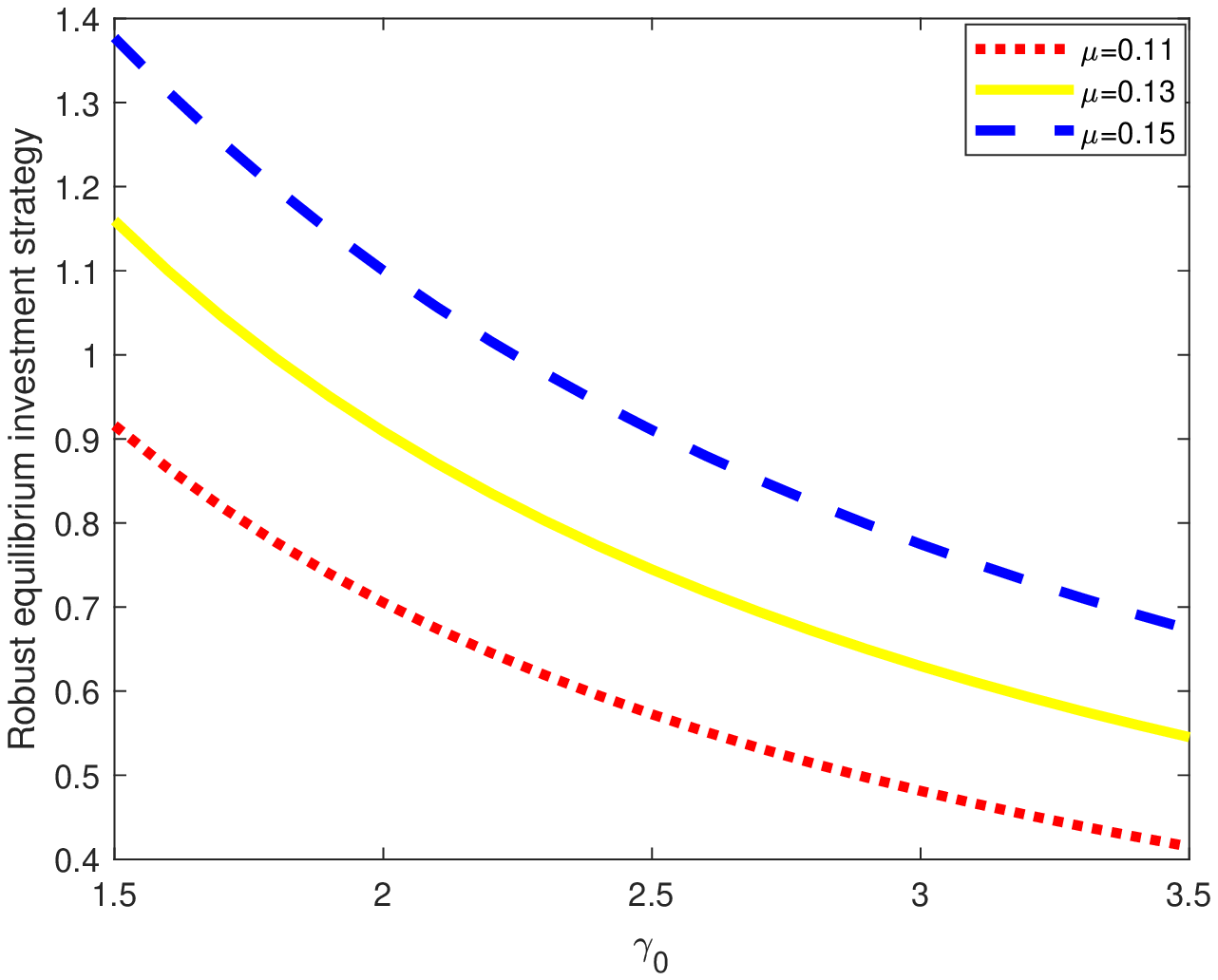}
        \caption*{(b)\;the impacts of $\mu$ and $\gamma_{0}$ on $\hat{u}^{\ast}$}
	\end{minipage}
    \caption{The impacts of $w_0$, $\xi$, $\mu$ and $\gamma_{0}$ on the robust equilibrium investment strategy $\hat{u}^{\ast}$.}
	\label{fig4}
\end{figure}

\begin{figure}[htbp]
	\begin{minipage}[t]{0.5\linewidth}
		\centering
		\includegraphics[width=3.0in]{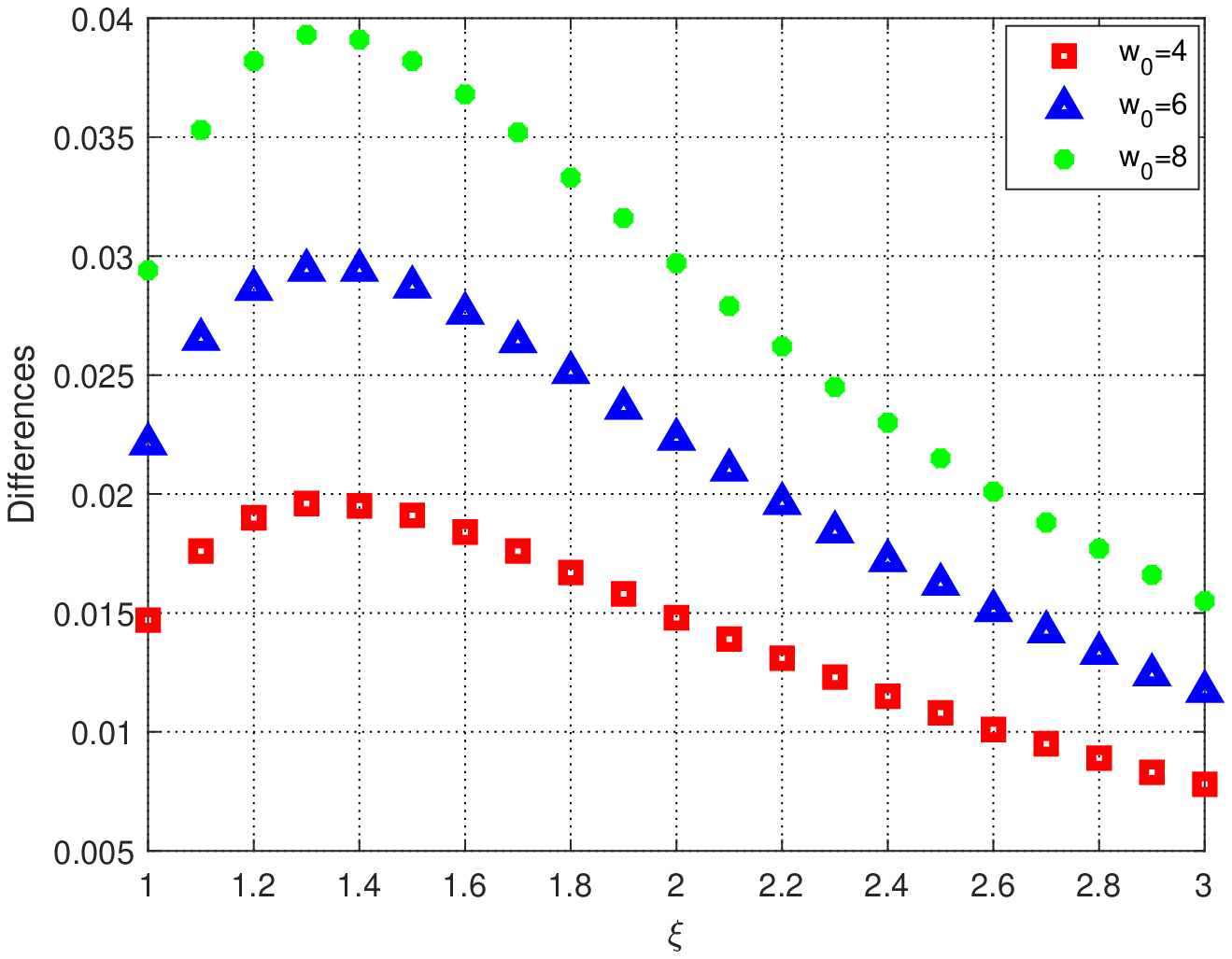}
        \caption*{(a)\;the differences under $w_0$ and $\xi$}
	\end{minipage}%
	\begin{minipage}[t]{0.5\linewidth}
		\centering
		\includegraphics[width=3.0in]{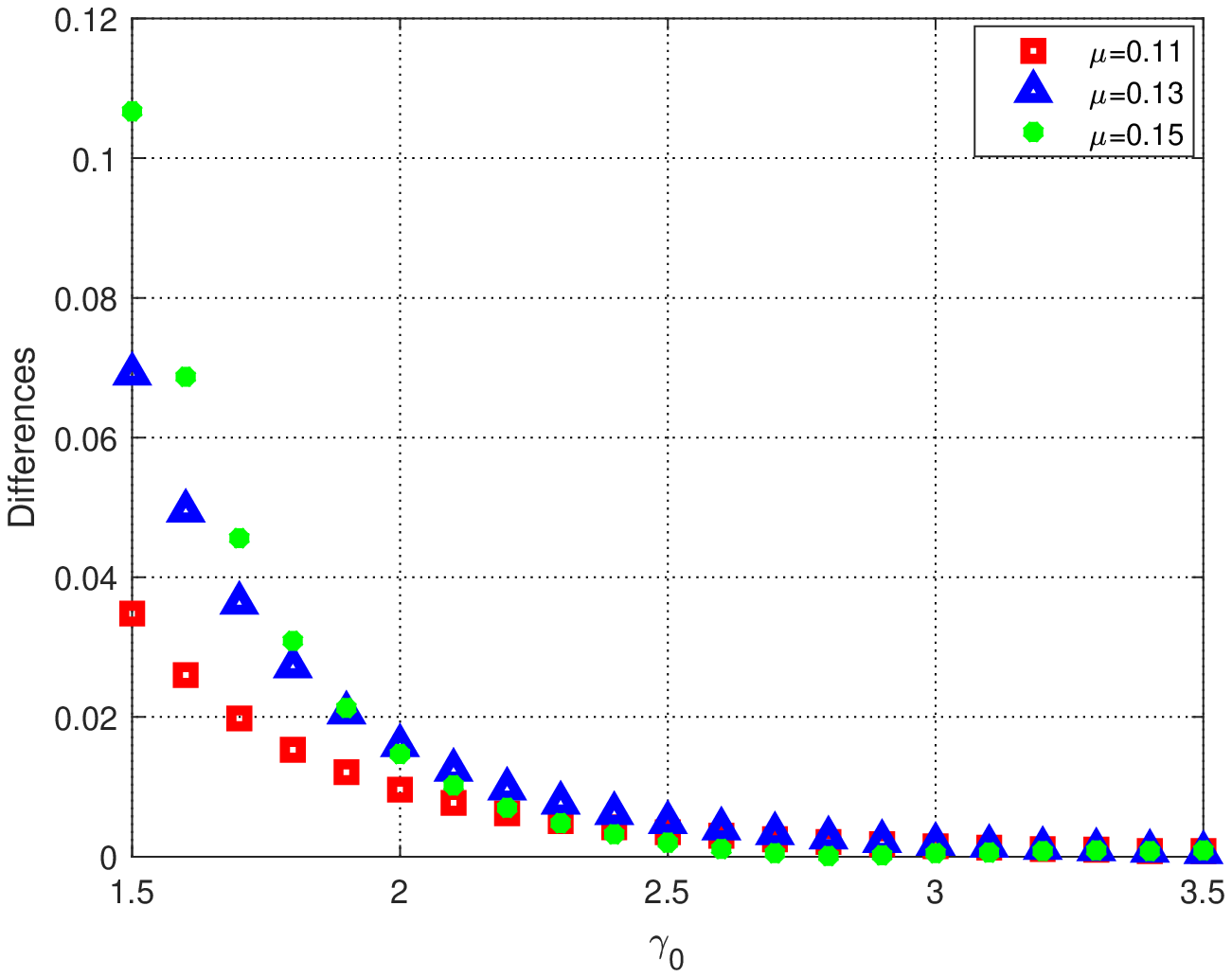}
        \caption*{(b)\;the differences under $\mu$ and $\gamma_{0}$}
	\end{minipage}
    \caption{The difference between robust equilibrium investment strategies under $w_0$, $\xi$, $\mu$ and $\gamma_{0}$.}
	\label{fig5}
\end{figure}

Compared with Figures \ref{fig1} and \ref{fig2}, Figure \ref{fig4} shows that the robust equilibrium investment strategy $\hat{u}_{t}^{\ast}$ in the equation \eqref{eq42} under the mean-variance criterion changes similarly w.r.t. $w_0$, $\xi$, $\mu$ and $\gamma_{0}$ than $u_{t}^{\ast}$. Figure \ref{fig5} presents the difference $(u_{t}^{\ast}-\hat{u}_{t}^{\ast})$ between the robust equilibrium investment strategy $u_{t}^{\ast}$ under the mean-variance-skewness criterion and the robust equilibrium investment strategy $\hat{u}_{t}^{\ast}$ under the mean-variance criterion with different values of $w_0$, $\xi$, $\mu$ and $\gamma_{0}$. We can find that the difference is always positive under different values of $w_0$, $\xi$, $\mu$ and $\gamma_{0}$. That is to say, $u_{t}^{\ast}$ is larger than $\hat{u}_{t}^{\ast}$. As shown in Figure \ref{fig5} (a), the difference between robust equilibrium investment strategies with different skewness preferences generates $U$-shape lines with the increase of $\xi$ when $w_0$ is fixed. The difference becomes more obvious when the initial wealth takes larger values. Additionally, the difference first increases from $\xi=1$ to its peak around $\xi=1.3$ and then decreases thereafter. A possible explanation is that the introduction of skewness would stimulate investment and thus offset the investment reduction because of ambiguity aversion, i.e., the skewness preference would slow down the reduction of investment owing to ambiguity aversion. Specially, compared with the case of only considering ambiguity-aversion, the effect of slowing down by skewness is more obvious when the skewness preference dominates the ambiguity aversion, which corresponds to the increase of the difference in Figure \ref{fig5} (a). When the ambiguity aversion parameter becomes larger, the effect of slowing down by the skewness preference weakens, which is consistent with the case that the difference decreases. Figure \ref{fig5} (b) displays that the difference between robust equilibrium investment strategies with different skewness preferences decreases w.r.t. $\gamma_{0}$ when $\mu$ is fixed and increases w.r.t. $\mu$ when $\gamma_{0}$ is fixed. If the risk aversion coefficient $\gamma_{0}$ is smaller, the difference is more obvious when $\mu$ takes larger values. However, when $\gamma_{0}$ becomes larger, the difference is very small even if $\mu$ takes larger values. This phenomenon implies that the skewness preference has no impact on the robust equilibrium investment strategy when the risk aversion coefficient is large enough. Furthermore, we would like to point out that when we take $\phi_{0}\equiv0$, it follows from the equation $\eqref{eq39}$ that $\delta_{3}=\gamma_{0}\hat{h}_{2}$ is always positive, which means that the condition in Remark \ref{remark5} is fulfilled. Thus, our results are effective.

\begin{figure}[htbp]
    \centering
    {
        \begin{minipage}[t]{0.5\textwidth}
            \centering
            \includegraphics[width=1\textwidth]{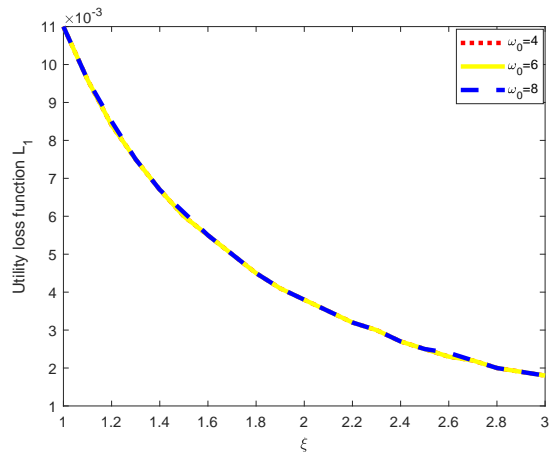}
        \end{minipage}
    }
    \caption{The impacts of $w_0$ and $\xi$ on the utility loss function $L_1$.}
    \label{fig6}
\end{figure}

Figure \ref{fig6} illustrates the effects of $w_0$ and $\xi$ on the utility loss function $L_1$. It follows from equations $\eqref{eq35}$ and $\eqref{eq41}$ that $w_0$ is proportional to $\hat{V}$ and $V$ as well as the expression of $L_1$ is independent of $w_0$. This is consistent with the phenomenon reflected in Figure \ref{fig6}. Moreover, as shown in Figure \ref{fig6}, we find that $L_1$ decreases along with $\xi$. As $\xi$ becomes larger, the investor is more ambiguity-averse and she/he would reduce the investment to avoid the ambiguity whether she/he is in the framework of mean-variance or mean-variance-skewness. As a result, the difference between $\hat{V}$ and $V$ due to the skewness preference becomes smaller. Therefore, the utility loss from ignoring skewness preference $L_1$ decreases w.r.t. $\xi$.

\begin{figure}[htbp]
    \centering
    {
        \begin{minipage}[t]{0.5\textwidth}
            \centering
            \includegraphics[width=1\textwidth]{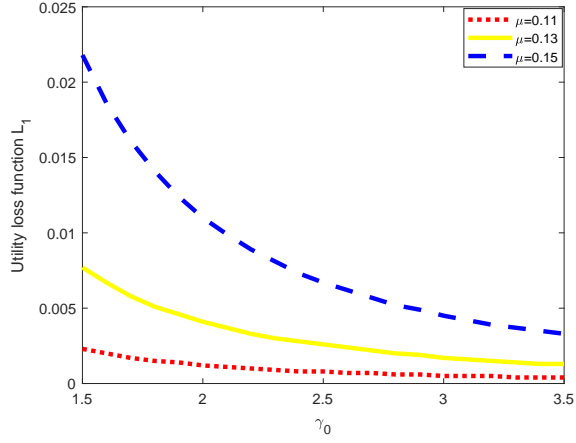}
        \end{minipage}
    }

    \caption{The impacts of $\mu$ and $\gamma_{0}$ on the utility loss function $L_1$.}
    \label{fig7}
\end{figure}

Figure \ref{fig7} demonstrates the impacts of $\mu$ and $\gamma_{0}$ on the utility loss function $L_1$. It shows that the bigger $\mu$, the bigger $L_1$. One can see from Figures \ref{fig2} (a) and \ref{fig4} (b) that the greater the value of $\mu$, the greater the investment of the investor. Thus, the difference between $\hat{V}$ and $V$ due to the skewness preference becomes larger, which leads to that the utility loss from ignoring skewness preference increases. On the other hand, when $\mu$ is fixed under our setting, we find that the utility loss function $L_1$ decreases w.r.t. $\gamma_{0}$. Figures \ref{fig2} (a) and \ref{fig4} (b) presents that the robust equilibrium investment strategies $\hat{u}_{t}^{\ast}$ and $u_{t}^{\ast}$ decrease w.r.t. $\gamma_{0}$. The difference between $\hat{V}$ and $V$ due to the skewness preference becomes smaller and then the utility loss from ignoring skewness preference decreases.

\begin{figure}[htbp]
    \centering
    {
        \begin{minipage}[t]{0.5\textwidth}
            \centering
            \includegraphics[width=1\textwidth]{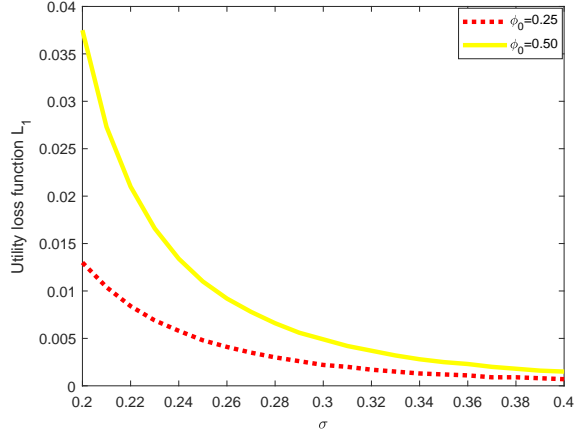}
        \end{minipage}
    }

    \caption{The impacts of $\phi_{0}$ and $\sigma$ on the utility loss function $L_1$.}
    \label{fig8}
\end{figure}

Figure \ref{fig8} discloses that the utility loss function $L_1$ decreases w.r.t. $\sigma$, whereas it increases w.r.t. $\phi_{0}$. It follows from the analysis of Figure \ref{fig3} (a) that the robust equilibrium investment strategy $u_{t}^{\ast}$ decreases evidently w.r.t. $\sigma$ when $\phi_{0}$ is fixed, which results in that the difference between $\hat{V}$ and $V$ becomes smaller. Therefore, the utility loss from ignoring skewness preference decreases along with the volatility of the risky asset.  As $\phi_{0}$ becomes larger, the investor puts more weight on skewness when making investment decisions and hence the utility loss from ignoring skewness preference increases. Thus, it is quite necessary to take the skewness preference into account in the investment decision-making.

Second, we discuss the utility loss from ignoring model uncertainty. Assume that the ambiguity-averse investor does not take the robust equilibrium investment strategy $u_{t}^{\ast}$ given in Theorem \ref{thm4.2}, but adopts strategy as if she/he were an ambiguity-neutral investor, i.e., the ambiguity-averse investor uses the strategy $\widetilde{u}_{t}^{\ast}$ given in Corollary \ref{cor1}. The equilibrium value function for the ambiguity-averse investor following the strategy $\widetilde{u}_{t}^{\ast}$ is defined by
\begin{equation*}
  V_{1}(t,w)=\inf\limits_{\textbf{q}\in\bm{Q}}\; \left\{\mathbb{E}^{\mathbb{Q}}_{t,w}\left[W^{\widetilde{u}_{t}^{\ast}}_{T}\right]-\frac{\gamma_{0}}{2w}Var^{\mathbb{Q}}_{t,w}\left[W^{\widetilde{u}_{t}^{\ast}}_{T}\right]
  +\frac{\phi_{0}}{3w^{2}}Skew^{\mathbb{Q}}_{t,w}\left[W^{\widetilde{u}_{t}^{\ast}}_{T}\right]+\mathbb{E}^{\mathbb{Q}}_{t,w}
\left[\int^{T}_{t}\frac{\textbf{q}'_{s}\textbf{q}_{s}}{2\xi\Phi(s,W_{s}) }\mathrm{d}s\right]\right\},
\end{equation*}
where $W^{\widetilde{u}_{t}^{\ast}}_{T}$ is the terminal wealth of the ambiguity-averse investor under the strategy $\widetilde{u}_{t}^{\ast}$ and $\Phi$ is given in Theorem \ref{thm4.2}. It is noteworthy that $\textbf{q}_{t}$, which characterizes the alternative model, is still determined endogenously and depends on the investment strategy. Different from the previous case of deriving the robust equilibrium investment strategy $u_{t}^{\ast}$, the investment strategy is now pre-specified. According to $\eqref{eq20}$, we can know that $V\geq V_{1}$ with the equality when the pre-specified strategy is taken as $u_{t}^{\ast}$. Similar to the previous derivations, we obtain the equilibrium value function $V_{1}$ under the strategy $\widetilde{u}_{t}^{\ast}$ as
\begin{equation}\label{eq46}
  V_{1}(t,w)=\left[a_{1}(t)-\frac{\gamma_{0}}{2}\left(a_{2}(t)-b_{1}^{2}(t)\right)+\frac{\phi_{0}}{3}
  \left(2b_{1}^{3}(t)-3b_{1}(t)a_{2}(t)+a_{3}(t)\right)\right]w,
\end{equation}
where $a(t)=\frac{a_{1}(t)+\gamma_{0}\left(b_{1}^{2}(t)-a_{2}(t)\right)+\phi_{0}
\left(a_{3}(t)+2b_{1}^{3}(t)-3b_{1}(t)c_{1}(t)\right)}{\gamma_{0}a_{2}(t)+2\phi_{0}(b_{1}(t)c_{1}(t)-a_{3}(t))}$ and
\begin{eqnarray*}
   a_{2}(t)&=&c_{1}(t)=e^{\int^{T}_{t}\left[2r_{s}+2\Theta_{s}\tilde{f}(s)-\frac{2\xi\Theta_{s}\tilde{f}^{2}(s)}{a(s)}+\Theta_{s}\tilde{f}^{2}(s)\right]\mathrm{d}s},\\
   a_{3}(t)&=&e^{\int^{T}_{t}3\left[r_{s}+\Theta_{s}\tilde{f}(s)-\frac{\xi\Theta_{s}\tilde{f}^{2}(s)}{a(s)}+\Theta_{s}\tilde{f}^{2}(s)\right]\mathrm{d}s},\\
   b_{1}(t)&=&e^{\int^{T}_{t}\left[r_{s}+\Theta_{s}\tilde{f}(s)-\frac{\xi\Theta_{s}\tilde{f}^{2}(s)}{a(s)}\right]\mathrm{d}s},\\
   a_{1}(t)&=&b_{1}(t)+\int^{T}_{t}\frac{\xi \Theta_{s}\tilde{f}^{2}(s)}{2}\left[\gamma_{0} a_{2}(s)+2\phi_{0}(b_{1}(s)a_{2}(s)-a_{3}(s))\right]
   \times e^{\int^{s}_{t}\left[r_{u}+\Theta_{u}\tilde{f}(u)-\frac{\xi\Theta_{u}\tilde{f}^{2}(u)}{a(u)}\right]\mathrm{d}u}\mathrm{d}s
\end{eqnarray*}
with $a_{1}(T)=a_{2}(T)=a_{3}(T)=b_{1}(T)=c_{1}(T)=1$ and $\tilde{f}(t)$ is given in Corollary \ref{remark3}.
Furthermore, the utility loss from ignoring model uncertainty is defined as follows \cite{Zeng2016}
$$
L_{2}(t)=1-\frac{V_{1}(t,w)}{V(t,w)},
$$
where $V_{1}(t,w)$ and $V(t,w)$ are given by $\eqref{eq46}$ and $\eqref{eq35}$, respectively.

We modify the value of parameter $\mu$ to 0.10 and present the values of $\delta_{3}$ for each case in Figures \ref{fig9} (b), \ref{fig10} (b) and \ref{fig11} (b), which shows that the condition in Remark \ref{remark5} is also satisfied and our results are effective.

\begin{figure}[htbp]
	\begin{minipage}[t]{0.5\linewidth}
		\centering
		\includegraphics[width=3.0in]{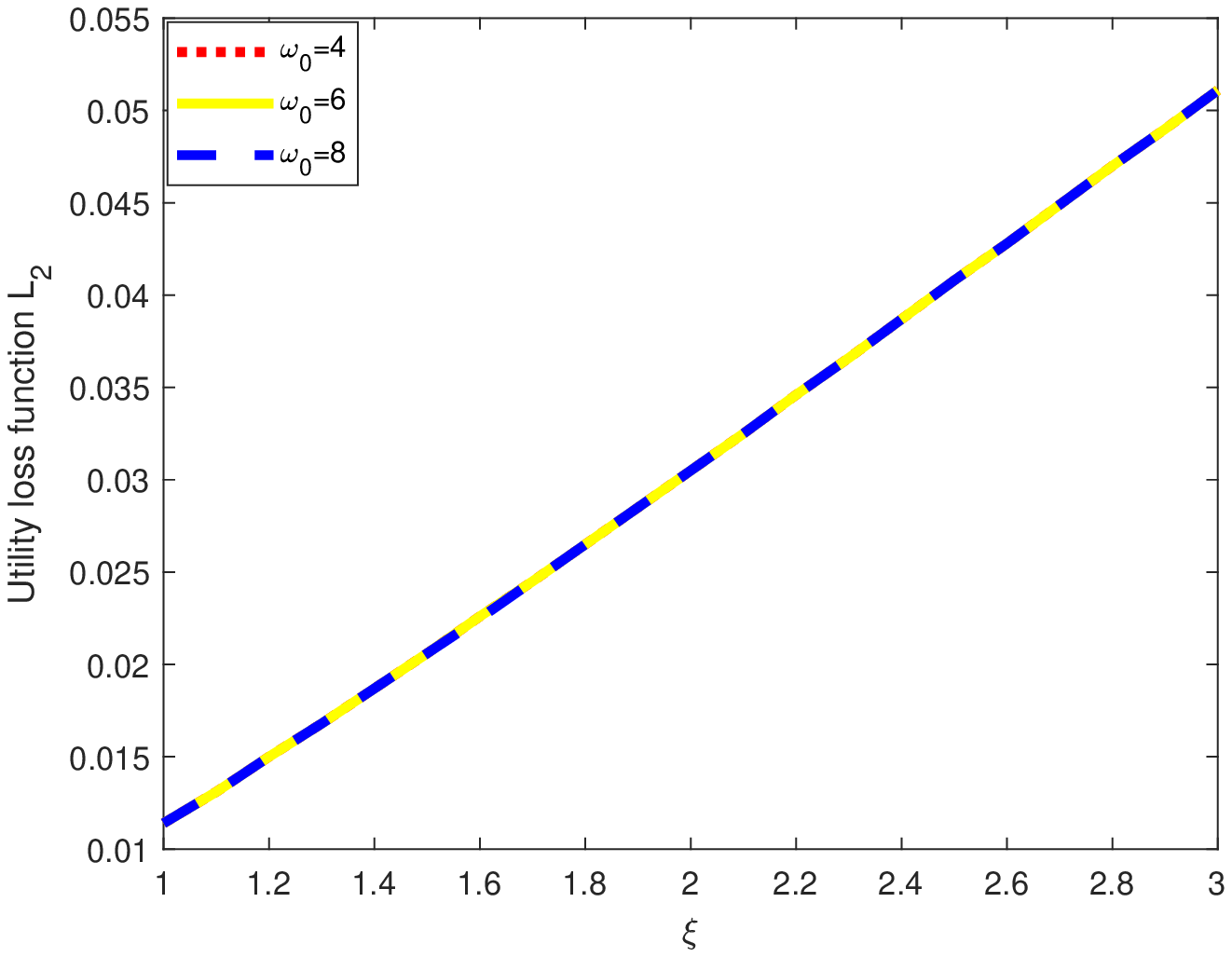}
        \caption*{(a)\;the impacts of $w_0$ and $\xi$ on $L_2$}
	\end{minipage}%
	\begin{minipage}[t]{0.5\linewidth}
		\centering
		\includegraphics[width=3.0in]{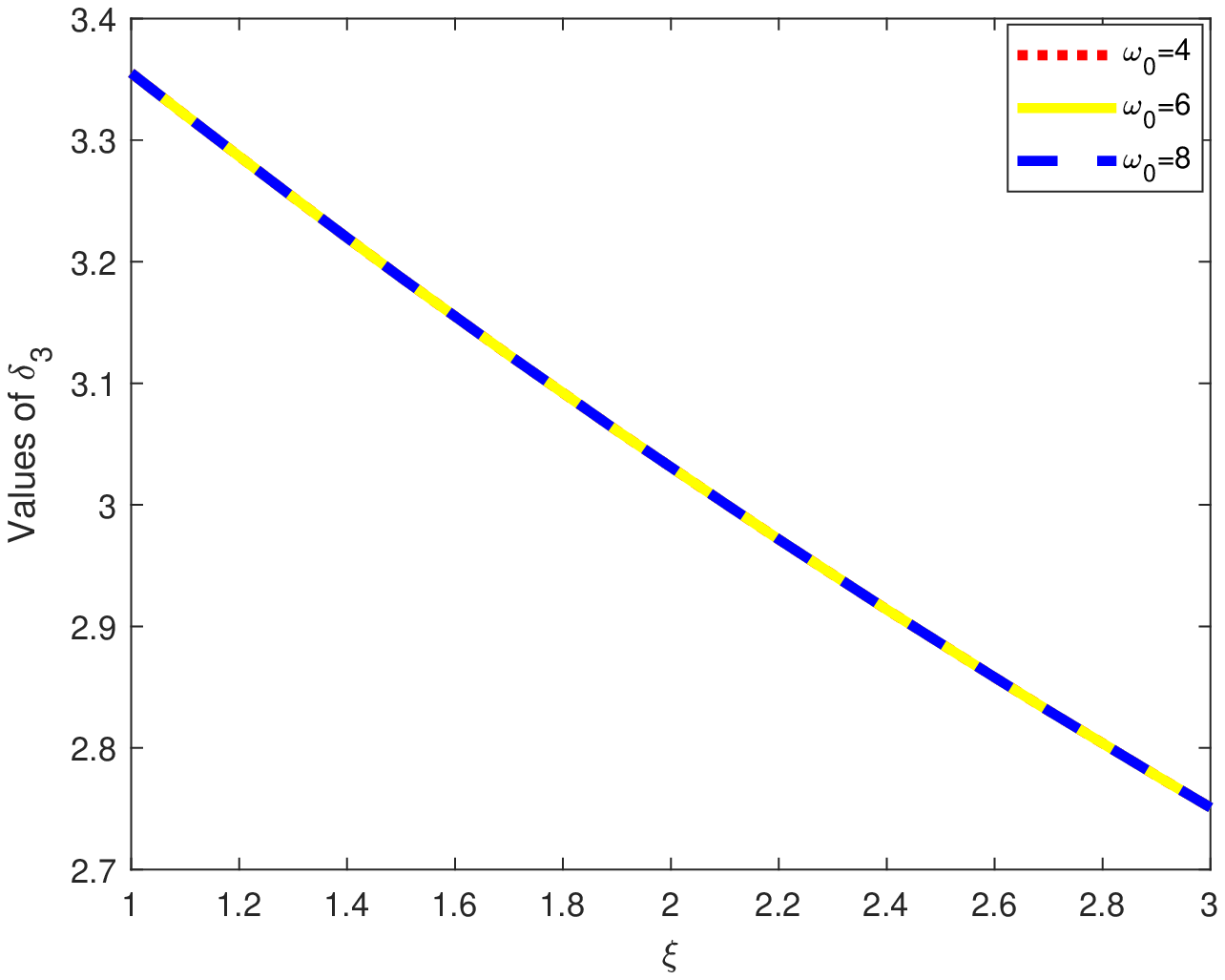}
        \caption*{(b)\;the corresponding values of $\delta_3$}
	\end{minipage}
    \caption{The impacts of $w_0$ and $\xi$ on the utility loss function $L_2$ and values of $\delta_3$.}
	\label{fig9}
\end{figure}

Figure \ref{fig9} (a) discloses the impacts of the initial wealth $w_0$ and the ambiguity aversion parameter $\xi$ on the utility loss function $L_2$. From Figure \ref{fig9} (a), we find that the utility loss function $L_2$ is independent of $w_0$, which is the same with that in Figure \ref{fig6} (a) and can be explained in the same way. On the other hand, we can see that the utility loss function $L_2$ increases w.r.t. $\xi$. This phenomenon is in line with our intuition, if the ambiguity averse investor is more uncertain about the reference measure, then ignoring model uncertainty may result in higher utility losses. Thus, it is quite necessary to study model uncertainty for the portfolio selection problems.

\begin{figure}[htbp]
	\begin{minipage}[t]{0.5\linewidth}
		\centering
		\includegraphics[width=3.0in]{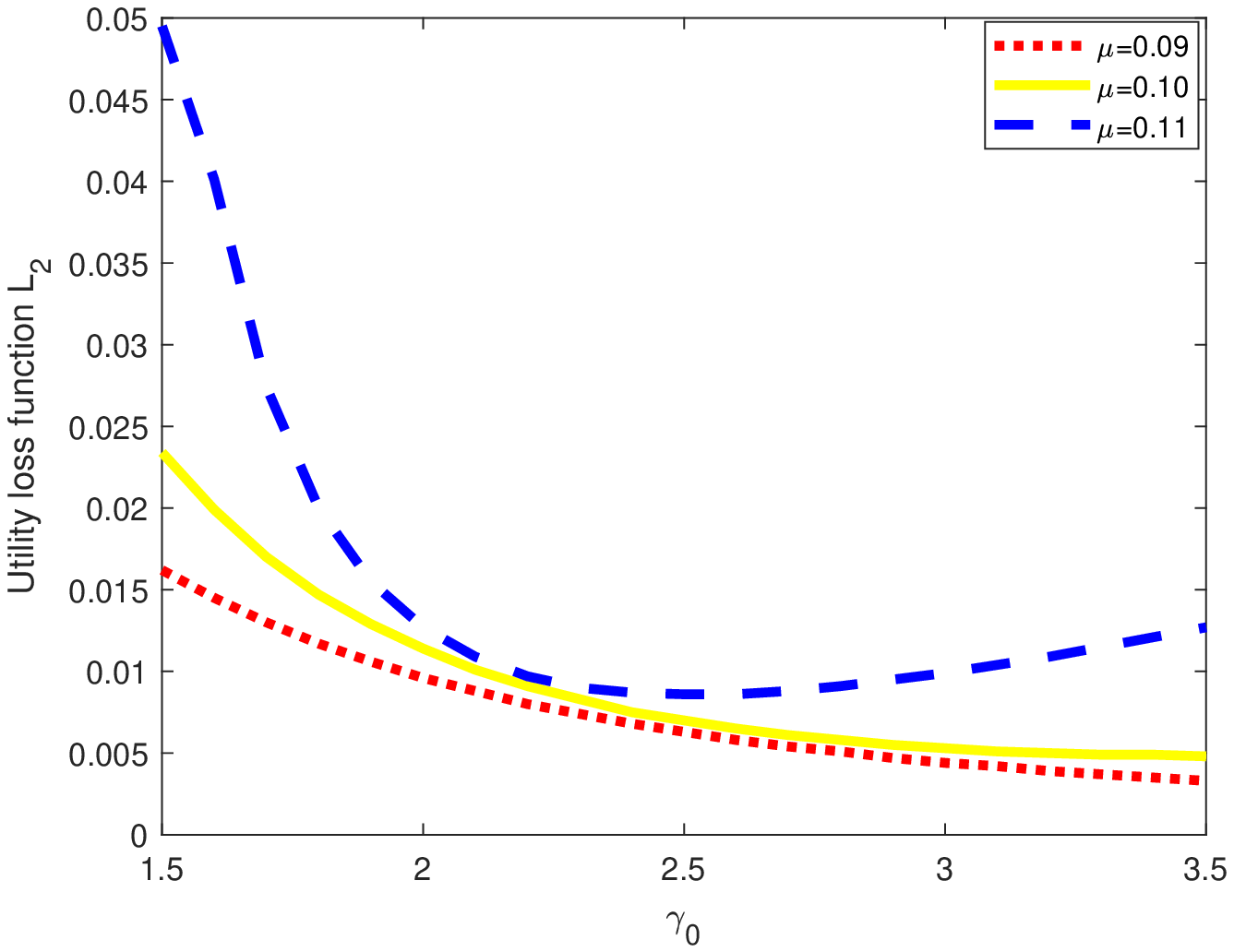}
        \caption*{(a)\;the impacts of $\mu$ and $\gamma_{0}$ on $L_2$}
	\end{minipage}%
	\begin{minipage}[t]{0.5\linewidth}
		\centering
		\includegraphics[width=3.0in]{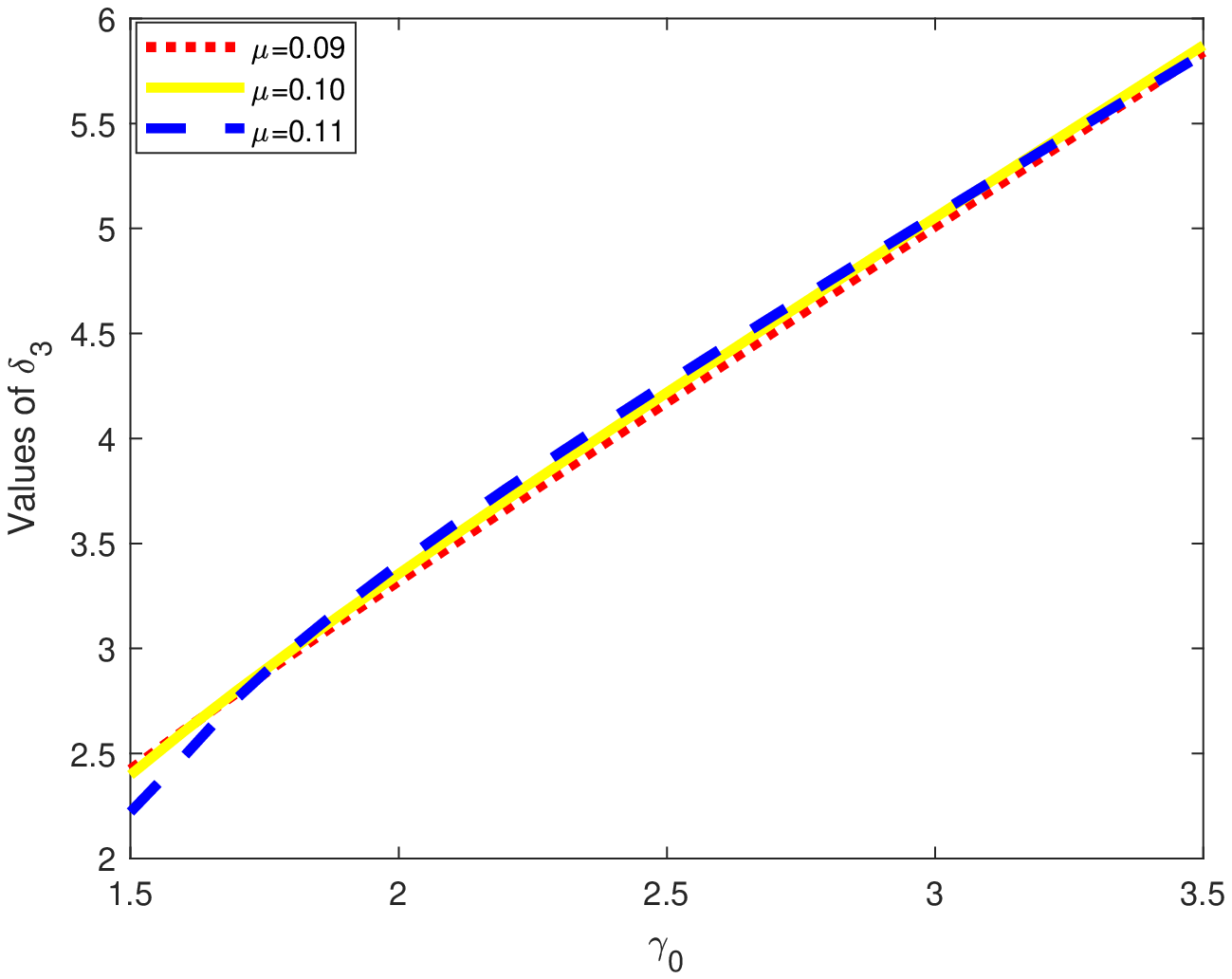}
        \caption*{(b)\;the corresponding values of $\delta_3$}
	\end{minipage}
    \caption{The impacts of $\mu$ and $\gamma_{0}$ on the utility loss function $L_2$ and values of $\delta_3$.}
	\label{fig10}
\end{figure}

Figures \ref{fig10} (a) and \ref{fig11} (a) show the impacts of the expected return rate of the risky asset $\mu$, the risk aversion coefficient $\gamma_{0}$, the skewness preference parameter $\phi_{0}$ and the volatility coefficient of the risky asset $\sigma$ on the utility loss function $L_2$. Specifically, we can see from Figure \ref{fig10} (a) that $L_2$ increases w.r.t. $\mu$ as well as changes with different risk aversion coefficients $\gamma_{0}$ and is not monotone in $\gamma_{0}$ under our setting. From Figure \ref{fig11} (a), we observe that $L_2$ increases w.r.t. $\phi_{0}$ and decreases w.r.t $\sigma$. Moreover, the utility loss function $L_2$ is always positive as expected, which also indicates that it is important and meaningful to consider model uncertainty when investing in the financial market.

\begin{figure}[htbp]
	\begin{minipage}[t]{0.5\linewidth}
		\centering
		\includegraphics[width=3.0in]{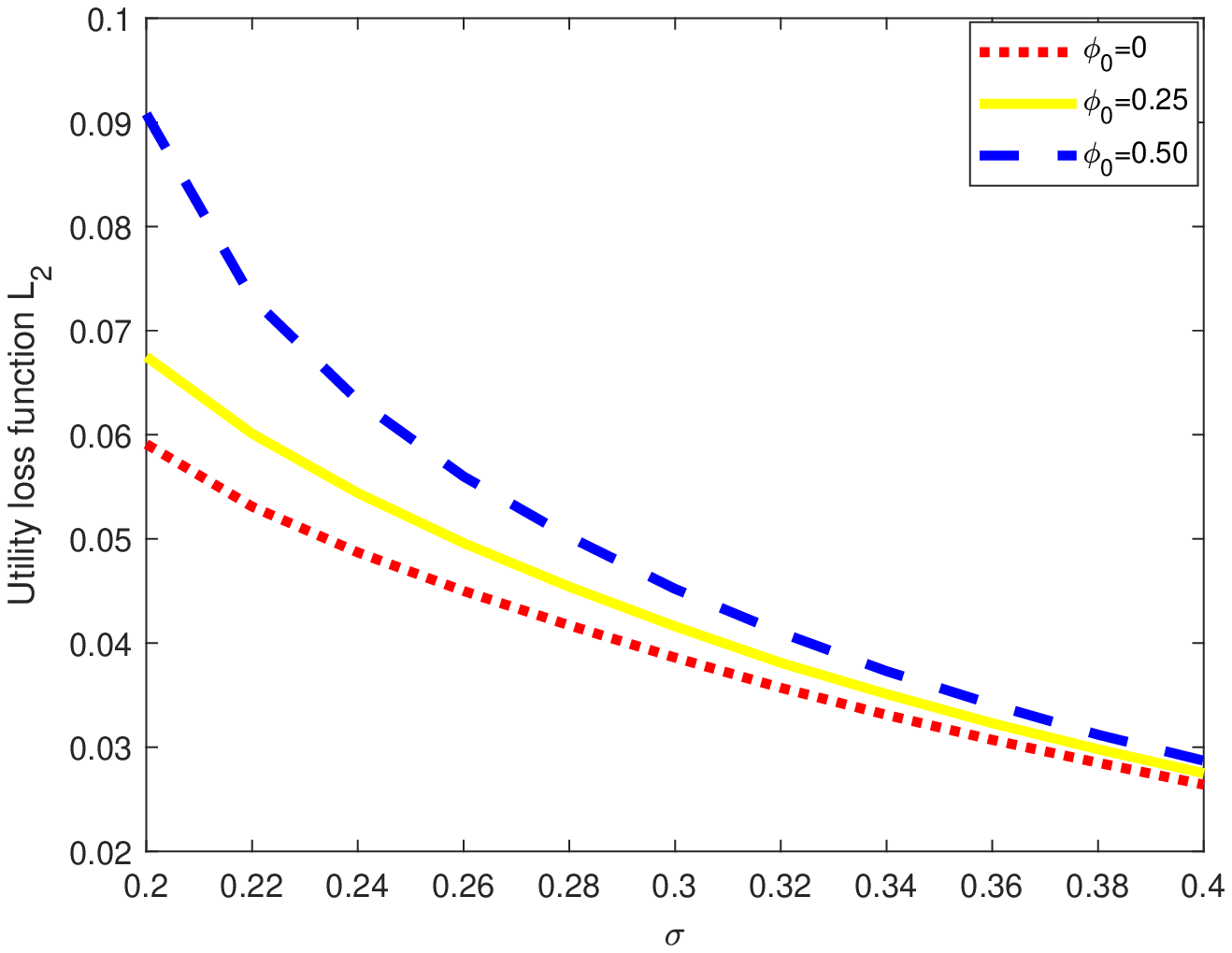}
        \caption*{(a)\;the impacts of $\phi_{0}$ and $\sigma$ on $L_2$}
	\end{minipage}%
	\begin{minipage}[t]{0.5\linewidth}
		\centering
		\includegraphics[width=3.0in]{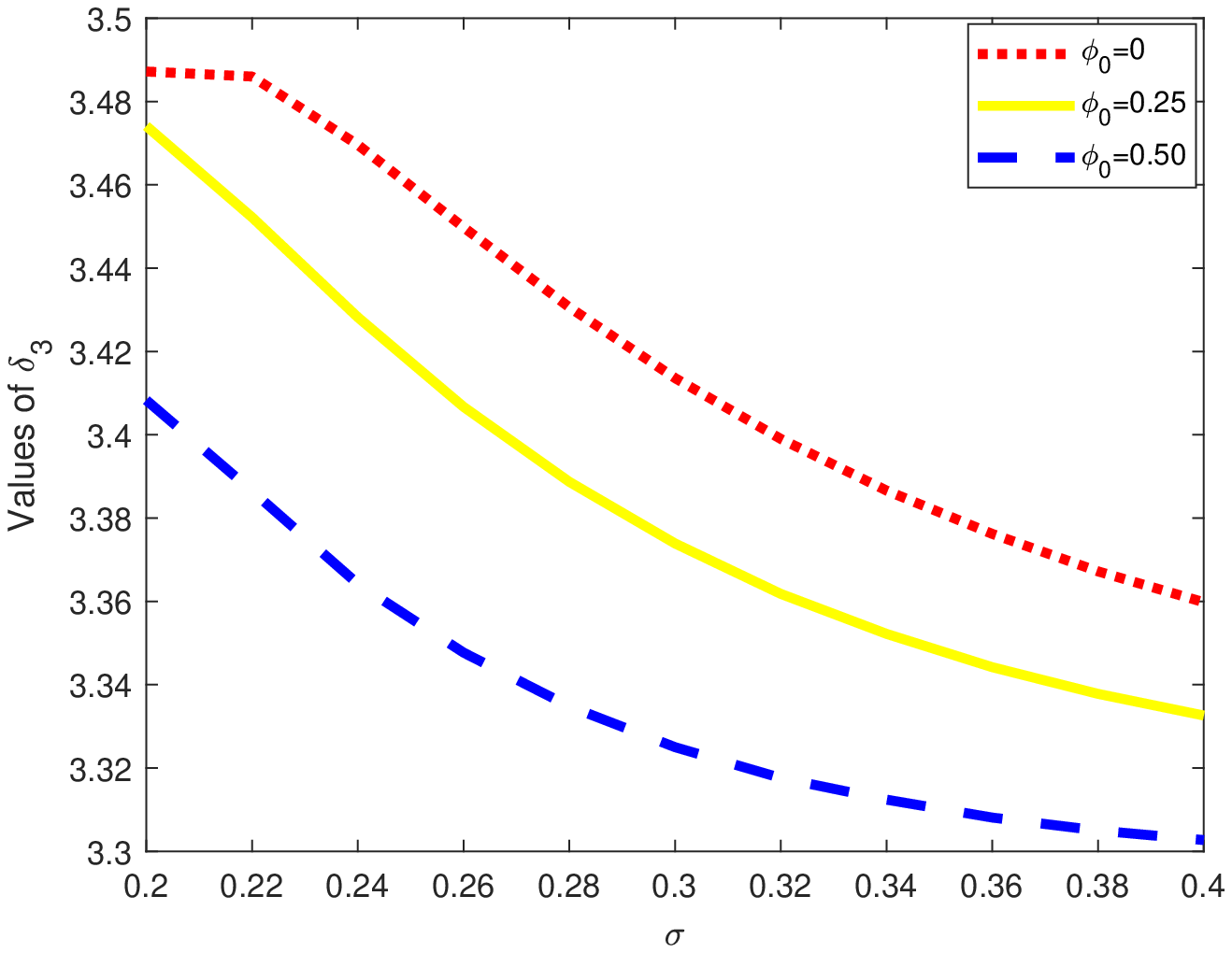}
        \caption*{(b)\;the corresponding values of $\delta_3$}
	\end{minipage}
    \caption{The impacts of $\phi_{0}$ and $\sigma$ on the utility loss function $L_2$ and values of $\delta_3$.}
	\label{fig11}
\end{figure}

Finally, we study the utility loss from ignoring both skewness preference and model uncertainty. Assume that the ambiguity-averse investor does not take the robust equilibrium investment strategy $u_{t}^{\ast}$ given in Theorem \ref{thm4.2}, but adopts strategy as if she/he were an ambiguity-neutral investor under the mean-variance criterion, i.e., the ambiguity-averse investor uses the strategy $\bar{u}_{t}^{\ast}$ given in Corollary \ref{cor3}. The equilibrium value function for the ambiguity-averse investor following the strategy $\bar{u}_{t}^{\ast}$ is defined by
\begin{equation*}
  V_{2}(t,w)=\inf\limits_{\textbf{q}\in\bm{Q}}\; \left\{\mathbb{E}^{\mathbb{Q}}_{t,w}\left[W^{\bar{u}_{t}^{\ast}}_{T}\right]-\frac{\gamma_{0}}{2w}Var^{\mathbb{Q}}_{t,w}\left[W^{\bar{u}_{t}^{\ast}}_{T}\right]
  +\mathbb{E}^{\mathbb{Q}}_{t,w}
\left[\int^{T}_{t}\frac{\textbf{q}'_{s}\textbf{q}_{s}}{2\xi\Phi(s,W_{s}) }\mathrm{d}s\right]\right\},
\end{equation*}
where $W^{\bar{u}_{t}^{\ast}}_{T}$ is the terminal wealth of the ambiguity-averse investor under the strategy $\bar{u}_{t}^{\ast}$ and $\Phi$ is the same as the one given in Theorem \ref{thm4.2} with $\phi_{0}=0$. Similar to the previous derivations, we derive the equilibrium value function $V_{2}$ under the strategy $\bar{u}_{t}^{\ast}$ as
\begin{equation}\label{eq47}
  V_{2}(t,w)=\left[\bar{a}_{1}(t)-\frac{\gamma_{0}}{2}\left(\bar{a}_{2}(t)-\bar{b}_{1}^{2}(t)\right)\right]w,
\end{equation}
where $\bar{a}(t)=\frac{\bar{a}_{1}(t)+\gamma_{0}\left(\bar{b}_{1}^{2}(t)-\bar{a}_{2}(t)\right)}{\gamma_{0}\bar{a}_{2}(t)}$ and
\begin{eqnarray*}
   \bar{a}_{2}(t)&=&e^{\int^{T}_{t}\left[2r_{s}+2\Theta_{s}\bar{f}(s)-\frac{2\xi\Theta_{s}\bar{f}^{2}(s)}{\bar{a}(s)}+\Theta_{s}\bar{f}^{2}(s)\right]\mathrm{d}s},\\
   \bar{b}_{1}(t)&=&e^{\int^{T}_{t}\left[r_{s}+\Theta_{s}\bar{f}(s)-\frac{\xi\Theta_{s}\bar{f}^{2}(s)}{\bar{a}(s)}\right]\mathrm{d}s},\\
   \bar{a}_{1}(t)&=&\bar{b}_{1}(t)+\int^{T}_{t}\frac{\xi \Theta_{s}\bar{f}^{2}(s)}{2}\gamma_{0} \bar{a}_{2}(s)
   \times e^{\int^{s}_{t}\left[r_{u}+\Theta_{u}\bar{f}(u)-\frac{\xi\Theta_{u}\bar{f}^{2}(u)}{\bar{a}(u)}\right]\mathrm{d}u}\mathrm{d}s
\end{eqnarray*}
with $\bar{a}_{1}(T)=\bar{a}_{2}(T)=\bar{b}_{1}(T)=1$ and $\bar{f}(t)$ is given in Corollary \ref{remark5}.
Then, the utility loss from ignoring both skewness preference and model uncertainty is defined as follows \cite{Zeng2016}
$$
L_{3}(t)=1-\frac{V_{2}(t,w)}{V(t,w)},
$$
where $V_{2}(t,w)$ and $V(t,w)$ are given by $\eqref{eq47}$ and $\eqref{eq35}$, respectively.

\begin{figure}[htbp]
    \centering
    {
        \begin{minipage}[t]{0.5\textwidth}
            \centering
            \includegraphics[width=1\textwidth]{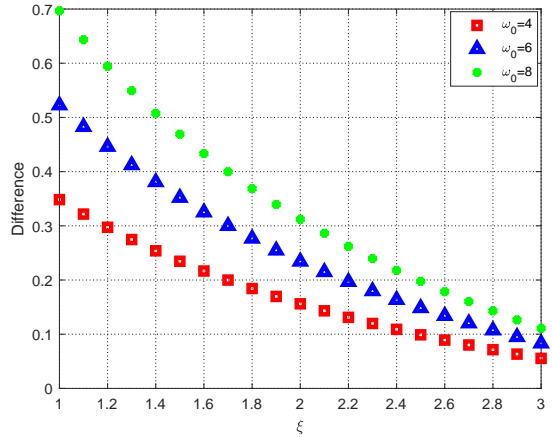}
        \end{minipage}
    }
    \caption{The difference between equilibrium investment strategies under different $w_{0}$ and $\xi$.}
    \label{fig13}
\end{figure}

\begin{figure}[htbp]
    \centering
    {
        \begin{minipage}[t]{0.5\textwidth}
            \centering
            \includegraphics[width=1\textwidth]{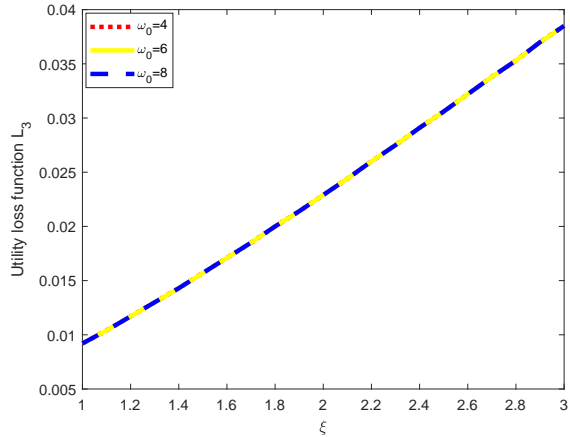}
        \end{minipage}
    }
    \caption{The impacts of $w_0$ and $\xi$ on the utility loss function $L_3$.}
    \label{fig14}
\end{figure}

We keep $\mu=0.10$ and carry out numerical experiments to show the difference between the robust equilibrium investment strategy for the ambiguity-averse investor under the mean-variance-skewness criterion with the skewness preference parameter $\phi_{0}=0.50$ and the equilibrium investment strategy for the ambiguity-neutral investor under the mean-variance criterion as well as the corresponding utility loss function $L_3$. From Figures \ref{fig13} and \ref{fig14}, one can find that the difference is always positive, so is $L_3$. Moreover, the difference decreases w.r.t. $\xi$ and increases w.r.t. $w_{0}$, while $L_3$ increases w.r.t. $\xi$ and is independent of $w_{0}$. Similar to the previous explanation, this can be understood as that the skewness preference dominates the ambiguity aversion level under our settings, which leads to that with the effect of both the skewness preference and ambiguity aversion, the investor invests more than the one who maintains a ambiguity-neutral attitude under the mean-variance criterion. Specially, when the ambiguity averse parameter increases, the degree of dominance is weakened, which corresponds to the case where the difference decreases w.r.t. $\xi$. Since the investment strategy is proportional to the initial wealth, the difference is more obvious when the initial wealth is larger. Note that $\delta_{3}=\gamma_{0}\bar{a}_{2}$ is positive for each case in Figures \ref{fig13} and \ref{fig14} and the condition in Remark 4.3 is satisfied.

\section{Conclusions}
This paper investigates a robust portfolio selection under the mean-variance-skewness criterion with wealth-dependent risk aversion and wealth-dependent skewness preference. Under a game theoretic framework, we characterize the robust optimization problem with alternative models and give the characterizations of the robust equilibrium investment strategy and corresponding equilibrium value function. We then derive an EDPE for the equilibrium value function and an extended HJBI system as well as an verification theorem to determine the robust equilibrium investment strategy and corresponding equilibrium value function for the robust time-consistent mean-variance-skewness portfolio selection problem. Moreover, we obtain the robust equilibrium investment strategy and corresponding equilibrium value function in semi-closed forms for a special robust portfolio selection problem which conclude some known results in the literature. Finally, numerical experiments indicate the following findings: (i) The investor with more initial wealth would invest more in the risky asset and should pay more attention to the ambiguity; (ii) The greater the risk aversion coefficient of the investor, the less she/he invests in the risky asset, and the risk aversion coefficient has a greater impact on the investor's decisions when the expected return rate of the risky asset is larger; (iii) In most cases, the mean-variance-skewness investor would invest more in the risky asset than the mean-variance investor and the greater the skewness preference parameter, the greater both the investment in risky asset and the corresponding utility loss function; (iv) When the volatility of the risky asset is small, the mean-variance-skewness investor would invest less in the risky asset than the mean-variance investor and the greater the skewness preference parameter, the less the investment in risky asset;  (v) The skewness preference has no impact on the robust equilibrium investment strategy when the risk aversion coefficient is large enough; (vi) The investor with higher degree of ambiguity aversion would reduce the amount invested in risky asset and the corresponding utility loss function becomes larger; (vii) The skewness preference could slow down the reduction of investment in the risky asset due to the ambiguity aversion of an investor and the effect of slowing down depends on the values of parameters of the skewness preference and the ambiguity aversion.

Recall that there exists a state variable $Z$ in the financial market, which can be used to describe the dynamics of stochastic volatility, stochastic interest rate, regime switching and so on \cite{Yan2020, Zhu2020}. Therefore, it would be interesting to consider some specific cases with additional state variable $Z$ and attempt to derive the (semi-)closed form solutions. Moreover, the price processes of risky assets may not be continuous in practice, but may jump in response to news or other surprise events \cite{Zeng2016, zhang2013fast}. Thus, we can incorporate jumps to the equations $\eqref{eq2}$ and $\eqref{eq3}$ to enrich theoretical and practical research. We leave these work for future study.

\section*{Acknowledgements}
The authors are grateful to the editors and reviewers whose helpful comments and suggestions have led to much improvement of the paper.

\section*{Appendices}
\appendix
\renewcommand{\appendixname}{Appendix~\Alph{section}}
\section{Proof of Lemma \ref{lemma2.1}}
\begin{proof}
By the definition of $J$ in the equation $\eqref{eq10}$ and the above notations in this lemma, we can obtain
$$
J(t,x;\bm{\pi},\textbf{q})=h^{\bm{\pi},\textbf{q}}(t,x,t,x)+G(t,x,g^{\bm{\pi},\textbf{q}}(t,x))+K(t,x,g^{\bm{\pi},\textbf{q}}(t,x),k^{\bm{\pi},\textbf{q}}(t,x)).
$$
Since $g^{\bm{\pi},\textbf{q}}(t,x)=\mathbb{E}^{\mathbb{Q}}_{t,x}\left[g^{\bm{\pi},\textbf{q}}(s,X_{s})\right]$ and $k^{\bm{\pi},\textbf{q}}(t,x)=\mathbb{E}^{\mathbb{Q}}_{t,x}\left[k^{\bm{\pi},\textbf{q}}(s,X_{s})\right]$, we have
\begin{align*}
 K(t,x,g^{\bm{\pi},\textbf{q}}(t,x),k^{\bm{\pi},\textbf{q}}(t,x))
 =\mathbb{E}^{\mathbb{Q}}_{t,x}\left[K(s,X_{s},g^{\bm{\pi},\textbf{q}}(s,X_{s}),k^{\bm{\pi},\textbf{q}}(s,X_{s}))\right]+L^{\bm{\pi},\textbf{q}}_{K}(t,x,s).
\end{align*}
By the proof of Lemma 8 in \cite{Pun2}, one has
$$
G(t,x,g^{\bm{\pi},\textbf{q}}(t,x))=\mathbb{E}^{\mathbb{Q}}_{t,x}\left[G(s,X_{s},g^{\bm{\pi},\textbf{q}}(s,X_{s}))\right]+L^{\bm{\pi},\textbf{q}}_{G}(t,x,s)
$$
and
$$
h^{\bm{\pi},\textbf{q}}(t,x,t,x)=\mathbb{E}^{\mathbb{Q}}_{t,x}\left[h^{\bm{\pi},\textbf{q}}(s,X_{s},s,X_{s})+\int^{s}_{t}C(\tau,X_{\tau},\tau,X_{\tau};\textbf{q})\mathrm{d}\tau\right]
+L^{\bm{\pi},\textbf{q}}_{C}(t,x,s)+L^{\bm{\pi},\textbf{q}}_{H}(t,x,s).
$$
Then the result follows.
\end{proof}

\section{Proof of Theorem \ref{thm3.1}}
\begin{proof} Step 1. We prove that $V(t,x)=J(t,x;\bm{\pi}^{\ast},\bm{q}^{\ast})$. In light of the Feynman-Kac theorem \cite{Yong1999}, one has $g(t,x)=\mathbb{E}^{\mathbb{Q^{*}}}_{t,x}\left[W_{T}\right]$, $k(t,x)=\mathbb{E}^{\mathbb{Q^{*}}}_{t,x}\left[(W_{T})^{2}\right]$ and
$$
h(t,x,s,y)=\mathbb{E}^{\mathbb{Q^{*}}}_{t,x}\left[\int^{T}_{t}C(s,y,\tau,X_{\tau};\textbf{q}^{*})\mathrm{d}\tau+F(s,y,W_{T})\right].
$$
By $(\bm{\pi}^{*}, \textbf{q}^{*})$, the HJBI equation $\eqref{eq16}$ becomes a PDE
$$
\mathcal{A^{\bm{\pi}^{*},\textbf{q}^{*}}}V(t,x)+C(t,x,t,x;\textbf{q}^{*})+\mathcal{L}(t,x,\bm{\pi}^{*},\textbf{q}^{*},h,g,k)=0.
$$
It follows from equations $\eqref{eq17}$-$\eqref{eq19}$ that
\begin{align*}
\mathcal{L}(t,x,\bm{\pi}^{*},\textbf{q}^{*},h,g,k)=&-C(t,x,t,x;\textbf{q}^{*})-\mathcal{A^{\bm{\pi}^{*},\textbf{q}^{*}}}h(t,x,t,x)\\
&-\mathcal{A^{\bm{\pi}^{*},\textbf{q}^{*}}}
G(t,x,g(t,x))-\mathcal{A^{\bm{\pi}^{*},\textbf{q}^{*}}}K(t,x,g(t,x),k(t,x)).
\end{align*}
By using the Dynkin's Theorem \cite{Oksendal2013}, we can derive
\begin{align*}
 \mathbb{E}^{\mathbb{Q^{*}}}_{t,x}\left[V(T,X_{T})\right]=&V(t,x)+ \mathbb{E}^{\mathbb{Q^{*}}}_{t,x}\left[\int^{T}_{t}\mathcal{A^{\bm{\pi}^{*},\textbf{q}^{*}}}V(\tau,X_{\tau})\mathrm{d}\tau\right]\\
 =&V(t,x)+ \mathbb{E}^{\mathbb{Q^{*}}}_{t,x}\left[\int^{T}_{t}\mathcal{A^{\bm{\pi}^{*},\textbf{q}^{*}}}h(\tau,X_{\tau},\tau,X_{\tau})\mathrm{d}\tau\right]\\
 &+\mathbb{E}^{\mathbb{Q^{*}}}_{t,x}\left[\int^{T}_{t}\left(\mathcal{A^{\bm{\pi}^{*},\textbf{q}^{*}}}
G(\tau,X_{\tau},g(\tau,X_{\tau}))+\mathcal{A^{\bm{\pi}^{*},\textbf{q}^{*}}}K(\tau,X_{\tau},g(\tau,X_{\tau}),k(\tau,X_{\tau}))\right)\mathrm{d}\tau\right]\\
=&V(t,x)+ \mathbb{E}^{\mathbb{Q^{*}}}_{t,x}\left[h(T,X_{T},T,X_{T})-h(t,x,t,x)-G\left(t,x,\mathbb{E}^{\mathbb{Q^{*}}}_{t,x}\left[W_{T}\right]\right)\right]\\
&+\mathbb{E}^{\mathbb{Q^{*}}}_{t,x}\left[G(T,X_{T},W_{T})+
K\left(T,X_{T},W_{T},(W_{T})^{2}\right)-K\left(t,x,\mathbb{E}^{\mathbb{Q^{*}}}_{t,x}\left[W_{T}\right],\mathbb{E}^{\mathbb{Q^{*}}}_{t,x}\left[(W_{T})^{2}\right]\right)\right].
\end{align*}
Thus, we have
\begin{equation*}
  V(t,x)=h(t,x,t,x)+G\left(t,x,\mathbb{E}^{\mathbb{Q^{*}}}_{t,x}\left[W_{T}\right]\right)
  +K\left(t,x,\mathbb{E}^{\mathbb{Q^{*}}}_{t,x}\left[W_{T}\right],\mathbb{E}^{\mathbb{Q^{*}}}_{t,x}\left[(W_{T})^{2}\right]\right)=J(t,x;\bm{\pi}^{\ast},\bm{q}^{\ast}).
\end{equation*}

Step 2. Deploying the similar procedure of the proof of Theorem 10 in \cite{Pun2}, we can show that $(\bm{\pi}^{*}, \textbf{q}^{*})$ satisfies two inequalities in Definition \ref{definition3.2}, which implies that $(\bm{\pi}^{*}, \textbf{q}^{*})$ is the equilibrium control-measure strategy and $\bm{\pi}^{*}$ is the robust equilibrium investment strategy.
\end{proof}

\section{Proof of Lemma \ref{lemma4.1}}
\begin{proof}
In view of the condition $(ii)$ in Definition \ref{definition4.1}, we have $\mathbb{E}^{\mathbb{P}}\left[\left(\int^{T}_{0}\|q^{S}_{s}\|^{2}\mathrm{d}s\right)^{p}\right]<\infty$ for any $p\geq1$. Moreover, by using the H\"{o}lder's inequality, one has
\begin{align*}
\mathbb{E}^{\mathbb{Q}}\left[\left(\int^{T}_{0}\|q^{S}_{s}\|^{2}\mathrm{d}s\right)^{p}\right]
\leq\left(\mathbb{E}^{\mathbb{P}}\left[\left(\int^{T}_{0}\|q^{S}_{s}\|^{2}\mathrm{d}s\right)^{2p}\right]\right)^{\frac{1}{2}}
\times\left(\mathbb{E}^{\mathbb{P}}\left[\exp\left(6\int^{T}_{0}\|q^{S}_{s}\|^{2}\mathrm{d}s\right)\right]\right)^{\frac{1}{4}}<\infty,
\end{align*}
where the penultimate equality holds because one can define a new measure with the Radon-Nikodym derivative $\mathcal{\varepsilon}_{T}\left(4q^{S}\cdot W^{S}\right)$
by the Novikov's condition $\eqref{eq7}$.

By the conditions $(iii)$ and $(iv)$ in Definition \ref{definition4.1}, Doob's martingale maximal inequality and the H\"{o}lder's inequality, we have
\begin{align*}
&\mathbb{E}^{\mathbb{Q}}\left[\sup\limits_{t\in[0,T]}\mid W_{t}\mid^{4}\right]\\
\leq&L+L\mathbb{E}^{\mathbb{Q}}\left[\left(\int^{T}_{0}\left(\| u_{s}'\beta_{s}\|+\| u_{s}'\sigma_{s}q^{S}_{s}\|\right)\mathrm{d}s\right)^{4}\right]
+L\mathbb{E}^{\mathbb{Q}}\left[\left(\sup\limits_{t\in[0,T]}\int^{t}_{0}e^{\int^{T}_{s}r_{v}\mathrm{d}v}u_{s}'\sigma_{s}\mathrm{d}W^{S,\mathbb{Q}}_{s}\right)^{4}\right]\\
\leq&L+L\mathbb{E}^{\mathbb{Q}}\left[\left(\int^{T}_{0}\| u_{s}\|^{2}\mathrm{d}s\right)^{4}\right]+L\mathbb{E}^{\mathbb{Q}}\left[\left(\int^{T}_{0}\| q^{S}_{s}\|^{2}\mathrm{d}s\right)^{4}\right]+L\left(\mathbb{E}^{\mathbb{Q}}\left[\left(\int^{T}_{0}\| u_{s}\|^{2}\mathrm{d}s\right)^{4}\right]\right)^{\frac{1}{2}}<\infty,
\end{align*}
where we use the property that $r_{t}$, $\mu_{t}$ and $\sigma_{t}$ are bounded for $t\in[0,T]$ and $L$ is a positive constant that may change from line to line.
\end{proof}

\section{Proof of Theorem \ref{thm4.1}}
\begin{proof}
Step 1.  We show that $f_{7}(t)$ is nonzero for any $t\in[0,T]$, namely, the integral equation $\eqref{eq32}$ is well defined, where
 $$f_{7}(t)=\gamma_{0}e^{\int^{T}_{t}\left[2r_{s}+\frac{2\Theta_{s}f(s)}{(\xi+1)^{2}}+\frac{\Theta_{s}f^{2}(s)}{(\xi+1)^{2}}\right]\mathrm{d}s}+2\phi_{0}
   e^{\int^{T}_{t}\left[3r_{s}+\frac{3\Theta_{s}f(s)}{(\xi+1)^{2}}+\frac{\Theta_{s}f^{2}(s)}{(\xi+1)^{2}}\right]\mathrm{d}s}\left(
   1-e^{\int^{T}_{t}\frac{2\Theta_{s}f^{2}(s)}{(\xi+1)^{2}}\mathrm{d}s}\right).$$
Otherwise, if
\begin{equation}\label{eq33}
  \gamma_{0}e^{\int^{T}_{t_{0}}\left[2r_{s}+\frac{2\Theta_{s}f(s)}{(\xi+1)^{2}}+\frac{\Theta_{s}f^{2}(s)}{(\xi+1)^{2}}\right]\mathrm{d}s}+2\phi_{0}
   e^{\int^{T}_{t_{0}}\left[3r_{s}+\frac{3\Theta_{s}f(s)}{(\xi+1)^{2}}+\frac{\Theta_{s}f^{2}(s)}{(\xi+1)^{2}}\right]\mathrm{d}s}\left(
   1-e^{\int^{T}_{t_{0}}\frac{2\Theta_{s}f^{2}(s)}{(\xi+1)^{2}}\mathrm{d}s}\right)=0
\end{equation}
for some $t_{0}\in[0,T]$, then
$$
e^{\int^{T}_{t_{0}}\left[r_{s}+\frac{\Theta_{s}f(s)}{(\xi+1)^{2}}\right]\mathrm{d}s}\left(
   e^{\int^{T}_{t_{0}}\frac{2\Theta_{s}f^{2}(s)}{(\xi+1)^{2}}\mathrm{d}s}-1\right)=\frac{\gamma_{0}}{2\phi_{0}},
$$
which implies that $f(t)$ is finite over $[t_{0},T]$. However, this contradicts with the fact that the left side of the integral equation $\eqref{eq32}$ is infinite by $\eqref{eq33}$. Thus, $f_{7}(t)$ is nonzero for any $t\in[0,T]$ and the integral equation $\eqref{eq32}$ is well defined.

Step 2. We prove that $f(t)$ must be bounded over $[0,T]$. Suppose that $f(t)$ is positive and large enough such that $e^{\int^{T}_{t}3\left[r_{s}+\frac{\Theta_{s}f(s)}{(\xi+1)^{2}}+\frac{\Theta_{s}f^{2}(s)}{(\xi+1)^{2}}\right]\mathrm{d}s}$ dominates $e^{\int^{T}_{t}\left[r_{s}+\frac{\Theta_{s}f(s)}{(\xi+1)^{2}}\right]\mathrm{d}s}$ and $e^{\int^{T}_{t}\frac{\Theta_{s}f^{2}(s)}{(\xi+1)^{2}}\mathrm{d}s}$ for $t\in[0,T]$. We can show that the ratio of the sum of the first term and the second term in curly brackets on the right side of $\eqref{eq32}$ to $f_{7}(t)$ and the ratio of the fourth item in curly brackets on the right side of $\eqref{eq32}$ to $f_{7}(t)$ are both close to 0, and the ratio of the third term in curly brackets on the right side of $\eqref{eq32}$ to $f_{7}(t)$ closes to $-\frac{1}{2}$. Therefore, the right side of $\eqref{eq32}$ closes to $-\frac{1}{2}$. This is impossible by the left side of the equation $\eqref{eq32}$. Thus, $f(t)$ has an upper bound over $[0,T]$. Similarly, we can prove that $f(t)$ has a lower bound over $[0,T]$. On the whole, there exist constants $L_{0}$ and $L_{1}$ satisfying $-\infty<L_{0}\leq f(t)\leq L_{1}<\infty$ for $t\in[0,T]$.

Step 3. It follows from the equation $\eqref{eq31}$ and the step 2 that the functions $h_{1}(t)$, $h_{2}(t)$, $h_{3}(t)$ and $g_{1}(t)$ are also bounded for $t\in[0,T]$. Since $f_{7}(t)=\gamma_{0}h_{2}(t)+2\phi_{0}(g_{1}(t)h_{2}(t)-h_{3}(t))$ is nonzero for $t$ over $[0, T]$ by step 1, we denote
$$
f_{1}(t)=\frac{h_{1}(t)}{f_{7}(t)},\;f_{2}(t)=\frac{g_{1}^{2}(t)}{f_{7}(t)},\;f_{3}(t)=\frac{h_{2}(t)}{f_{7}(t)},\;f_{4}(t)=\frac{h_{3}(t)}{f_{7}(t)},\;
f_{5}(t)=\frac{g_{1}^{3}(t)}{f_{7}(t)},\;f_{6}(t)=\frac{g_{1}(t)h_{2}(t)}{f_{7}(t)}.
$$
Then one has
$$f(t)=f_{1}(t)+\gamma_{0}\left(f_{2}(t)-f_{3}(t)\right)+\phi_{0}\left(f_{4}(t)+2f_{5}(t)-3f_{6}(t)\right).$$
Moreover, by the equation $\eqref{eq30}$, we can deduce $\frac{\partial f_{i}(t)}{\partial t}=P_{i}(t,f_{1},f_{2},f_{3},f_{4},f_{5},f_{6})$, where $P_{i}\;(i=1,\cdots,6)$ is a quartic or cubic polynomial function of $(f_{1},f_{2},f_{3},f_{4},f_{5},f_{6})$ with bounded Lipschitz continuous coefficient functions of $t$. In light of the uniqueness theorem \cite{Zwillinger1998}, we derive that the system of equations $\frac{\partial f_{i}(t)}{\partial t}=P_{i}(t,f_{1},f_{2},f_{3},f_{4},f_{5},f_{6})$ admits a unique solution. Therefore, the integral equation $\eqref{eq32}$ admits a unique solution.
\end{proof}

\section{Proof of Theorem \ref{thm4.2}}
\begin{proof}
We verify that the equilibrium control-measure strategy $(u_{t}^{\ast},{q^{S}_{t}}^{\ast})\in\Pi_{1} \times \bm{Q}_{1}$. Since $r_{t}$, $\beta_{t}$, $\sigma_{t}$, $\Theta_{t}$  and $f(t)$ are deterministic and bounded functions in $t$ over $[0,T]$ and $\xi$ is a positive constant, it follows from $\eqref{eq31}$ that $h_{1}(t)$, $h_{2}(t)$, $h_{3}(t)$ and $g_{1}(t)$ are also bounded for $t\in[0,T]$. Then, it is easy to see that the conditions $(i)$ and $(ii)$ in Definition \ref{definition4.1} hold. With the equilibrium control-measure strategy $\eqref{eq36}$, the corresponding wealth process can be rewritten as
$$
\mathrm{d}W_{t}^{\ast}=\left[r_{t}W_{t}^{\ast}+\frac{\Theta'_{t}f(t)}{(\xi+1)^{2}}W_{t}^{\ast}\right]\mathrm{d}t
+\frac{\beta'_{t}(\Sigma^{-1}_{t})'}{\xi+1}\sigma_{t}f(t)W_{t}^{\ast}\mathrm{d}W^{S,\mathbb{Q}^{\ast}}_{t}
$$
and so
$$
W_{t}^{\ast}=w_{0}\exp\left[\int^{t}_{0}\left(r_{s}+\frac{\Theta'_{s}f(s)}{(\xi+1)^{2}}-\frac{\Theta'_{s}f^{2}(s)}{2(\xi+1)^{2}}\right)\mathrm{d}s
+\int^{t}_{0}\frac{\beta'_{s}(\Sigma^{-1}_{s})'}{\xi+1}\sigma_{s}f(s)\mathrm{d}W^{S,\mathbb{Q}^{\ast}}_{s}\right].
$$
Thus, the condition $(iii)$ in Definition \ref{definition4.1} is fulfilled. In addition, given a positive initial wealth $w_{0}$, the wealth process remains positive. Then, the conditions $W_{t}=w\neq0$, $\gamma(W_{t})>0$, $\frac{\mathrm{d} \gamma(W_{t})}{\mathrm{d}W_{t}}<0$, $\phi(W_{t})>0$ and $\frac{\mathrm{d} \phi(W_{t})}{\mathrm{d}W_{t}}<0$ hold. With the equilibrium control-measure strategy $\eqref{eq36}$ and the corresponding measure $\mathbb{Q}^{*}$, we have
\begin{align*}
\mathbb{E}^{\mathbb{Q}^{*}}\left[\left(\int^{T}_{0}\|u^{*}_{s}\|^{2}\mathrm{d}s\right)^{4}\right]=&
\mathbb{E}^{\mathbb{Q}^{*}}\left[\left(\int^{T}_{0}\|\frac{W_{s}}{\xi+1}\Sigma_{s}^{-1}\beta_{s}f(s)\|^{2}\mathrm{d}s\right)^{4}\right]
\leq L\mathbb{E}^{\mathbb{Q}^{*}}\left[\left(\int^{T}_{0}\|W_{s}\|^{2}\mathrm{d}s\right)^{4}\right]
\end{align*}
and
\begin{align*}
&\mathbb{E}^{\mathbb{Q}^{*}}\left[\left(\int^{T}_{0}\|W_{s}\|^{2}\mathrm{d}s\right)^{4}\right]\\
\leq&L\left(\int^{T}_{0}w_{0}^{4}\exp\left[\int^{s}_{0}\left(4r_{u}+\frac{4\Theta'_{u}f(u)}{(\xi+1)^{2}}\right)\mathrm{d}u
\right]\mathrm{d}s\right)^{2}
\times\int^{T}_{0}\exp\left[\int^{s}_{0}\frac{6\Theta'_{u}f^{2}(u)}{(\xi+1)^{2}}\mathrm{d}u\right]\mathrm{d}s<\infty,
\end{align*}
where $L$ is a positive constant that may differ from line to line. Thus, $\mathbb{E}^{\mathbb{Q}^{*}}\left[\left(\int^{T}_{0}\|u^{*}_{s}\|^{2}\mathrm{d}s\right)^{4}\right]<\infty$. Namely, the condition $(iv)$ in Definition \ref{definition4.1} is verified. Since the proof of Lemma \ref{lemma4.1} is independent of the condition $(v)$ in Definition \ref{definition4.1} and
$\Phi^{-1}(t,w)=-\delta_{1}^{2}/\delta_{2}=f^{2}(t)\left(\gamma_{0}h_{2}(t)+2\phi_{0}(g_{1}(t)h_{2}(t)-h_{3}(t))\right)w,$
we can derive $\mathbb{E}^{\mathbb{Q}^{*}}\left[\sup\limits_{t\in[0,T]}\Phi^{-2}(t,W_{t}) \right]\leq L\mathbb{E}^{\mathbb{Q}^{*}}\left[\sup\limits_{t\in[0,T]} W^{2}_{t}\right]<\infty$ by Lemma \ref{lemma4.1}. Furthermore, the smoothness condition of these functions in Definition \ref{definition3.3} is satisfied under our setting.
\end{proof}

\end{document}